\documentclass[11pt,twoside,a4paper,reqno]{amsart}
\usepackage[utf8]{inputenc}
\usepackage{amsmath}
\usepackage{amssymb}
\usepackage{amsthm}
\usepackage[english]{babel}
\usepackage{tikz-cd}
\usepackage{url}
\usepackage{parskip}
\usepackage{enumerate}
\usepackage[normalem]{ulem}
\usepackage{hyperref}
\newtheorem{theorem}{Theorem}

\newtheorem{definition}[theorem]{Definition}
\newtheorem{example}[theorem]{Example}
\newtheorem{lemma}[theorem]{Lemma}
\newtheorem{remark}[theorem]{Remark}
\newtheorem{proposition}[theorem]{Proposition}

\newcommand{\as}[1]{\text{as } #1 \rightarrow\infty}
\renewcommand{\d}{\, \text{d}}

\newcommand{\f}{\varphi}
\newcommand{\fceil}[2]{\left \lceil \frac{#1}{#2}  \right \rceil}

\newcommand{\R}{\mathbb{R}}

\newcommand{\C}{\mathcal{C}}

\newcommand{\N}{\mathbb{N}}
\renewcommand{\P}{\mathbb{P}}
\newcommand{\E}{\mathbb{E}}
\newcommand{\B}{\mathcal{B}}
\newcommand{\D}{\mathcal{D}}
\renewcommand{\S}{\mathcal{S}}
\newcommand{\G}{\mathcal{G}}
\newcommand{\vol}{\mathrm{vol}}
\newcommand{\xRightarrow}[1]{\overset{#1}{\Rightarrow}}
\newcommand{\one}{\mathbf{1}}
\renewcommand{\to}{\rightarrow}

\renewcommand{\binom}[2]{{{#1} \choose {#2}}}

\DeclareMathOperator*{\esssup}{ess\,sup}
\DeclareMathOperator*{\essinf}{ess\,inf}
\usepackage{makecell}
\usepackage[margin=1in]{geometry}
\usepackage{todonotes}
\usepackage{soul}

\begin{document}
\title[Quenched limit theorems for random dynamical systems]{Quenched limit theorems via mixing of all orders for random dynamical systems}

\author{M. Auer}
\address{Department of Mathematics \\ University of Queensland, Brisbane, Australia}
\email{\href{mailto:m.auer@uq.edu.au}{m.auer@uq.edu.au}}
\thanks{The author would like to thank Cecilia Gonz\'{a}lez-Tokman for helpful discussions and the University of Queensland for their hospitality during this work. This work was supported by the Australian Research Council (FT240100592).}
\keywords{Central limit theorem, Poisson limit theorem, random dynamical systems, mixing of all orders, quenched limit theorems.}
\subjclass{Primary: 37H12, 37A25, 60F05 Secondary: 37A05, 37A50}
\date{}
\maketitle
\begingroup
\leftskip4em
\rightskip\leftskip
\begin{small}
\paragraph*{Abstract}
We adapt the notion of mixing of all orders to random dynamical systems and use it to establish quenched limit theorems. Assuming quenched exponential mixing of all orders, we prove both the central limit theorem and the Poisson limit theorem. Compared to existing approaches, this framework offers two main advantages. First, it allows for substantially weaker assumptions on the random constants appearing in mixing estimates: in the exponential mixing regime, logarithmic integrability suffices for the central limit theorem, while in the Poisson case, no integrability condition is required. Second, it applies naturally to invertible systems, where standard methods based on $L^{\infty}$-type estimates or spectral decay are less well suited.
\end{small}
\par
\endgroup

\part{Results}

\section{Introduction}

Deterministic dynamical systems with sufficient chaotic properties often exhibit statistical behaviour analogous to that of random processes. In particular, under suitable mixing assumptions, one can establish probabilistic limit theorems such as the central limit theorem (CLT) or the Poisson limit theorem (PLT) for observables along orbits.

A natural extension of this framework is given by random dynamical systems, where the evolution is driven by a sequence of maps chosen according to an underlying stochastic process. Such models arise, for instance, when deterministic dynamics are subject to small random perturbations at each iteration. These random systems form the main object of study of this manuscript.

In the deterministic setting, the CLT for dynamical systems has been extensively studied. For systems exhibiting sufficiently strong mixing properties, Birkhoff sums of sufficiently regular observables satisfy a central limit theorem. More precisely, if $(T, M,\mu)$ is a measure-preserving dynamical system and $f$ is a sufficiently regular centred observable, then the partial sums $\sum_{n=0}^{N-1} f \circ T^n$, suitably normalised, converge in distribution to a Gaussian random variable. We refer to \cite{Gouezel04, melbournenicol05, DeCLT} and references therein for representative results.

A related class of results concerns limit laws for rare events, where suitably rescaled hitting or return times to shrinking targets converge to exponential statistics. This limiting behaviour is commonly referred to as the Poisson limit theorem (PLT). Equivalently, one may formulate it in terms of the convergence of the counting process of visits to shrinking targets to a Poisson process. In deterministic dynamical systems, such results are typically obtained under strong mixing assumptions together with an aperiodicity condition preventing periodic clustering. We refer to \cite{HSV99,HV09,FFT12} for foundational references.

Proofs of the CLT and PLT typically rely on a form of \emph{quantitative mixing}. More precisely, for some decay rate $\psi(n)\searrow 0$, a Banach space $\mathfrak{B}$ of sufficiently regular observables, and all $f,g\in \mathfrak{B}$, one assumes that
\begin{equation*}
\left|\int_M f\cdot g\circ T^n \d\mu - \int_M f \d\mu \int_M g \d\mu \right|
\leq \psi(n)\,\|f\|_{\mathfrak{B}} \|g\|_{\mathfrak{B}} \quad \forall n\geq 1.
\end{equation*}

In many situations, such estimates can be extended to higher-order correlations under additional structure: for all $m\geq 2$, all functions $f_0,\dots,f_{m-1}\in \mathfrak{B}$, and all integers $0=k_0\leq \cdots \leq k_{m-1}$, one has
\begin{equation*}
\left|
\int_M \prod_{j=0}^{m-1} f_j\circ T^{k_j}\d\mu
- \prod_{j=0}^{m-1} \int_M f_j\d\mu
\right|
\leq \psi\left(\min_{0\leq j\leq m-2}(k_{j+1}-k_j)\right)
\prod_{j=0}^{m-1}\|f_j\|_{\mathfrak{B}}.
\end{equation*}
We refer to this property as \emph{mixing of all orders}. It provides a convenient framework to establish limit theorems such as the CLT or PLT via moment (or cumulant) methods. While the combinatorial bookkeeping can be substantial, the underlying strategy is conceptually transparent. In \cite{bjorklundgordnik17}, a CLT is proved under exponential mixing of all orders, and in \cite{dfl22}, a PLT is obtained under similar assumptions.

Our goal is to formulate an appropriate notion of mixing of all orders for random dynamical systems $(T_{\omega})$, and to use it to establish both the CLT and the PLT in this setting. To the best of our knowledge, there is currently no widely accepted formulation of mixing of all orders for random dynamical systems. The present work proposes a natural definition adapted to this setting.

In the setting of random dynamical systems, limit theorems can be formulated in different ways depending on how the randomness is treated. In the \emph{annealed} (or averaged) setting, one studies the distribution of observables after averaging over both the phase space and the driving randomness. In contrast, in the \emph{quenched} setting, one fixes a typical realisation of the randomness and considers limit theorems with respect to the corresponding sample measures.

Quenched and annealed formulations are related but not directly comparable in general; in many contexts, the quenched version is considered more informative, as it captures individual realisations. The distinction is discussed, for instance, in \cite{ANV15, DDWequiv} in the context of mixing and the CLT.

In this paper, we focus exclusively on quenched limit theorems, where randomness is treated as a fixed environment and probabilistic statements hold for $\mathbb{P}$-almost every realisation.

The quenched CLT for random dynamical systems has been extensively studied in a variety of settings. Early results go back to Kifer \cite{Ki98}, who established the quenched CLT under a random version of $\phi$-mixing. Since then, stronger results, often in the form of almost sure invariance principles, have been obtained under similarly strong mixing assumptions; see, for instance, \cite{S23, ntv18}.

Such results are typically derived via spectral methods for random transfer operators. When sufficient decay properties are available, alternative approaches can be employed, including random analogues of the Guivarc'h--Nagaev method \cite{GHclt} and Gou\"{e}zel's approach \cite{Gasip}. These techniques have been successfully implemented in a range of works, including \cite{ALS09, DFGTV18a, DFGTV18b, DFGTV20,DH20, DH21, DHS23,Haf24,dolgopyat2024ratesconvergencecltasip, DRAGI_EVI__2026}.

The quenched PLT for random dynamical systems is less well developed than its deterministic counterpart, although a number of results are available in the literature. Existing works establish the quenched PLT under strong mixing assumptions; see, for instance, \cite{RSV14, AFV15,FFV17, FFMV20,CS22,AHV25}. An alternative approach, based on spectral methods, is used in \cite{AFGTV25}, where a spectral gap for the random transfer operator is assumed.

Compared to the existing literature, the use of mixing of all orders in the present random setting offers two advantages, which we now describe. We use the notation introduced in Section~\ref{setsec}.

\textbf{(A) Random constants governing mixing.}
A central difficulty in establishing statistical limit theorems such as the quenched CLT or quenched PLT under mixing or spectral assumptions is that the relevant rates typically depend on the realisation. This dependence is commonly expressed through a random multiplicative constant, such as the quantity $C(\omega)$ in \eqref{linfmixdefintro}.

In several works, this constant is assumed to be uniform in $\omega$; see, for example, \cite{ALS09,ntv18, AFGTV25}. Such an assumption, however, is rather restrictive.

Various approaches have been developed to relax this requirement. One direction is to assume integrability of $C(\omega)$ (or higher moments), as in \cite{CT26,Haf24,RSV14,CS22}. Alternative methods avoid such integrability assumptions by modifying the underlying framework, for instance by inducing on subsets of $\Omega$ where $C(\omega)$ is uniformly bounded \cite{Ki98, DH21}, or by using adapted norms \cite{DHS23}. In \cite{S23}, random Young towers are employed; in their setting, $C(\omega)$ is only known to have polynomial tails (see \cite{zhikun2015mixing}). Such approaches, however, typically require additional structural conditions that can be difficult to verify in concrete examples.

The framework based on mixing of all orders allows for a substantial weakening of these assumptions without requiring additional structural input. In particular, we show that under exponential mixing of all orders,
\begin{itemize}
\item logarithmic integrability of $C(\omega)$ suffices for the quenched CLT;
\item no integrability assumption on $C(\omega)$ is required for the quenched PLT.
\end{itemize}
These conclusions follow from Theorems \ref{cltthm} and \ref{pltthm}, together with the relation \eqref{n*crel} linking our formulation to the standard mixing estimates involving the random multiplicative constant $C(\omega)$.

\textbf{(B) Invertible systems.}
In several works establishing quenched CLT or PLT results directly via mixing assumptions, one typically employs asymmetric estimates of the form
\begin{equation}\label{linfmixdefintro}
\left|\int_M f \cdot g\circ T_{\omega}^n \d\mu_{\omega}- \int_M f \,\mathrm{d}\mu_{\omega} \int_M g \d\mu_{\sigma^n \omega} \right| \leq C(\omega) \psi(n) \|f\|_{\mathfrak{B}} \|g\|_{L^{\infty}},
\end{equation}
for all $f$ in a Banach space $\mathfrak{B}$, $g\in L^{\infty}$, and $n\geq 1$, where $C:\Omega\to(0,\infty)$ is measurable. Alternatively, stronger notions of mixing, such as $\alpha$-mixing and related notions \cite{bradmix}, can be imposed, as is e.g.\ the case in \cite{Ki98}. These assumptions are closely tied to approaches based on Stein's method, martingale approximations, or conditioning arguments on the fibre measures.

Such estimates are, however, generally incompatible with invertible systems. Indeed, if \eqref{linfmixdefintro} were to hold in the invertible setting, one could precompose $g$ with inverse iterates and obtain stronger bounds. In invertible systems, one instead expects decay of correlations only when both observables belong to suitable regularity classes (e.g.\ H\"{o}lder, $BV$, or Sobolev spaces).

Using spectral properties of the random transfer operator, several authors \cite{ALS09, DFGTV20,DH20} have treated invertible systems. In this setting, however, good spectral properties are typically available only on anisotropic Banach spaces, such as those introduced by Gou\"{e}zel and Liverani \cite{GL06}. Since their construction depends on stable and unstable directions, in the random setting, the corresponding spaces would naturally depend on $\omega$. Constructing such $\omega$-dependent Banach spaces with suitable spectral properties is, however, highly non-trivial, and existing works often circumvent this issue by imposing strong uniformity assumptions.

To summarise, standard approaches either rely on mixing properties that are incompatible with invertibility, or on spectral properties of transfer operators, which require delicate Banach space constructions. While the latter can accommodate invertible systems (e.g.\ via anisotropic spaces), this typically involves substantial technical effort and strong uniformity assumptions in the random setting.

By contrast, the framework of mixing of all orders avoids the use of Banach space techniques and spectral theory, and instead works directly at the level of mixing properties. At the same time, it avoids asymmetric mixing assumptions such as \eqref{linfmixdefintro} that would not be compatible with invertibility. This makes it well-suited to the study of invertible systems, in line with what is observed in the deterministic theory.

\textbf{Future directions.}
While the approach of mixing of all orders provides several advantages, it also involves certain trade-offs. For instance, in our setting, we require a stretched exponential rate of mixing, i.e.\ $\psi(n)=e^{-\gamma n^{\theta}}$ for some constants $\gamma,\theta>0$, whereas some existing techniques apply under slower rates. For example, in \cite{Ki98} the quenched CLT is obtained under polynomial decay, and in \cite{RSV14} superpolynomial decay suffices to establish the quenched PLT.

Some of the arguments developed here are not intrinsically tied to mixing of all orders and may be applicable in other settings. For instance, the ideas surrounding Lemma~\ref{badseqclem1}, where one excludes a sparse sequence in the proof of the PLT, allow for improved conditions on $C(\omega)$. Exploring such extensions appears to be a promising direction for future work.

The remainder of the paper is organised as follows. Section~\ref{setsec} introduces the notation, and Section~\ref{mixsec} defines quenched mixing of all orders. The quenched CLT and PLT are stated in Sections~\ref{cltsec} and~\ref{pltsec}, with proofs given in Sections~\ref{cltproofsec} and~\ref{pltproofsec}. Examples are presented in Section~\ref{explesec} and proved in Section~\ref{expleproofsec}. Finally, Section~\ref{condsec} establishes quenched mixing of all orders under condition~\eqref{linfmixdefintro} or under a spectral assumption on the random transfer operator.

\section{Setting and notation}\label{setsec}

Let $(\Omega, \P)$ be a probability space and let $\sigma: \Omega\rightarrow \Omega$ be an invertible map that preserves $\P$, and is ergodic. For $\kappa>0$ let $M$ be a $C^{\kappa}$ smooth manifold of dimension $\dim(M)=d$, and, for $\omega\in \Omega$, let $T_{\omega}:M\rightarrow M$ be transformations. We will consider $(T_{\omega})$ as a random dynamical system, with $\sigma$ being the driving system. 

Throughout, we will assume that the skew product $\S(x,\omega)=(T_{\omega}(x), \sigma \omega)$ preserves a probability measure $\nu$ that projects onto $\P$ in the second coordinate, and denote is decomposition by $\nu=\int \mu_{\omega} \d\P(\omega)$. Denote the stationary measure by
\begin{equation*}
\mu(A) = \int_{\Omega} \mu_{\omega}(A) \d\P \quad \forall A \subset M \text{ measurable}.
\end{equation*}

We will denote iterates along a certain trajectory of $\sigma$ by
\begin{equation*}
T_{\omega}^n := T_{\omega} \circ T_{\sigma \omega} \circ \cdots \circ T_{\sigma^{n-1} \omega} \quad n\geq 1, \omega \in \Omega,
\end{equation*}
and, for a measurable set $A\subset M$, its inverse images are defined by
\begin{equation}\label{inversedef}
T_{\sigma^n \omega}^{-n}(A) := (T_{\omega}^n)^{-1}(A)  \quad n\geq 1, \omega \in \Omega.
\end{equation}
If each $T_{\omega}$ is invertible, then $T_{\sigma^n \omega}^{-n}:M\rightarrow M$ is a well defined function, and it holds
\begin{equation*}
T_{\sigma^n \omega}^{-n} = (T_{\sigma^{n-1} \omega})^{-1} \circ (T_{\sigma^{n-2} \omega})^{-1} \circ \cdots \circ (T_{\omega})^{-1}  \quad n\geq 1, \omega \in \Omega.
\end{equation*}
In \eqref{inversedef}, we choose to use the subscript $\sigma^n \omega$ to remember that $\S^{-n}$ carries the fibre $M\times \{\sigma^n \omega\}$ to the fibre $M\times \{\omega\}$. For a function $f:M\rightarrow \R$ denote the ergodic sums by 
\begin{equation*}
S_{N,\omega}(f)=\sum_{n=1}^{N} f \circ T_{\omega}^n.
\end{equation*} 

For $\P$-a.e.\ $\omega$, the fibre measures are interrelated by the identity
\begin{equation}\label{fibrerel}
\mu_{\omega}(T_{\sigma\omega}^{-1} (A)) = \mu_{\sigma \omega}(A) \quad \forall A \subset M \text{ measurable}.
\end{equation}
Indeed, for a measurable set $B\subset\Omega$, it holds that $\S^{-1}(A\times B) = \bigcup_{\omega\in \sigma^{-1} (B)} T_{\sigma\omega}^{-1}(A) \times \{\omega\}$ and therefore
\begin{align*}
\int_{\sigma^{-1} (B)} \mu_{\sigma\omega}(A) \d\P &=\int_{B} \mu_{\omega}(A) \d (\P\circ \sigma^{-1}) \\
&=\int_{B} \mu_{\omega}(A) \d\P\\
&= \nu(A\times B)\\
&= \nu(\S^{-1}(A\times B))\\
& = \int_{\sigma^{-1} (B)} \mu_{\omega}(T_{\sigma\omega}^{-1}(A)) \d\P.
\end{align*}
If $\sigma$ is invertible, then $\sigma^{-1}(B)$ can be replaced by any measurable set $C\subset \Omega$, and the relation \eqref{fibrerel} follows.

\begin{remark}
Without invertibility of $\sigma$, the above argument yields only equality of the conditional expectations with respect to $\sigma^{-1}\mathcal{F}$, where $\mathcal{F}$ is the $\sigma$-algebra for $\Omega$, and not pointwise equality of the fibre measures.

However, if $\sigma$ is not invertible, we can always replace it with its invertible extension.
\end{remark}

Throughout this paper, we adhere to the following notational conventions:
\begin{itemize}
\item[$\bullet$] For real-valued functions $\alpha$ and $\beta$ defined on $\mathbb{N}$ or $\mathbb{R}$, we write $\beta = O(\alpha)$ if there exists a constant $C>0$ such that $|\beta(x)| \leq C |\alpha(x)|$ for all $x$, and $\beta = o(\alpha)$ if $|\beta(x)|/|\alpha(x)| \to 0$ as $x \to \infty$. If $\lim_{x\to\infty} \alpha(x)/\beta(x) = 1$, we write $\alpha \sim \beta$.

\item[$\bullet$] For $[0,\infty)$-valued functions $\alpha$ and $\beta$, we often write $\beta \ll \alpha$ instead of $\beta = O(\alpha)$, and $\beta \asymp \alpha$ if both $\beta \ll \alpha$ and $\alpha \ll \beta$.

\item[$\bullet$] For $x^* \in M$ and $r > 0$, let $B_r(x^*)$ denote the ball of radius $r$ centred at $x^*$. For simplicity, set $B_r(x^*) = \emptyset$ for $r \leq 0$. The notation $\mathbf{1}_A$ denotes the indicator function of a measurable set $A$.

\item[$\bullet$] For a function $\phi\in L^1(\P)$ we will denote its expectation by $\E(\phi):=\int_{\Omega} \phi \d\P$.

\item[$\bullet$] For a sequence of measurable functions $f_n:M\rightarrow\R$, a sequence of measures $\eta_n$, and a random variable $X$, we will write 
\begin{equation*}
f_n \xRightarrow{\eta_n} X \quad \as{n}
\end{equation*}
to denote convergence in distribution with respect to the measures $\eta_n$, or more explicitly
\begin{equation*}
\lim_{n\rightarrow\infty} \eta_n(f_n\leq t) = \P(X\leq t)
\end{equation*}
at all $t\in \R$ for which $t\rightarrow \P(X\leq t)$ is continuous.
\end{itemize}

Unless stated otherwise, all functions and sets that appear in the following are assumed to be measurable.

\section{Mixing or all orders}\label{mixsec}

Here we will discuss quenched mixing of all orders, and motivate the main Definition \ref{qmemdef}. For simplicity of the presentation below, we will confine ourselves now to \textit{quantitative mixing} only; it seems similar arguments can be made when the mixing is only qualitative.

Let $\mathfrak{B} \subset L^{\infty}(M)$ be a Banach space, and let $\psi(n)\searrow 0$. We say that $(T_{\omega})$ is \textit{quenched mixing in $\mathfrak{B}$ of speed $\psi$} if there is a function $C:\Omega \rightarrow(0,\infty)$ such that, for all $f,g\in \mathfrak{B}$ and $n\geq 1$, it holds that
\begin{equation}\label{mixingdef}
\left|\int_M f \cdot g\circ T_{\omega}^n \d\mu_{\omega}- \int_M f \d\mu_{\omega} \int_M g \d\mu_{\sigma^n \omega} \right| \leq C(\omega) \psi(n) \|f\|_{\mathfrak{B}} \|g\|_{\mathfrak{B}} \quad \text{for $\P$-a.e.\ $\omega$}.
\end{equation}
For some systems, we can even choose $C(\omega)=\text{constant}$, for example, this can happen when $(T_{\omega})$ consists of small perturbations of a hyperbolic map. In this case, we will say that the mixing is \textit{uniform}. Since having a uniform speed of mixing is quite a restrictive property, we will try to avoid it in the following.

As mentioned in the introduction, it is often helpful to not just be mixing but \textit{mixing of all orders}. Taking inspiration from deterministic dynamical systems, we might want to simply define it as follows; $(T_{\omega})$ is \textit{quenched mixing of order $m$ in $\mathfrak{B}$ of speed $\psi$} if there is a function $C:\Omega \rightarrow(0,\infty)$ such that, for $\P$-a.e.\ $\omega$, for all $f_0, \dots, f_{m-1} \in \mathfrak{B}$, and integers $0 = k_0 \leq \dots \leq k_{m-1}$, it holds that
\begin{equation}\label{qmemwrongdef}
\left| \int_M \prod_{j=0}^{m-1} f_j \circ T_{\omega}^{k_j} \d\mu_{\omega} - \prod_{j=0}^{m-1} \int_M f_j \d\mu_{\sigma^{k_j}\omega} \right| \leq C(\omega) \psi(\min_{0\leq j\leq m-2} (k_{j+1}-k_j)) \prod_{j=0}^{m-1} \|f_j\|_{\mathfrak{B}}.
\end{equation}

However, there is a problem with \eqref{qmemwrongdef}. Suppose it holds with $m=3$, and let $f,g\in \mathfrak{B}$, $n\geq 1$ and $l\geq n$. Using invariance, and then applying the above estimate with $f_0=1$, $f_1=f$, $f_2=g$, $k_1=l$, $k_2=l+n$, and $\omega'=\sigma^{-l}\omega$ yields
\begin{equation}\label{unifspeed}
\begin{aligned}
&\left|\int_M f \cdot g\circ T_{\sigma^l \omega'}^n \d\mu_{\sigma^l \omega'}- \int_M f \d\mu_{\sigma^l \omega'} \int_M g \d\mu_{\sigma^{n+l} \omega'} \right| \\
& = \left|\int_M f\circ T_{\omega'}^l \cdot g\circ T_{\omega'}^{n+l} \d\mu_{\omega'}- \int_M f \d\mu_{\sigma^l \omega'} \int_M g \d\mu_{\sigma^{n+l} \omega'} \right|\\
& \leq C(\omega') \psi(\min(n,l)) \|f\|_{\mathfrak{B}} \|g\|_{\mathfrak{B}}\\
& \leq C(\omega') \psi(n) \|f\|_{\mathfrak{B}} \|g\|_{\mathfrak{B}},
\end{aligned}
\end{equation}
for $\P$-a.e.\ $\omega'$. Now, for $\P$-a.e.\ $\omega$, we can apply the above for any $l\geq n$, therefore
\begin{equation*}
\left|\int_M f \cdot g\circ T_{\omega}^n \d\mu_{\omega}- \int_M f \d\mu_{\omega} \int_M g \d\mu_{\sigma^n \omega} \right| \leq \inf_{l\geq n} C(\sigma^{-l}\omega) \psi(n) \|f\|_{\mathfrak{B}} \|g\|_{\mathfrak{B}}.
\end{equation*}
Denote $\tilde{C}_n(\omega)=\inf_{l\geq n} C(\sigma^{-l}\omega)$. Then $\tilde{C}_n(\sigma\omega)\leq \tilde{C}_n(\omega)$, hence by ergodicity $\tilde{C}_n$ is almost surely constant and equal to $\essinf_{\omega' \in \Omega} C(\omega')$. In particular, it is independent of $n$. Therefore
\begin{equation*}
\left|\int_M f \cdot g\circ T_{\omega}^n \d\mu_{\omega}- \int_M f \d\mu_{\omega} \int_M g \d\mu_{\sigma^n \omega} \right| \leq \essinf_{\omega' \in \Omega} C(\omega') \psi(n) \|f\|_{\mathfrak{B}} \|g\|_{\mathfrak{B}} \quad \text{for $\P$-a.e.\ $\omega$}.
\end{equation*}
And we have deduced that the speed of mixing is in fact uniform, which is exactly what we were trying to avoid.

Therefore, we will try to formulate a weaker definition that
\begin{itemize}
\item is still strong enough, so it can be used to show CLT, PLT and other statistical limit theorems,
\item but does not imply uniform speed of mixing.
\end{itemize}

For inspiration, we consider the following. Often, when the system $(T_{\omega})$ is sufficiently expanding (see e.g.\ \cite{ANV15, CT26,ntv18}), then it holds that, in fact one of the norms in the mixing estimate \eqref{mixingdef} is the $L^{\infty}$-norm, and it holds
\begin{equation}\label{linfmix}
\left|\int_M f \cdot g\circ T_{\omega}^n \d\mu_{\omega}- \int_M f \d\mu_{\omega} \int_M g \d\mu_{\sigma^n \omega} \right| \leq C(\omega) \psi(n) \|f\|_{\mathfrak{B}} \|g\|_{L^{\infty}} \quad \text{for $\P$-a.e.\ $\omega$}.
\end{equation}
Now, for $f_0, \dots, f_{m-1} \in \mathfrak{B}$ integers $0 = k_0 \leq \dots \leq k_{m-1}$, we can apply \eqref{linfmix} with 
\begin{equation*}
f=f_0, g= \prod_{j=1}^{m-1} f_j \circ T_{\omega}^{k_j-k_1} \quad \text{and} \quad n=k_1
\end{equation*}
to obtain
\begin{equation}\label{linfiter}
\left| \int_M \prod_{j=0}^{m-1} f_j \circ T_{\omega}^{k_j} \d\mu_{\omega} - \int_M f_0 \d\mu_{\omega} \int_M \prod_{j=1}^{m-1} f_j \d\mu_{\sigma^{k_1}\omega} \right| \leq C(\omega) \psi(k_1) \|f_0\|_{\mathfrak{B}} \prod_{j=1}^{m-1} \|f_j\|_{L^{\infty}}.
\end{equation}
Assuming $\|\cdot\|_{\mathfrak{B}} \geq \|\cdot \|_{L^{\infty}}$, we can iterate this procedure (see Section \ref{condsec}) and obtain
\begin{equation}\label{linfmem}
\left| \int_M \prod_{j=0}^{m-1} f_j \circ T_{\omega}^{k_j} \d\mu_{\omega} - \prod_{j=0}^{m-1} \int_M f_j \d\mu_{\sigma^{k_j}\omega} \right| \leq \left(\sum_{j=0}^{m-1} C(\sigma^{k_j}\omega) \psi(k_{j+1}-k_j)\right) \prod_{j=0}^{m-1} \|f_j\|_{\mathfrak{B}}.
\end{equation}

To make \eqref{linfmem} more aesthetically pleasing, define 
\begin{equation*}
N^*_m(\omega') = (\sqrt{\psi})^{-1}\left(\frac{1}{m C(\omega')} \right),
\end{equation*}
where $(\sqrt{\psi})^{-1}(t)=\inf(s \geq 0 \;|\; \sqrt{\psi(s)}\leq t\})$. Then, if $\psi$ is continuous, it holds that
\begin{equation*}
\left| \int_M \prod_{j=0}^{m-1} f_j \circ T_{\omega}^{k_j} \d\mu_{\omega} - \prod_{j=0}^{m-1} \int_M f_j \d\mu_{\sigma^{k_j}\omega} \right| \leq \sqrt{\psi\left(\min_{0\leq j\leq m-2} (k_{j+1}-k_j)\right)} \prod_{j=0}^{m-1} \|f_j\|_{\mathfrak{B}}
\end{equation*}
if
\begin{equation}\label{asm:gaps}
k_{j+1} - k_j \geq N^*_m(\sigma^{k_j}\omega) \quad \forall j=0,..,m-1.
\end{equation}
It is straightforward to see how \eqref{asm:gaps} prevents the argument in \eqref{unifspeed}, and uniform speed of mixing does not follow.

For simplicity, and in view of the limit theorems proved below, we restrict to the case of stretched exponential mixing rates. More explicitly there are constants $\gamma>0$ and $\theta\in (0,1]$ such that 
\begin{equation*}
\psi(n) = e^{-\gamma n^{\theta}}.
\end{equation*}
In this case, it holds that 
\begin{equation}\label{n*crel}
N^*_m(\omega) = \left( \frac{2}{\gamma} \log(m C(\omega)) \right)^{\frac{1}{\theta}}.
\end{equation}

Motivated by this, we are ready to formulate our main definition.

\begin{definition}[Quenched stretched exponential mixing of all orders]\label{qmemdef}
Let $\mathfrak{B}$ be a Banach algebra. We say that $(T_{\omega})$ is quenched stretched exponentially mixing of order $m$ in $\mathfrak{B}$ if there exists a measurable function $N^*_m : \Omega \to \mathbb{N}$ and constants $\gamma_m>0$, $\theta_m \in (0,1]$ such that, for all $f_0, \dots, f_{m-1} \in \mathfrak{B}$ and integers $0 = k_0 \leq \dots \leq k_{m-1}$ satisfying
\begin{equation*}
k_{j+1} - k_j \geq N^*_m(\sigma^{k_j}\omega),
\end{equation*}
it holds that
\begin{equation}\label{asm:qmem}
\begin{aligned}
\left| \int_M \prod_{j=0}^{m-1} f_j \circ T_{\omega}^{k_j} \d\mu_{\omega} - \prod_{j=0}^{m-1} \int_M f_j \d\mu_{\sigma^{k_j}\omega} \right|
\leq e^{-\gamma_m \min_{0\leq j\leq m-2} (k_{j+1}-k_j)^{\theta_m}} \prod_{j=0}^{m-1} \|f_j\|_{\mathfrak{B}} \quad \text{for $\P$-a.e.\ $\omega$}.
\end{aligned}
\end{equation}
The function $N^*_m$ will be referred to as the random threshold.

For $m=2$, we simply say that $(T_{\omega})$ is quenched stretched exponentially mixing.

We say that $(T_{\omega})$ is quenched stretched exponentially mixing of all orders in $\mathfrak{B}$ if \eqref{asm:qmem} is satisfied for all $m\geq 2$, and $\gamma_m=\gamma$ and $\theta_m=\theta$ are independent of $m$.

If $\theta = 1$, we omit the word stretched and say that $(T_{\omega})$ is quenched exponentially mixing.
\end{definition}

To simplify notation, we will omit the dependence of $N^*_m$ on $m$ and simply write $N^*$. In the proofs below, the precise dependence on $m$ plays no role, since all estimates are carried out for fixed finite $m$. Moreover, in the examples considered in Section~\ref{explesec}, the dependence on $m$ only affects multiplicative constants, and therefore does not influence the integrability properties of $N^*_m$.

\begin{remark}
We highlight here the role of stretched exponential mixing.

In many natural examples, $(T_{\omega})$ is quenched exponentially mixing of all orders, but the associated threshold $N^*(\omega)$ fails to be integrable. A typical situation is that \eqref{linfmix} holds with $\psi(n)=e^{-\gamma n}$ and a random constant $C(\omega)$ which is only tempered, and not log-integrable. In this case (see also Lemma \ref{linfmemlem}), $(T_{\omega})$ is quenched exponentially mixing of all orders, but the corresponding threshold $N^*(\omega)$ is not integrable.

Nevertheless, it is sometimes possible to restore integrability of the threshold by weakening the rate of mixing and passing to a stretched exponential regime (see Remark \ref{renewalstretchrem}).

This illustrates a key distinction with the deterministic setting: there, exponential mixing (e.g.\ via a spectral gap) is typically sufficient, whereas in the random setting one must balance the rate of mixing with the integrability of the associated threshold. In particular, stretched exponential mixing is not merely a weaker substitute, but sometimes the natural regime in which limit theorems can be established.
\end{remark}

Additionally, we assume the following regularity property.

(R) There exists a constant $K>0$ such that for all $f \in \mathfrak{B}$ and $\mathbb{P}$-almost all $\omega$, one has
\begin{equation*}
f \circ T_{\omega} \in \mathfrak{B} \quad \text{and} \quad \|f \circ T_{\omega}\|_{\mathfrak{B}} \leq K \|f\|_{\mathfrak{B}}.
\end{equation*}

This assumption ensures that iterates remain exponentially controlled in the $\mathfrak{B}$-norm.

The above definition is formulated for a general Banach space $\mathfrak{B}$. In the smooth setting, one may take $\mathfrak{B} = C^{\kappa}$, while in non-uniform or discontinuous settings, one can work with spaces such as Sobolev spaces or spaces of bounded variation.

Even though the definition of quenched mixing of all orders depends on the choice of Banach space $\mathfrak{B}$, for simplicity we omit $\mathfrak{B}$ from the notation.

\section{The Central Limit Theorem (CLT)}\label{cltsec}

Here we consider collections of functions $\textbf{f} = (f_{\omega})_{\omega\in \Omega}$ such that each $f_{\omega}\in \mathfrak{B} \cap L^{\infty}$ is a bounded function in the Banach algebra, with uniform norm in $\omega$, and centred with respect to $\mu_{\omega}$. More precisely, we consider the collection
\begin{equation*}
\bar{\mathfrak{B}} := \left\{\mathbf{f} = (f_{\omega})_{\omega\in \Omega} \:\left|
\begin{aligned} 
&\forall\omega\in \Omega \text{ it holds that } f_{\omega}\in \mathfrak{B} \text{ and } \int_M f_{\omega} \d\mu_{\omega} = 0,\\
&\sup_{\omega \in \Omega} \|f_{\omega}\|_{\mathfrak{B}} + \|f_{\omega}\|_{L^{\infty}} <\infty
\end{aligned}
\right. \right\}.
\end{equation*}

\begin{theorem}\label{cltthm}
Suppose $(T_{\omega})$ is quenched stretched exponentially mixing of all orders, and (R) holds. Furthermore, assume that the random threshold $N^*_m\in L^1(\P)$ is integrable for each $m\geq 2$. Let $\mathbf{f}\in \bar{\mathfrak{B}}$ then there is a $\sigma_{\textbf{f}}\geq 0$ such that
\begin{equation}\label{cltcon}
\frac{S_{N, \omega}(\mathbf{f})}{\sqrt{N}} \xRightarrow{\mu_{\omega}} \mathcal{N}(0,\sigma_{\mathbf{f}}^2)  \quad \text{for $\P$-a.e.\ $\omega$},
\end{equation}
where $S_{N, \omega}(\mathbf{f})=\sum_{n=0}^{N-1} f_{\sigma^n\omega} \circ T_{\omega}^n$. Furthermore, for $\P$-a.e.\ $\omega$, it holds that
\begin{equation}\label{sigmadef}
\sigma_{\textbf{f}}^2=\lim_{N\rightarrow \infty} \frac{1}{N} \int_M (S_{N, \omega}(\mathbf{f}))^2 \d\mu_{\omega} = \E\left( \int_M f_{\omega}^2 \d\mu_{\omega} + 2\sum_{n=1}^{\infty} \int_M f_{\omega} \cdot f_{\sigma^n\omega} \circ T_{\omega}^n \d\mu_{\omega} \right).
\end{equation}
\end{theorem}

We refer to \eqref{cltcon} as the \emph{quenched central limit theorem} (quenched CLT) for $(T_{\omega})$.

It is important to understand when the variance $\sigma_{\mathbf{f}}^2$ vanishes.

\begin{definition}
A collection of functions $\mathbf{f} = (f_{\omega})_{\omega \in \Omega}$ is called a random coboundary in $L^2$ if there exists a family $(g_{\omega})_{\omega \in \Omega}$ such that
\begin{itemize}
\item $g_{\omega} \in L^2(\mu_{\omega})$ for $\P$-almost every $\omega$,
\item $\|g_{\omega}\|_{L^2(\mu_{\omega})} \in L^2(\P)$,
\item and, for $\mathbb{P}$-almost every $\omega$, it holds that
\begin{equation*}
f_{\omega} = g_{\omega} - g_{\sigma \omega} \circ T_{\omega},
\end{equation*}
where the equality is to be understood in $L^2(\mu_{\omega})$.
\end{itemize}
\end{definition}

In this case, for $N \ge 1$, the sum $S_{N, \omega}(\mathbf{f})$ telescopes to
\begin{equation}\label{telescope}
S_{N, \omega}(\mathbf{f}) = g_{\omega} - g_{\sigma^N \omega} \circ T_{\omega}^{N}.
\end{equation}

The following Proposition from \cite{DFGTV18a} gives a precise criterion of when the asymptotic variance vanishes. As our setting slightly varies from theirs, we shall include a proof for convenience.

\begin{proposition}\label{simgaprop}
Suppose that $(T_{\omega})$ is quenched stretched exponentially mixing in a dense Banach space $\mathfrak{B}\subset L^{\infty}$ and $N^* \in L^2(\mathbb{P})$. Let $\mathbf{f}\in \bar{\mathfrak{B}}$. Then $\sigma_{\mathbf{f}} = 0$ if and only if $\mathbf{f}$ is a random coboundary in $L^2$.
\end{proposition}

\section{The Poisson Limit Theorem (PLT)}\label{pltsec}

For the PLT, we will, at various points, approximate indicators of balls with functions in $\B$. To avoid the technicalities that come from approximations in general Banach spaces, we will assume that $\mathfrak{B}=C^{\kappa}$ and replace assumption (R) by

(R') It holds that $\|(T_{\omega})\|_{C^{\kappa}}:=\sup_{\omega\in \Omega} \|T_{\omega}\|_{C^{\kappa}} <\infty$.

For a set $A \subset M$ its \emph{random first hitting time} are given by
\begin{equation*}
\f_A(x,\omega):= \min(k\geq 1 \;|\; T_{\omega}^k x \in A) = \min(k\geq 1\;|\; \S^k(x,\omega) \in A \times \Omega).
\end{equation*}
Since $(\S,\nu)$ is ergodic (quenched mixing implies that $\S$ is mixing as well), it holds that $\nu(\f_A < \infty)=1$. When restricted to $A\times \Omega$, we call $\f_A|_{A\times \Omega}$ the \emph{random first return time}.
For $j\geq 1$ the \emph{random $j$th return time} is defined as $\f^{(1)}_A(x,\omega)=\f_A(x,\omega)$ and
\begin{equation*}
\f^{(j)}_A(x,\omega):= \min(k\geq 1 \;|\; T_{\omega}^{\f^{(j-1)}_A(x,\omega) + k} x \in A) = \min(k\geq 1\;|\; \S^{\f^{(j-1)}_A(x,\omega) + k}(x,\omega) \in A \times \Omega)
\end{equation*}
and the \emph{random sequence of return times} is
\begin{equation*}
\Phi_A(x,\omega) := (\f^{(1)}_A(x,\omega),\f^{(2)}_A(x,\omega),...)
\end{equation*}
Since $(\S,\nu)$ is ergodic, Kac's formula dictates $\int \f_A \d\nu|_{A\times \Omega} = \frac{1}{\mu(A)}$. Therefore one would hope for
\begin{equation*}
\mu(A) \Phi_A(\cdot, \omega) \xRightarrow{\mu_{\omega}|_A} \text{something} \quad \text{as } \mu(A)\to 0 \quad \text{for $\P$-a.e.\ $\omega$}.
\end{equation*}
The typical situation is when $A$ is a ball of shrinking radius centred at a point $x^* \in M$ and the limit is an i.i.d.\ sequence of standard exponentially distributed random variables.

\begin{remark}[Heuristic form of the Poisson limit theorem]
For $\nu$-a.e\ $(x^*,\omega)$, one expects that
\begin{equation}\label{poisdef}
\mu(B_r(x^*)) \Phi_{B_r(x^*)}(\cdot, \omega)
\xRightarrow{\mu_{\omega}|_{B_r(x^*)}} \Phi_{Exp}
\quad \text{as } r \to 0,
\end{equation}
where $\Phi_{Exp}$ is an i.i.d.\ sequence of standard exponential random variables.
\end{remark}

\begin{remark}
In addition to the return-time formulation \eqref{poisdef}, one can consider the hitting-time version
\begin{equation}\label{poishitdef}
\mu(B_r(x^*)) \Phi_{B_r(x^*)}(\cdot, \omega)
\xRightarrow{\mu_{\omega}} \Phi_{Exp}
\quad \text{as } r \to 0,
\end{equation}
where the first marginal now corresponds to the first hitting time.

In the deterministic setting, \eqref{poisdef} and \eqref{poishitdef} are equivalent. More generally, even when the limiting law is not exponential, the two limits are linked by an integral relation (see \cite[Main Theorem]{HLV05}). This relation can be exploited in proofs: if one shows that the limit of $\mu(B_r(x^*)) \Phi_{B_r(x^*)}$ is unchanged when restricting to $B_r(x^*)$, then the limit is necessarily exponential. This strategy is used, for instance, in \cite{Z22rare}.

In the random setting, however, no analogue of the relation in \cite[Main Theorem]{HLV05} is known in general. For this reason, in Theorem \ref{pltthm} we establish both \eqref{poisdef} and \eqref{poishitdef} separately.
\end{remark}

Note crucially that it is equivalent to show that the number of visits to the ball $B_r(x^*)$ converge, after suitable normalisation, to a standard Poisson process. Indeed, first observe that return times are related to the number of visits to $B_r(x^*)$ by
\begin{equation*}
\sum_{j=1}^J \f^{(j)}_{B_r(x^*)}(x,\omega) \leq N \iff S_{N,\omega}(\one_{B_r(x^*)}) \geq J,
\end{equation*}
where, by slight abuse of notation, we use the convention $S_{N,\omega}(f)=\sum_{n=1}^N f\circ T_{\omega}^n$ instead. Indeed, the left side says that we have to wait at most up to time $N$ to visit $B_r(x^*)$ $J$ times, while the right side says that we have visited $B_r(x^*)$ at least $J$ times up to time $N$. Normalising correctly and taking limits, it is straightforward to show that \eqref{poisdef} is equivalent to 
\begin{equation}\label{sumpoiscon}
\left(S_{\fceil{t}{\mu(B_r(x^*))},\omega} (1_{B_r(x^*)}) \right)_{t>0} \xRightarrow{\mu_{\omega}|_{B_r(x^*)}} \mathcal{P} \quad \text{as } r\to 0.
\end{equation}

In addition to the standing assumptions \eqref{asm:qmem} and (R'), we impose the following hypotheses, that
\begin{itemize}
\item guarantee that the measure of a ball is comparable for all $\omega$,
\item and rule out periodicity, which clearly would obstruct \eqref{poisdef}.
\end{itemize}

(VOL) For $\P$-almost all $\omega$, the fibre measures $\mu_{\omega}$ are absolutely continuous with respect to the volume measure $\vol$, and the density $\rho_{\omega}=\frac{\d \mu}{\d\vol}$ is continuous $\vol$-a.e. Furthermore it holds that 
\begin{equation}\label{asm:volcontrol}
\esssup_{(x,\omega)\in M\times \Omega} \rho_{\omega}(x) <\infty,
\end{equation}
where $\esssup$ is the essential supremum with respect to $\vol \times \P$.

(APER) The system is \emph{random aperiodic}, it holds that
\begin{equation}\label{asm:aper}
\nu((x,\omega) \;|\; \exists n\in\N \text{ such that } T_{\omega}^n x = x) = 0.
\end{equation} 

\begin{remark}
It seems that the reference measure $\vol$ in (VOL) can be replaced by a more general measure $\lambda$, provided $\lambda$ has positive dimension almost everywhere. This allows one to treat situations where the fibre measures are supported on fractal sets, similarly to condition (II') in \cite{RSV14}.

In this case, additional regularity assumptions on $\lambda$ are required, such as bounds on the measure of annuli as in condition (IV') of \cite{RSV14}. For simplicity of exposition, we restrict to the case $\lambda = \vol$.
\end{remark}

\begin{theorem}\label{pltthm}
Suppose $(T_{\omega})$ is quenched stretched exponentially mixing of all orders in $C^{\kappa}$, and (R'), (VOL), and (APER) hold. Then, for $\nu$-a.e\ $(x^*,\omega)\in M\times \Omega$, it holds that
\begin{equation}\label{hitconthmclaim}
\mu(B_r(x^*)) \Phi_{B_r(x^*)}(\cdot, \omega) \xRightarrow{\mu_{\omega}} \Phi_{Exp} \quad \text{as } r\to 0,
\end{equation}
and
\begin{equation}\label{retconthmclaim}
\mu(B_r(x^*)) \Phi_{B_r(x^*)}(\cdot, \omega) \xRightarrow{\mu_{\omega}|_{B_r(x^*)}} \Phi_{Exp} \quad \text{as } r\to 0,
\end{equation}
where $\Phi_{Exp}$ is an i.i.d.\ sequence of standard exponentially distributed random variables.
\end{theorem}

We refer to \eqref{hitconthmclaim}--\eqref{retconthmclaim} as the \emph{quenched Poisson limit theorem} (quenched PLT) for $(T_{\omega})$. 

\begin{remark}\label{memcondpoisrem}
The conclusion of Theorem~\ref{pltthm} remains valid if condition~\eqref{asm:qmem} in the definition of (QMEM) is replaced by the following more general assumption: there exists a locally bounded function $H_m: \mathbb{R}^{m-1} \to \mathbb{R}$ such that, for $\mathbb{P}$-almost every $\omega$, there is a constant $C(\omega) > 0$ with the property that for all $f_0,\dots,f_{m-1} \in \mathfrak{B}$ and all integers $0 \le k_0 \le \cdots \le k_{m-1}$,
\begin{equation}\label{asm:mem}
\begin{aligned}
&\left|\int_M \prod_{j=0}^{m-1} f_j \circ T_{\omega}^{k_j} \, d\mu_{\omega}
- \prod_{j=0}^{m-1} \int_M f_j \, d\mu_{\sigma^{k_j}\omega} \right| \\
&\quad \le H_m\left(C(\sigma^{k_0}\omega),\dots,C(\sigma^{k_{m-2}}\omega)\right)
\, e^{-\gamma \min_{0 \le j \le m-2}(k_{j+1}-k_j)^{\theta}}
\prod_{j=0}^{m-1} \|f_j\|_{\mathfrak{B}}.
\end{aligned}
\end{equation}
\end{remark}

Indeed, condition~\eqref{asm:qmem} implies \eqref{asm:mem}, since
\begin{equation*}
\begin{aligned}
&\left|\int_M \prod_{j=0}^{m-1} f_j \circ T_{\omega}^{k_j} \, d\mu_{\omega}
- \prod_{j=0}^{m-1} \int_M f_j \, d\mu_{\sigma^{k_j}\omega} \right|\\
&\quad\le e^{\gamma \max_{1 \le j \le m-1} N^*(\sigma^{k_j}\omega)^{\theta}}
\, e^{-\gamma \min_{0 \le j \le m-2}(k_{j+1}-k_j)^{\theta}}
\prod_{j=0}^{m-1} \|f_j\|_{\mathfrak{B}},
\end{aligned}
\end{equation*}
so that \eqref{asm:mem} holds with
\begin{equation*}
C(\omega) = e^{\gamma N^*(\omega)^{\theta}},
\quad
H_m(x_1,\dots,x_{m-1}) = \max_{1 \le j \le m-1} x_j.
\end{equation*}

\section{Examples}\label{explesec}

We present two examples that illustrate some aspects of Theorems \ref{cltthm} and \ref{pltthm}. The claims will be proved in Section \ref{expleproofsec}. The examples are chosen to highlight the role of the random threshold $N^*$.

In both examples, we will use a \textit{renewal chain} for the base transformation $\sigma$. Let $(p_k)_{k\geq 0}$ be a sequence of probabilities with $p_k>0$ and $\sum_{k=0}^{\infty} p_k = 1$. A Markov chain $(X_n)$ on $\N_0=\{0,1,\dots\}$ is called the renewal chain with probabilities $(p_k)$ if
\begin{equation*}
\P(X_{n+1}=k \mid X_n=j) =
\begin{cases}
p_k & \text{if } j=0,\\
1 & \text{if } j\geq 1 \text{ and } k=j-1,\\
0 & \text{otherwise}.
\end{cases}
\end{equation*}
If $\sum_{k=0}^{\infty} k p_k < \infty$, then this Markov chain is positively recurrent, and the stationary measure $\lambda$ is given by
\begin{equation*}
\lambda(j)= \frac{\sum_{k=j}^{\infty} p_k}{\sum_{k=0}^{\infty} k p_k}.
\end{equation*}
By a slight abuse of notation, we also denote by $\sigma$ the corresponding shift on $(\Omega=\N_0^{\N},\P)$ and refer to it as the renewal chain with probabilities $(p_k)$. Each $\omega\in \Omega$ corresponds to a realised path $(X_n)$ of the Markov chain.

In \cite[Appendix A]{DHS23} (see also \cite[Appendix A]{buzzi99}), an example was constructed where the quenched CLT fails, even though the fibre maps are expanding. In our first example, we will discuss this in the context of mixing of all orders, and our Theorem \ref{cltthm}. Furthermore, we demonstrated how their example, in a certain sense, stands at the edge of not satisfying the quenched CLT.

\begin{example}\label{renewalexple}
Let $\sigma$ be the renewal chain with $p_k = \frac{c}{(1+k)^3}$, where $c>0$ is a constant so that $\sum_{k=0}^{\infty} p_k = 1$, and let the maps $T_{\omega}:[0,1)\rightarrow [0,1)$ be given by 
\begin{equation*}
T_{\omega}(x) := \begin{cases}
2x \; \text{(mod 1)} & \textit{if } \omega_0=0\\
x & \textit{otherwise}.
\end{cases}
\end{equation*}
In \cite[Appendix A]{DHS23} it was shown that the asymptotic variance $\sigma_{\mathbf{f}}$ does not converge, for suitably chosen $\mathbf{f}$, hence the quenched CLT does not hold. 

However, if we instead set $p_k = \frac{c}{(1+k)^{3+\epsilon}}$, for some $\epsilon>0$, then the system $(T_{\omega})$ is quenched stretched exponentially mixing of all orders, and the random threshold $N^*$ is integrable. By Theorem \ref{cltthm}, it satisfies the quenched CLT.
\end{example}

\begin{remark}\label{renewalrem}
The system $(T_{\omega})$ is, in fact, quenched exponentially mixing of all orders for every choice of $(p_k)$ satisfying $\sum_{k=0}^{\infty} k p_k < \infty$. This follows from the fact that the doubling map $T_0$ is applied along a set of times of asymptotic density $p_0$, by the strong law of large numbers.

However, in general, the associated random threshold $N^*$ is not integrable, and therefore the main theorem does not apply in this regime, and we have to move to the stretched exponential regime.

If instead $p_k \asymp k^{-4-\epsilon}$ for some $\epsilon>0$, then the random threshold $N^*$ in the exponential regime is integrable.
\end{remark}

Now we consider the same example in view of the PLT. By Remark \ref{renewalrem}, it is quenched exponentially mixing of all orders. In contrast to Theorem \ref{cltthm} for the CLT, where we required the threshold $N^*$ to be integrable, Theorem \ref{pltthm} imposes no such condition.

However, Example \ref{renewalexple} does not satisfy the PLT, since the aperiodicity condition (APER) fails. Indeed, after every return to a small ball $B_r(x^*)$, there is a positive probability of remaining in the ball, namely when the next transformation is $ID$. As a consequence, the limiting law is compound Poisson, similarly to \cite{AFGTV25} (see also \cite[Theorem 3.3]{Z22rare} in the deterministic setting).

Although the arguments in these references do not directly apply here, the simple structure of the example allows one to analyse the limit law using only properties of the doubling map. Since this is not essential for our purposes, we omit the proof.

We can, however, modify Example \ref{renewalexple} so that (APER) is satisfied by replacing $ID$ with a non-periodic transformation. In this way, we obtain an example that satisfies the quenched PLT but not the CLT, highlighting the role of the random threshold $N^*$. We use a slightly different example, based on applying an Anosov flow at different times, to highlight a spectral approach.

\begin{example}\label{pltexple}
Let $\sigma$ be the renewal chain with $p_k = \frac{c}{(1+k)^3}$, where $c>0$ is chosen so that $\sum_{k=0}^{\infty} p_k = 1$. Let $\phi$ be a $C^4$ contact Anosov flow (as defined in \cite{Lcontact}) preserving a smooth measure, and let $(t_k)_{k\geq 1}$ be a sequence with $\sum_{k\geq 1} t_k = 1$. We define the fibre maps by
\begin{equation*}
T_{\omega} :=
\begin{cases}
\phi_1 & \text{if } \omega_0=0,\\
\phi_{t_k} & \text{if } \omega_0=k\geq 1.
\end{cases}
\end{equation*}
Then the system $(T_{\omega})$ satisfies the quenched PLT, but not the quenched CLT.

If instead $p_k = \frac{c}{(1+k)^{3+\epsilon}}$ for some $\epsilon>0$, then the system satisfies both the quenched PLT and the quenched CLT.
\end{example}

This example illustrates that the integrability of the random threshold $N^*$ is essential for the CLT, but not for the PLT.

\part{Proofs}

\section{Proof of the CLT}\label{cltproofsec}

\begin{lemma}\label{intsumlem}
Let $\phi\in L^1(\P)$ and $\Lambda_N\rightarrow \infty$ be a sequence of positive real numbers. Then, for $\P$-almost every $\omega$ it holds that
\begin{equation*}
\lim_{N\rightarrow\infty} \frac{1}{N} \sum_{n=0}^{N-1} \phi(\sigma^n \omega) \cdot \one_{|\phi(\sigma^n \omega)|> \Lambda_N} = 0.
\end{equation*}
\end{lemma}
\begin{proof}
Let $\epsilon>0$, then there is a $L>0$ such that
\begin{equation*}
\E(\phi \cdot \one_{|\phi|> L}) <\epsilon.
\end{equation*}
Let $N_{\epsilon}\geq 1$ be such that $\Lambda_N>L$ for $N>N_{\epsilon}$, and
\begin{equation*}
\frac{1}{N} \sum_{n=0}^{N-1} \phi(\sigma^n \omega) \cdot \one_{|\phi(\sigma^n \omega)|> L} < 2 \E(\phi \cdot \one_{|\phi|> L}) \quad \forall N>N_{\epsilon}.
\end{equation*}
Then it holds that
\begin{equation*}
\frac{1}{N} \sum_{n=0}^{N-1} \phi(\sigma^n \omega) \cdot \one_{|\phi(\sigma^n \omega)|> \Lambda_N} \leq \frac{1}{N} \sum_{n=0}^{N-1} \phi(\sigma^n \omega) \cdot \one_{|\phi(\sigma^n \omega)|> L} < 2 \epsilon,
\end{equation*}
for $N>N_{\epsilon}$.
\end{proof}

\begin{lemma}[Existence of the asympotic variance]\label{sigmaexlem}
Under the assumptions of Theorem \ref{cltthm}, both of the limits in \eqref{sigmadef} exist, and it holds that
\begin{equation*}
\lim_{N\rightarrow \infty} \frac{1}{N} \int_M (S_{N, \omega}(\mathbf{f}))^2 \d\mu_{\omega} = \E\left( \int_M f_{\omega}^2 \d\mu_{\omega} + 2\sum_{n=1}^{\infty} \int_M f_{\omega} \cdot f_{\sigma^n\omega} \circ T_{\omega}^n \d\mu_{\omega} \right).
\end{equation*}
\end{lemma}
\begin{proof}
The claim follows once we show that 
\begin{itemize}
\item[(A)] $\frac{1}{N} \int_M (S_{N, \omega}(\mathbf{f}))^2 \d\mu_{\omega}$ converges almost surely as $N\rightarrow \infty$,
\item[(B)] for almost every $\omega$, the partial sums $\sum_{n=1}^{N} \int_M f_{\omega} \cdot f_{\sigma^n \omega} \circ T_{\omega}^n \d\mu_{\omega}$ converge as $N\rightarrow \infty$,
\item[(C)] the infinite sum $\sum_{n=1}^{\infty} \int_M f_{\omega} \cdot f_{\sigma^n \omega} \circ T_{\omega}^n \d\mu_{\omega}$ is in $L^1(\P)$,
\item[(D)] for almost all $\omega$ it holds that 
\begin{equation*}
\lim_{N\rightarrow \infty} \frac{1}{N} \int_M (S_{N, \omega}(\mathbf{f}))^2 \d\mu_{\omega} = \E\left( \int_M f_{\omega}^2 \d\mu_{\omega} + 2\sum_{n=1}^{\infty} \int_M f_{\omega} \cdot f_{\sigma^n \omega} \circ T_{\omega}^n \d\mu_{\omega} \right).
\end{equation*}
\end{itemize}

Using quenched mixing \eqref{asm:qmem}, for each $N\geq 1$, it holds that
\begin{equation}\label{sumn*}
\left| \sum_{n=1}^{N} \int_M f_{\omega} \cdot f_{\sigma^n \omega} \circ T_{\omega}^n \d\mu_{\omega} \right| \leq N^*(\omega) + K,
\end{equation}
and for $N'>N$ it holds that
\begin{equation}\label{sumgn*}
\left| \sum_{n=N}^{N'} \int_M f_{\omega} \cdot f_{\sigma^n} \circ T_{\omega}^n \d\mu_{\omega} \right| \leq Ke^{-\frac{\gamma}{2} N^{\theta}} + \max(N^*(\omega) - N, 0),
\end{equation}
for a big enough constant $K>0$. Then assertions (B) and (C) follow immediately by dominated convergence.

For each $N$ we expand
\begin{align*}
\int_M (S_{N, \omega}(\mathbf{f}))^2 \d\mu_{\omega} & = \sum_{n,n' \in \{0,...,N-1\}} \int_M f_{\sigma^n \omega} \circ T_{\omega}^n \cdot f_{\sigma^{n'} \omega} \circ T_{\omega}^{n'} \d\mu_{\omega}\\
&= \sum_{n=0}^{N-1} \int_M f_{\sigma^n \omega}^2 \d\mu_{\sigma^n \omega} + 2 \sum_{n=0}^{N-2} \sum_{k=1}^{N-n-1} \int_M f_{\sigma^n \omega} \cdot f_{\sigma^{n+k} \omega} \circ T_{\omega}^{k} \d\mu_{\sigma^n \omega}
\end{align*}
By ergodicity, the first summand is asymptotically equivalent to $N\cdot \E\left( \int_M f_{\omega}^2 \d\mu_{\omega}\right)$. Using \eqref{sumn*} and \eqref{sumgn*} it holds that
\begin{align*}
&\left| \sum_{n=0}^{N-2} \sum_{k=1}^{N-n-1} \int_M f_{\sigma^n \omega} \cdot f_{\sigma^{n+k} \omega} \circ T_{\omega}^{k} \d\mu_{\sigma^n \omega} - \sum_{n=0}^{N-1} \sum_{k=1}^{\infty} \int_M f_{\sigma^n \omega} \cdot f_{\sigma^{n+k} \omega} \circ T_{\omega}^{k} \d\mu_{\sigma^n \omega} \right|\\
&\ll \left| \sum_{n=0}^{N-\sqrt{N}} \sum_{k=1}^{N-n-1} \int_M f_{\sigma^n \omega} \cdot f_{\sigma^{n+k} \omega} \circ T_{\omega}^{k} \d\mu_{\sigma^n \omega} - \sum_{n=0}^{N-\sqrt{N}} \sum_{k=1}^{\infty} \int_M f_{\sigma^n \omega} \cdot f_{\sigma^{n+k} \omega} \circ T_{\omega}^{k} \d\mu_{\sigma^n \omega} \right|\\
&+ \left| \sum_{n=N-\sqrt{N} + 1}^{N} \sum_{k=1}^{\infty} \int_M f_{\sigma^n \omega} \cdot f_{\sigma^{n+k} \omega} \circ T_{\omega}^{k} \d\mu_{\sigma^n \omega} \right|\\
&\ll \sum_{n=0}^{N-\sqrt{N}} \tilde{N}^* (\sigma^n \omega) + N e^{-\frac{\gamma}{2} N^{\frac{\theta}{2}}} + \sum_{n=N-\sqrt{N} + 1}^{N} N^*(\sigma^n \omega) + \sqrt{N}\\
&=o(N),
\end{align*}
where $\tilde{N}^*(\omega)=N^*(\omega) \cdot \one_{N^*(\omega)>\sqrt{N}}$. Ergodicity of $\sigma$ then yields that 
\begin{equation*}
\sum_{n=0}^{N-2} \sum_{k=1}^{N-n-1} \int_M f_{\sigma^n \omega} \cdot f_{\sigma^{n+k} \omega} \circ T_{\omega}^{k} \d\mu_{\sigma^n \omega} = N \cdot \E\left( \sum_{k=1}^{\infty} \int_M f_{\omega} \cdot f_{\sigma^k \omega} \circ T_{\omega}^k \d\mu_{\omega}\right) + o(N),
\end{equation*}
and (A) and (D) follow.
\end{proof}

\begin{proof}[Proof of Theorem \ref{cltthm}]
Wlog assume $\sup_{\omega \in \Omega} \|f_{\omega}\|_{\mathfrak{B}} + \|f_{\omega}\|_{L^{\infty}}=1$.

(i) The asymptotic variance $\sigma_{\mathbf{f}}$ exists by Lemma \ref{sigmaexlem}. Utilising the method of moments, convergence in \eqref{cltcon} will follow once we show that
\begin{equation}\label{momentcltclaim}
\lim_{N \rightarrow\infty} \frac{1}{N^{\frac{m}{2}}} \int_M (S_{N, \omega}(\mathbf{f}))^m \d\mu_{\omega} = \begin{cases}
0 & \text{if } m \text{ is odd}\\
\sigma_f^m (m-1)!! & \text{if } m \text{ is even},
\end{cases} 
\end{equation}
for all $m\geq 1$ and for $\P$-a.e.\ $\omega$. For $m=1$ we have $\int_M S_{N, \omega}(\mathbf{f}) \d\mu_{\omega} = 0$ by definition, and by Lemma \ref{sigmaexlem} it holds $\lim_{N\rightarrow \infty} \frac{1}{N} \int_M (S_{N, \omega}(\mathbf{f}))^2 \d\mu_{\omega} = \sigma_{\mathbf{f}}^2$. So in the following, we will focus on $m\geq 3$.

Expanding $(S_{N, \omega}(\mathbf{f}))^m$ we obtain 
\begin{equation*}
(S_{N, \omega}(\mathbf{f}))^m = \sum_{\textbf{k}=(k_1,\dots,k_{m}) \in \{0,\dots,N-1\}^m} \prod_{j=1}^{m} f_{\sigma^{k_j} \omega} \circ T_{\omega}^{k_j}.
\end{equation*}
From now on, to simplify notation, we will omit the dependence on $N$ from our notation. 

The rest of the proof is organised in steps (ii)--(v). Due to the random threshold $N^*$, we cannot apply quenched mixing to all well-separated multi-indices $\mathbf{k}$, as one would in the deterministic setting. To overcome this difficulty, we introduce the notion of bad pairs in Step (ii). In Step (iii), we treat the case where no bad pairs are present; this part is analogous to the deterministic argument.

In Step (iv), we reduce the $m$-th moment to the $(m-1)$-st one, at the cost of replacing the sum over $\{0,\dots, N-1\}$ by a sum over arbitrary subsets $A \subset \{0,\dots, N-1\}$. Finally, in Step (v) we estimate the resulting $(m-1)$-st moment, where $S_{N,\omega}(\mathbf{f})$ is replaced by the corresponding sum over an arbitrary subset $A$.

An induction argument then yields control of moments of all orders.

(ii) The idea is that we want to use mixing whenever $|k_j-k_{j'}|$ is "big" for some $(j,j')$, and simply group the functions as one whenever $|k_j-k_{j'}|$ is "small". However, for example if $k_j>k_{j'}$ and we group together
\begin{equation*}
f_{\sigma^{k_j} \omega} \circ T_{\omega}^{k_j} \cdot f_{\sigma^{k_{j'}} \omega} \circ T_{\omega}^{k_{j'}} = \left( f_{\sigma^{k_j} \omega} \circ T_{\omega}^{k_j-k_{j'}} \cdot f_{\sigma^{k_{j'}} \omega} \right) \circ T_{\omega}^{k_{j'}}
\end{equation*}
the $\mathfrak{B}$ norm of the function $f_{\sigma^{k_j} \omega} \circ T_{\omega}^{k_j-k_{j'}} \cdot f_{\sigma^{k_{j'}}}$ is of order $K^{k_j-k_{j'}}$, where the constant
\begin{equation*}
K=\sup_{g\in\mathfrak{B}, \|g\|_{\mathfrak{B}}=1, \omega\in\Omega} \|g\circ T_{\omega} \|_{\mathfrak{B}}
\end{equation*}
is finite because of assumption (R). When applying mixing, we have to compensate for this gain by only separating pairs for sufficiently large gaps $|k_j-k_{j'}|$. Accordingly we define $G_1 , \dots, G_{m^2+1} $ recursively by choosing a sequence $s(N)\nearrow \infty$ and setting
\begin{equation*}
G_1=\frac{1}{\gamma} (\log N)^{\frac{100m}{\theta}} \quad \text{and} \quad G_{l+1}=\log N \cdot G_l^{\frac{1}{\theta}} \;\forall l\geq 1.
\end{equation*}

Let $k_0=0$. Notice that, for every $\textbf{k}=(k_1,..., k_m)$, by a straightforward counting argument, there is a minimal $l(\textbf{k})\in\{1,\dots, {{m\choose 2}+m+2}\}$ such that there is no pair $(j,j')\in \{0,\dots m\}$ such that
\begin{equation*}
|k_j-k_{j'}|\in [G_{l(\textbf{k})}, G_{l(\textbf{k}) +1} ),
\end{equation*}
or in other words, for every pair $(j,j')\in \{0,\dots m\}$, it holds that
\begin{equation*}
|k_j-k_{j'}| < G_{l(\textbf{k})} \quad \text{or} \quad |k_j-k_{j'}| \geq G_{l(\textbf{k})+1}.
\end{equation*}
We call a collection $\Gamma\subset \{1,\dots,m\}$ a \textit{cluster} for $\textbf{k}$ if it is a maximal set of indices with
\begin{equation*}
|k_j-k_{j'}| < G_{l(\textbf{k})} \quad \forall j,j'\in \Gamma.
\end{equation*}
Denote by $\C(\textbf{k})$ the collection of all clusters for $\textbf{k}$ and by $s(\textbf{k})=\# \C(\textbf{k})$ the number of distinct clusters. For each $\textbf{k}$ we will split up $\int_M \prod_{j=1}^{m} f_{\sigma^{k_j} \omega} \circ T_{\omega}^{k_j} \d\mu_{\omega}$ into
\begin{equation}\label{clustersplit}
\prod_{\Gamma\in\C(\textbf{k})} \int_M \prod_{j\in \Gamma} f_{\sigma^{k_j} \omega} \circ T_{\omega}^{k_j} \d\mu_{\omega},
\end{equation}
using mixing of all orders. 

However, there is an issue; according to \eqref{asm:qmem} we cannot apply mixing of all orders if, for some $(j,j')$ with $k_{j}< k_{j'}$, it holds that $|k_{j'}-k_{j}|<N^*(\sigma^{k_{j}} \omega)$. Note that this is not a problem if $N^*(\sigma^{k_{j}} \omega) < G_1 = (\log N)^{\frac{100m}{\theta}}$, since in this case we will never separate $k_j$ and $k_{j'}$ anyway. Accordingly, we will say that $(k,k')$ with $0\leq k < k' \leq N-1$ is a \textit{bad pair} if
\begin{equation*}
N^*(\sigma^{k}\omega) > (\log N)^{\frac{100m}{\theta}} \quad \text{and} \quad |k'-k|<N^*(\sigma^{k'}\omega).
\end{equation*}
For a multi-index $\textbf{k}\in \{0,\dots, N-1\}^m$ we will say that $(j,j')\in \{1,\dots,m\}^2$ is a \textit{bad pair for \textbf{k}} if $k_j<k_{j'}$ and $(k_j,k_{j'})$ is a bad pair. Note that the order is important; if $(j,j')$ is a bad pair, then $(j',j)$ is not a bad pair.

Let $\B=\{l\leq N-1 \;|\; N^*(\sigma^l \omega) > (\log N)^{\frac{100m}{\theta}} \}$, then we define the \textit{danger zone} $\D$ as
\begin{equation*}
\D = \{0,\dots N-1\} \cap \bigcup_{l\in \B} \{l,\dots l+N^*(\sigma^l \omega) - 1\}.
\end{equation*}
Then $\D$ is exactly the set of those $l$ such that there exists an $l'$ where either $(l,l')$ or $(l',l)$ is a bad pair. We compute
\begin{equation*}
\# \D\leq \sum_{n=0}^{n-1} N^*(\sigma^n \omega) \cdot \one_{ N^*(\sigma^n \omega) > (\log N)^{\frac{100m}{\theta}}} = o(N).
\end{equation*} 
At the same time, denote by $BP$ the set of all possible bad pairs, then it also holds that
\begin{align*}
\# BP&=\# \{(k,k') \;|\; (k,k') \; \text{is a bad pair}\}\\ 
&\leq \# \{(k,k') \;|\; (k,k') \; N^*(\sigma^{k}\omega) > (\log N)^{\frac{100m}{\theta}} \; \text{and} \; k'\in \{k+1,\dots, k+N^*(\sigma^l \omega) - 1\}\}\\
&\leq \sum_{n=0}^{n-1} N^*(\sigma^n \omega) \cdot \one_{ N^*(\sigma^n \omega) > (\log N)^{\frac{100m}{\theta}}}\\
&=o(N).
\end{align*}

(iii) Now, if \textbf{k} is such that there are no bad pairs $(j,j')$, then we can simply use mixing and split up in the individual clusters as in \eqref{clustersplit}. Therefore define
\begin{align*}
\Delta_{good}^{(m)} &= \{\textbf{k}\in \{0,\dots, N-1\}^m \;|\; \forall (j,j')\in \{1,...,m\}^2 \; \text{ it holds that } \; (k_j,k_{j'}) \not\in BP \} ,\\
\Delta_{bad}^{(m)}& = \{\textbf{k}\in \{0,\dots, N-1\}^m \;|\; \exists (j,j')\in \{1,...,m\}^2 \; \text{ such that } \; (k_j,k_{j'}) \in BP \} ,
\end{align*}
then we can split up
\begin{equation*}
(S_{N, \omega}(\mathbf{f}))^m= \sum_{\textbf{k} \in \Delta^{(m)}_{good}} \prod_{j=1}^{m} f_{\sigma^{k_j} \omega} \circ T_{\omega}^{k_j} + \sum_{\textbf{k} \in \Delta^{(m)}_{bad}} \prod_{j=1}^{m} f_{\sigma^{k_j} \omega} \circ T_{\omega}^{k_j}.
\end{equation*}
Therefore, the claim \eqref{momentcltclaim} will follow once we show
\begin{equation}\tag{$\text{Good}_m$}
\lim_{N \rightarrow\infty} \frac{1}{N^{\frac{m}{2}}} \int_M \sum_{\textbf{k} \in \Delta^{(m)}_{good}} \prod_{j=1}^{m} f_{\sigma^{k_j} \omega} \circ T_{\omega}^{k_j} \d\mu_{\omega} = \begin{cases}
0 & \text{if } m \text{ is odd}\\
\sigma_{\mathbf{f}}^m (m-1)!! & \text{if } m \text{ is even},
\end{cases} 
\end{equation}
whereas the contribution of the bad terms is negligible
\begin{equation}\tag{$\text{Bad}_m$}
\left|\sum_{\textbf{k} \in \Delta^{(m)}_{bad}} \int_M \prod_{j=1}^{m} f_{\sigma^{k_j} \omega} \circ T_{\omega}^{k_j} \d\mu_{\omega}\right| = o(N^{\frac{m}{2}}).
\end{equation}

We will first show ($\text{Good}_m$) for all $m$, which is easier anyway since we can simply use mixing. 

For $\textbf{k}\in \{0,\dots, N-1\}^m$ and a cluster $\Gamma\in \C(\mathbf{k})$, denote $j^*(\Gamma)\in \Gamma$ as the index such that $k_{j^*(\Gamma)}=\min_{j\in \Gamma} k_j$, Then we can write 
\begin{equation*}
\prod_{j\in \Gamma} f_{\sigma^{k_j} \omega} \circ T_{\omega}^{k_j} = \left( \prod_{j\in \Gamma} f_{\sigma^{k_j} \omega} \circ T_{\sigma^{k_{j^*(\Gamma)}} \omega}^{k_j - k_{j^*(\Gamma)}} \right)\circ T_{\omega}^{k_{j^*(\Gamma)}}.
\end{equation*}
The $\mathfrak{B}$ norm of this function is
\begin{equation*}
\left| \left| \prod_{j\in \Gamma} f_{\sigma^{k_j} \omega} \circ T_{\sigma^{k_{j^*(\Gamma)}} \omega}^{k_j - k_{j^*(\Gamma)}} \right| \right|_{\mathfrak{B}} \ll K^{\#\Gamma \cdot G_{l(\textbf{k})}}.
\end{equation*}
Since $|k_{j^*(\Gamma)} - k_{j^*(\Gamma')}|\geq G_{l(\textbf{k})+1}$, we can use mixing of all orders to find that
\begin{equation}\label{cltsplitindividual}
\begin{aligned}
&\left| \int_M \prod_{j=1}^{m} f_{\sigma^{k_j} \omega} \circ T_{\omega}^{k_j} \d\mu_{\omega} - \prod_{\Gamma\in\C(\textbf{k})} \int_M \prod_{j\in \Gamma} f_{\sigma^{k_j} \omega} \circ T_{\omega}^{k_j} \d\mu_{\omega}\right| \\
&\ll K^{m G_{l(\textbf{k})}} e^{-\gamma G_{l(\textbf{k})+1}^{\theta}} \ll N^{-99m}.
\end{aligned}
\end{equation}

Summing this bound over $\textbf{k}\in\{1,\dots,m\}$ yields 
\begin{equation}\label{cltsplit}
\left| \int_M \left( (S_{N, \omega}(\mathbf{f}))^m - \sum_{\textbf{k}=(k_1,\dots,k_{m}) \in \{0,\dots,N-1\}^m} \prod_{\Gamma\in\C(\textbf{k})} \int_M \prod_{j\in \Gamma} f_{\sigma^{k_j} \omega} \circ T_{\omega}^{k_j} \right) \d\mu_{\omega}\right| \ll N^{-98m}.
\end{equation}

For $s=1,\dots,m$ denote by 
\begin{equation*}
\Delta^s_{good} = \{\textbf{k} \in \Delta_{good} \;|\; s(\textbf{k})=s\}
\end{equation*}
the set of all $\textbf{k}$ having exactly $s$ clusters.

If $s>\frac{m}{2}$ and $\textbf{k}\in \Delta^s_{good}$, then one of the clusters is necessarily a singleton. Since we defined $f_{\omega} = f - \int_M f \d\mu_{\omega}$, it follows that
\begin{equation}\label{manycluster}
\prod_{\Gamma\in\C(\textbf{k})} \int_M \prod_{j\in \Gamma} f_{\sigma^{k_j} \omega} \circ T_{\omega}^{k_j} \d\mu_{\omega} = 0.
\end{equation}

On the other hand, if $s<\frac{m}{2}$, then it holds that 
\begin{equation*}
\# \Delta^s_{good} \ll N^s G_{m^2}^{m-s} \ll N^{\frac{m}{2} -\frac{1}{2} +\frac{1}{100}}.
\end{equation*}
Since $f$ is bounded, we therefore have
\begin{equation}\label{fewcluster}
\left| \sum_{\textbf{k}\in \Delta_{good}^s} \prod_{\Gamma\in\C(\textbf{k})} \int_M \prod_{j\in \Gamma} f_{\sigma^{k_j} \omega} \circ T_{\omega}^{k_j} \d\mu_{\omega}\right| \ll \# \Delta_{good}^s \ll N^{\frac{m}{2} -\frac{1}{2} +\frac{1}{100}}.
\end{equation}

Therefore all terms coming $\Delta_{good}^s$ with $s\neq \frac{m}{2}$ are negligible by \eqref{manycluster} and \eqref{fewcluster}. In particular, together with \eqref{cltsplit} this shows ($\text{Good}_m$) if $m$ is odd.

For the rest of the proof of ($\text{Good}_m$), we focus on the case where $m$ is even and $s=\frac{m}{2}$. 

If one of the clusters $\Gamma\in \C(\textbf{k})$ is a singleton, then as before it holds that
\begin{equation*}
\prod_{\Gamma\in\C(\textbf{k})} \int_M \prod_{j\in \Gamma} f_{\sigma^{k_j} \omega} \circ T_{\omega}^{k_j} \d\mu_{\omega} = 0.
\end{equation*}
Therefore, the only remaining case is when all of the clusters consist of exactly two elements. More explicitly, letting 
\begin{equation*}
\G_m=\{\textbf{k}\in \Delta^{\frac{m}{2}}_{good} \;|\; \#\Gamma=2 \; \forall \Gamma\in \C(\textbf{k})\},
\end{equation*}
then the claim ($\text{Good}_m$) will follow once we show
\begin{equation}\tag{$\text{Good'}_m$}
\lim_{N \rightarrow\infty} \frac{1}{N^{\frac{m}{2}}} \sum_{\textbf{k}\in \G_m} \prod_{\Gamma\in\C(\textbf{k})} \int_M \prod_{j\in \Gamma} f_{\sigma^{k_j} \omega} \circ T_{\omega}^{k_j} \d\mu_{\omega} = \sigma_{\mathbf{f}}^m (m-1)!!.
\end{equation}
We will show ($\text{Good'}_m$) by induction on $s=\frac{m}{2}$.

For $s=1$, first note that we might as well just add back in all the bad pairs that form a cluster, since
\begin{equation*}
\sum_{(k,k') \text{ is a cluster and a bad pair}} \int_M f_{\sigma^k\omega} \circ T_{\omega}^{k} \cdot f_{\sigma^{k'}\omega} \circ T_{\omega}^{k'} \ll \# BP = o(N).
\end{equation*}
Therefore, we may as well just sum over all pairs that form a cluster, regardless of whether it is a bad pair. Therefore, let
\begin{equation*}
\tilde{G}=\{(k,k') \;|\; s(k,k')=1\}
\end{equation*}
then ($\text{Good'}_2$) is equivalent to showing
\begin{equation*}
\lim_{N \rightarrow\infty} \frac{1}{N} \sum_{(k,k')\in \tilde{G}} \int_M f_{\sigma^{k} \omega} \circ T_{\omega}^{k} \cdot f_{\sigma^{k'} \omega} \circ T_{\omega}^{k'} \d\mu_{\omega} = \sigma_{\mathbf{f}}^2.
\end{equation*}
Rewriting, we obtain
\begin{align*}
&\sum_{(k,k')\in \tilde{G}} \int_M f_{\sigma^{k} \omega} \circ T_{\omega}^{k} \cdot f_{\sigma^{k'} \omega} \circ T_{\omega}^{k'} \d\mu_{\omega} = \sum_{n=0}^{N-1} \int_M f_{\sigma^n \omega}^2 \d\mu_{\sigma^n \omega}+ 2 \Sigma,
\end{align*}
where
\begin{equation}\label{sigmasum}
\begin{aligned}
\Sigma&=\sum_{l=1}^3 \sum_{n=G_{l-1}}^{G_l-1} \sum_{k=1}^{\min(G_l - 1, N - 1 - n)} \int_M f_{\sigma^n \omega} \cdot f_{\sigma^{n+k} \omega} \circ T_{\omega}^{k} \d\mu_{\sigma^n \omega} \\
& + \sum_{n=G_{3}}^{N-1} \sum_{k=1}^{\min(G_4 - 1, N - 1 - n)} \int_M f_{\sigma^n \omega} \cdot f_{\sigma^{n+k} \omega} \circ T_{\omega}^{k} \d\mu_{\sigma^n \omega}
\end{aligned}
\end{equation}
and $G_0=0$. This sum might look a bit complicated, but that is only because we had to take such great care of the exact size of the gaps. We will compare with the full sum 
\begin{equation}\label{fullsum}
\sum_{n=0}^{N-2} \sum_{k=1}^{N-n-1} \int_M f_{\sigma^n \omega} \cdot f_{\sigma^{n+k} \omega} \circ T_{\omega}^{k} \d\mu_{\sigma^n \omega}.
\end{equation}
Notice that 
\begin{itemize}
\item every summand in \eqref{sigmasum} appears only once and also appears in the full sum \eqref{fullsum},
\item all of the summands of the full sum \eqref{fullsum} that are missing from \eqref{sigmasum} necessarily have $k\geq G_1$.
\end{itemize}
Using \eqref{sumgn*}, it follows that 
\begin{align*}
&\left| \sum_{n=0}^{N-2} \sum_{k=1}^{N-n-1} \int_M f_{\sigma^n \omega} \cdot f_{\sigma^{n+k} \omega} \circ T_{\omega}^{k} \d\mu_{\sigma^n \omega} - \Sigma\right|\\
&\leq \sum_{n=0}^{N-2} \sum_{k=G_1}^{N-1} \left| \int_M f_{\sigma^n \omega} \cdot f_{\sigma^{n+k} \omega} \circ T_{\omega}^{k} \d\mu_{\sigma^n \omega} \right|\\
&\ll Ne^{- \gamma G_1^{\theta}} + \sum_{n=0}^{N-2} \max(N^*(\sigma^n \omega) - G_1,0)\\
& \ll N^{-98} + \sum_{n=0}^{N-1} N^*(\sigma^n \omega) \cdot \one_{N^*(\sigma^n \omega) > (\log N)^{\frac{100m}{\theta}}}\\
&=o(N)
\end{align*}
where in the last step we used Lemma \ref{intsumlem}. Using Lemma \ref{sigmaexlem}, it follows that 
\begin{equation*}
\sum_{(k,k')\in \tilde{G}} \int_M f_{\sigma^{k} \omega} \circ T_{\omega}^{k} \cdot f_{\sigma^{k'} \omega} \circ T_{\omega}^{k'} \d\mu_{\omega} = N \sigma_{\mathbf{f}}^2 + o(N),
\end{equation*}
and ($\text{Good}'_2$) holds.

Now suppose that ($\text{Good'}_m$) holds for some $\frac{m}{2}=s\geq 1$, we will show that it also holds for $s+1$. Let $\hat{\G}$ be the set of all $\textbf{k}\in \G_{m+2}$ where $k_1$ and $k_2$ are in the same cluster, and note that
\begin{equation}\label{ghateq}
\sum_{\textbf{k}\in \G_{m+2}} \prod_{\Gamma\in\C(\textbf{k})} \int_M \prod_{j\in \Gamma} f_{\sigma^{k_j} \omega} \circ T_{\omega}^{k_j} \d\mu_{\omega} = (m+1) \cdot \sum_{\textbf{k}\in \hat{\G}} \prod_{\Gamma\in\C(\textbf{k})} \int_M \prod_{j\in \Gamma} f_{\sigma^{k_j} \omega} \circ T_{\omega}^{k_j} \d\mu_{\omega}.
\end{equation}
This is because $k_1$ can be paired with any of $k_2,\dots, k_{m+2}$.

For $\textbf{k}'\in \G_m$ denote $\G(\textbf{k}')$ to be the set of all $\textbf{k}=(k_1,k_2,\textbf{k}')\in \hat{\G}$. It is not true that $\hat{\G} = \bigcup_{\textbf{k}'\in \G_m} \G(\textbf{k}')$ because this union does not into account any cluster among the last $m$ indices that is separated by more than $G_{{m\choose 2}+m+2}$, whereas in principle the maximum separation possible should be $G_{{m+2\choose 2}+m+4}$. However, using \eqref{asm:qmem} it is sraightforward to deduce that
\begin{align*}
&\left|\sum_{\textbf{k}\in \hat{\G}} \prod_{\Gamma\in\C(\textbf{k})} \int_M \prod_{j\in \Gamma} f_{\sigma^{k_j} \omega} \circ T_{\omega}^{k_j} \d\mu_{\omega} - \sum_{\textbf{k}'\in \G_m} \sum_{\textbf{k}\in \G(\textbf{k}')} \prod_{\Gamma\in\C(\textbf{k})} \int_M \prod_{j\in \Gamma} f_{\sigma^{k_j} \omega} \circ T_{\omega}^{k_j} \d\mu_{\omega}\right| \\
&\ll \#\hat{\G} e^{-\gamma G_{{m\choose 2}+m+2}^{\theta} } \ll N^m e^{-\gamma G_{{m\choose 2}+m+2}^{\theta} } \ll N^{-99m}.
\end{align*}
Notice that every summand of the second sum has a common factor, namely
\begin{equation*}
\prod_{\Gamma\in\C(\textbf{k}')} \int_M \prod_{j\in \Gamma} f_{\sigma^{k_j} \omega} \circ T_{\omega}^{k_j} \d\mu_{\omega}
\end{equation*}
and the remaining term is
\begin{equation}\label{k'sum}
\sum_{\textbf{k}\in \G(\textbf{k}')} \int_M f_{\sigma^{k_1} \omega} \circ T_{\omega}^{k_1} \cdot f_{\sigma^{k_2} \omega} \circ T_{\omega}^{k_2} \d\mu_{\omega}.
\end{equation}
In order to keep the notation readable, we will not write out this sum any more explicitly. Now, if we can show that
\begin{equation}\label{k'claim}
\begin{aligned}
&\sum_{\textbf{k}\in \G(\textbf{k}')} \int_M f_{\sigma^{k_1} \omega} \circ T_{\omega}^{k_1} \cdot f_{\sigma^{k_2} \omega} \circ T_{\omega}^{k_2} \d\mu_{\omega}\\
&=\sum_{(k,k') \in \Delta^2_{good}} \int_M f_{\sigma^k \omega} \circ T_{\omega}^{k} \cdot f_{\sigma^{k'} \omega} \circ T_{\omega}^{k'} \d\mu_{\omega} + o(N),
\end{aligned}
\end{equation}
for all $\textbf{k}'$, then we can use the induction hypotheses ($\text{Good'}_m$) and ($\text{Good'}_2$) to conclude
\begin{align*}
&\sum_{\textbf{k}\in\hat{\G}} \prod_{\Gamma\in\C(\textbf{k})} \int_M \prod_{j\in \Gamma} f_{\sigma^{k_j} \omega} \circ T_{\omega}^{k_j} \d\mu_{\omega} \\
& = \sum_{\textbf{k}'\in \G_m} \left(\left(\prod_{\Gamma\in\C(\textbf{k}')} \int_M \prod_{j\in \Gamma} f_{\sigma^{k_j} \omega} \circ T_{\omega}^{k_j} \d\mu_{\omega}\right) \cdot \left(\sum_{\textbf{k}\in \G(\textbf{k}')} \int_M f_{\sigma^{k_1} \omega} \circ T_{\omega}^{k_1} \cdot f_{\sigma^{k_2} \omega} \circ T_{\omega}^{k_2} \d\mu_{\omega}\right) \right)\\
& \quad + O(N^{-100})\\
&= \left(N^{\frac{m}{2}} \sigma_f^m (m-1)!! + o(N^{\frac{m}{2}})\right) \cdot \left( N \sigma_f^2 + o(N) \right) + O(N^{-100})\\
&= N^{\frac{m+2}{2}} \sigma_f^{m+2} (m-1)!! + o(N^{\frac{m}{2}}).
\end{align*} 
Using \eqref{ghateq}, the claim ($\text{Good}'_{m+2}$) follows.

To show \eqref{k'claim}, note that the summands that are in the full sum \eqref{fullsum} but not in \eqref{k'sum} have
\begin{itemize}
\item $k_1$ and $k_2$ are within distance $2 G_{{m+2\choose 2}+m+4}$ of one of the $k'_j$, 
\item $|k_1-k_2|>G_{{m+2\choose 2}+m+4}$, or
\item one or both of $k_1$ or $k_2$ form a bad pair with one of the $k'_j$.
\end{itemize}
In the first two cases, it is rather straightforward; Using a counting argument in the first case and mixing - keep in mind that $(k_1,k_2)$ (or $(k_2,k_1)$ if $k_1>k_2$) is not a bad pair - in the second case, we obtain that those cases contribute an error of order
\begin{equation*}
N e^{- \gamma G_{{m+2\choose 2}+m+4}^{\theta}} + G_{{m+2\choose 2}+m+4}^2 = o(N).
\end{equation*}
In the third case, notice that $k_1$ or $k_2$ are necessarily in the danger zone $\D$. Choose integers $L=L_N$ with $L\rightarrow \infty$ as $N\rightarrow \infty$ but $L \cdot \#\D =o(N)$. Splitting the third case up into the subcases when $|k_1-k_2|<L$ or $|k_1-k_2|\geq L$, we can once again apply a counting argument in the first subcase, or mixing in the second, to see that the contribution is only of order
\begin{equation*}
N e^{- \gamma L^{\theta}} + L\cdot \#\D = o(N).
\end{equation*}

Altogether we obtain
\begin{align*}
&\left|\sum_{\textbf{k}\in \G(\textbf{k}')} \int_M f_{\sigma^{k_1} \omega} \circ T_{\omega}^{k_1} \cdot f_{\sigma^{k_2} \omega} \circ T_{\omega}^{k_2} \d\mu_{\omega} - \sum_{n=0}^{N-2} \sum_{k=1}^{N-n-1} \int_M f_{\sigma^n \omega} \cdot f_{\sigma^{n+k} \omega} \circ T_{\omega}^{k} \d\mu_{\sigma^n \omega}\right|\\
& \ll N e^{- \gamma G_{{m+2\choose 2}+m+4}^{\theta}} + G_{{m+2\choose 2}+m+4}^2 + N e^{- \gamma L^{\theta}} + L\cdot \#D = o(N).
\end{align*}
This shows \eqref{k'claim}, and ($\text{Good'}_{m+2}$), and therefore ($\text{Good}_{m+2}$), follows. 

We have thus shown ($\text{Good}_{m}$) for all $m\geq 1$.

(iv) Note that, for $m=1$, it trivially holds that $\Delta_{bad}^{(1)}=\emptyset$, therefore ($\text{Bad}_1$) is satisfied. For $m=2$, we bound by the number of bad pairs, and it holds that
\begin{equation*}
\sup_{x\in M} \sum_{\textbf{k} \in \Delta^{(m)}_{bad}} \left| \prod_{j=1}^{m} f_{\sigma^{k_j} \omega} \circ T_{\omega}^{k_j}(x)\right| \ll \#BP =o(N),
\end{equation*}
so ($\text{Bad}_2$) also holds.

From now on, assume $m\geq 3$. For $p=1,\dots, {m \choose 2}$ denote
\begin{equation}
\Lambda_m^p = \left\{\mathbf(l)=((l_1,l'_1), \dots ,(l_p,l'_p)) \in \{1,\dots,m\}^{2p} \;\left|
\begin{aligned}
& (l_1,l'_1), \dots ,(l_p,l'_p), l_i\neq l'_i \forall i \\
& (l_i,l'_i)\neq (l_{i'},l'_{i'}) \; \forall i\neq i'
\end{aligned}
\right.\right\},
\end{equation}
as the set of all collections of $p$ distinct pairs. For $p\geq 1$ and $\mathbf{l}\in\Lambda_m^p$ denote the collection of all multi-indices $\mathbf{k}$ where those are bad pairs by
\begin{equation*}
\Delta_{\mathbf{l}}^{(m)} = \left\{\textbf{k} \in \Delta_{bad}^{(m)} \;|\;(k_{j_i},k_{j'_i})\in BP \; \forall i=1,\dots ,p \right\},
\end{equation*}
note that $\mathbf{k}\in \Delta_{\mathbf{l}}^{(m)}$ might also have other bad pairs. We do not exclude cases where some of the pairs are \textit{overlapping}, so it might be that there are $i\neq i'$ with $l_i=l_{i'}$ or $l_i=l'_{i'}$ or $l'_i=l'_{i'}$. We can restrict to the case when $p\leq {m\choose 2}$ since this is the maximum number of bad pairs that an index $\mathbf{k}$ can have - namely when all pairs are bad pairs. By an inclusion-exclusion argument, it holds that
\begin{equation}\label{splitbadind}
\sum_{\textbf{k} \in \Delta^{(m)}_{bad}} \prod_{j=1}^{m} f_{\sigma^{k_j} \omega} \circ T_{\omega}^{k_j} = \sum_{p=1}^{{m\choose 2}} (-1)^{p+1} \frac{1}{p!} \sum_{\mathbf{l}\in\Lambda_m^p} \sum_{\textbf{k} \in \Delta_{\mathbf{l}}^{(m)}} \prod_{j=1}^m f_{\sigma^{k_j} \omega} \circ T_{\omega}^{k_j},
\end{equation}
the factor $\frac{1}{p!}$ comes from the fact that $\Delta_{\mathbf{l}}^{(m)}$ is invariant under permutations of the pairs $(l_1,l'_1), \dots ,(l_p,l'_p)$. Therefore the claim ($\text{Bad}_m$) will follow once we show that
\begin{equation*}
\left| \int_M \sum_{\textbf{k} \in \Delta_{\mathbf{l}}^{(m)}} \prod_{j=1}^m f_{\sigma^{k_j} \omega} \circ T_{\omega}^{k_j} \d\mu_{\omega} \right| = o(N^{\frac{m}{2}}),
\end{equation*}
for all $p\in \{1,\dots, {m\choose 2}\}$ and $\mathbf{l}\in \Lambda_m^p$.

Let $p\in \{1,\dots, {m\choose 2}\}$ and let $\mathbf{l}=((l_1,l'_1), \dots ,(l_p,l'_p)) \in \Lambda_m^p$. Let $a\in \{1,\dots, p\}$ be the maximum number of non-overlapping pairs, in fact, it holds that $a\leq \frac{m}{2}$. For simplicity of notation, we assume that the pairs $(l_1,l'_1), \dots, (l_a,l'_a)$ are non-overlapping. Now we split this sum even further, by considering all the bad pairs that $(l_1, l'_1),\dots, (l_a,l'_a)$ could be. More explicitly, for (not necessarily different) bad pairs $(b_1,b'_1),\dots, (b_a,b'_a)\in BP$, we consider
\begin{equation*}
\Gamma_{(b_1,b'_1),\dots, (b_a,b'_a)} = \{\textbf{k}\in \Delta_{\mathbf{l}}^{(m)} \;|\; (k_{l_1},k_{l'_1})=(b_1,b'_1), \dots , (k_{l_a}, k_{l'_a})=(b_a,b'_a)\}.
\end{equation*}
Notice that all $k_j$ can be freely chosen, unless $j$ is one of the $l_i$ or $l'_i$, and in this case it holds that
\begin{itemize}
\item if $j\in \{l_1,l'_1,\dots,l_a,l'_a\}$, then $k_j=b_j$, and
\item if $j\in \{l_1,l'_1,\dots,l_p,l'_p\} \setminus \{l_1,l'_1,\dots,l_a,l'_a\}$, then there is a proper subset $A_j\subset \{0,\dots,N-1\}$ such that $k_j\in A_j$. This set is determined by the pairs $(l_1,l'_1), \dots ,(l_p,l'_p), (b_1,b'_1),\dots, (b_a,b'_a)$ in the following way: if $i\leq a$ is so that $(l_i,j)\in \{(l_1,l'_1),\dots,(l_p,l'_p)\}$ (or $(j,l'_i)\in \{(l_1,l'_1),\dots,(l_p,l'_p)\}$) then $(b_{l_i}, k_j)\in BP$ (or $(k_j, b'_{l_i})\in BP$). These relations determine the sets $A_j$.
\end{itemize}
Therefore we can rewrite the set $\Gamma_{(b_1,b'_1),\dots, (b_a,b'_a)}$ as
\begin{equation*}
\Gamma_{(b_1,b'_1),\dots, (b_a,b'_a)} = \left\{ \textbf{k} \in \{0,\dots,N-1\}^m \;\left|
\begin{aligned}
&k_j=b_j \; \forall j\in \{l_1,l'_1,\dots,l_a,l'_a\}\\
&k_j\in A_j \; \forall j\in \{l_1,l'_1,\dots,l_p,l'_p\} \setminus \{l_1,l'_1,\dots,l_a,l'_a\}
\end{aligned}
\right. \right\}.
\end{equation*}

Let $r=\# \{l_1,l'_1,\dots,l_p,l'_p\} \setminus \{l_1,l'_1,\dots,l_a,l'_a\}$, and, for $A\subset\{0,\dots,N-1\}$, let
\begin{equation*}
S^{\omega}_A=\sum_{k\in A} f_{\sigma^k \omega} \circ T_{\omega}^k.
\end{equation*}
Then it holds
\begin{align*}
&\left| \int_M \sum_{\textbf{k} \in \Delta_{\mathbb{l}}^{(m)}} \prod_{j=1}^m f_{\sigma^{k_j} \omega} \circ T_{\omega}^{k_j} \d\mu_{\omega} \right|\\
&= \left| \sum_{(b_1,b'_1),\dots, (b_a,b'_a)\in BP} \int_M \sum_{\textbf{k} \in \Gamma_{(b_1,b'_1),\dots, (b_a,b'_a)}} \prod_{j=1}^m f_{\sigma^{k_j} \omega} \circ T_{\omega}^{k_j} \d\mu_{\omega} \right|\\
&= \left| \sum_{(b_1,b'_1),\dots, (b_a,b'_a)\in BP} \int_M \left(
\begin{aligned}
&\prod_{j\in \{l_1,l'_1,\dots,l_a,l'_a\}} f_{\sigma^{b_j} \omega} \circ T_{\omega}^{b_j}\\
& \cdot \prod_{j\in \{l_1,l'_1,\dots,l_p,l'_p\} \setminus \{l_1,l'_1,\dots,l_a,l'_a\}} S^{\omega}_{A_j}\\
& \cdot (S_{N, \omega}(\mathbf{f}))^{m-2a-r}
\end{aligned}
\right) \d\mu_{\omega} \right|\\
&\ll \sum_{(b_1,b'_1),\dots, (b_a,b'_a)\in BP} \prod_{j\in \{l_1,l'_1,\dots,l_p,l'_p\} \setminus \{l_1,l'_1,\dots,l_a,l'_a\}} \left| \left| S^{\omega}_{A_j} \right| \right|_{L^{m-2a}(\d\mu_{\omega})} \cdot \left|\int_M |S_{N, \omega}(\mathbf{f})|^{m-2a} \d\mu_{\omega}\right|^{\frac{m-2a-r}{m-2a}},
\end{align*}
where we used the generalised H\"{o}lder's inequality in the last step.

Now, finally, we claim that, for every $m\geq 1$, it holds that
\begin{equation}\tag{$\text{Sub}_m$}
\int_M |S^{\omega}_{A}|^m \d\mu_{\omega} \ll N^{\frac{m}{2}} \quad \forall A\subset \{0,\dots, N-1\}.
\end{equation}
Indeed, since $a\geq 1$, the claim ($\text{Sub}_m$) implies that
\begin{equation*}
\left| \int_M \sum_{\textbf{k} \in \Delta_{\mathbf{l}}^{(m)}} \prod_{j=1}^m f_{\sigma^{k_j} \omega} \circ T_{\omega}^{k_j} \d\mu_{\omega} \right| \ll (\# BP)^a N^{\frac{m}{2}-a} = o(N^{\frac{m}{2}}),
\end{equation*}
and using \eqref{splitbadind}, ($\text{Bad}_m$) follows.

(v) At first sight ($\text{Sub}_m$) might seem a bit problematic to show due to the absolute values inside the integral. Notice, however, that
\begin{itemize}
\item If $m$ is even, then the absolute value does not matter, and ($\text{Sub}_m$) becomes simply
\begin{equation*}
\int_M (S^{\omega}_{A})^m \d\mu_{\omega}.
\end{equation*}
In this case, we can use the same mixing techniques outlined in steps (iv) and (v).
\item If $m$ is odd, and ($\text{Sub}_{m+1}$) holds, then we can use H\"{o}lder's inequality to deduce
\begin{equation*}
\int_M |S^{\omega}_{A}|^m \d\mu_{\omega} \ll \left(\int_M |S^{\omega}_{A}|^{m+1} \d\mu_{\omega}\right)^{\frac{m}{m+1}} \ll N^{\frac{m}{2}},
\end{equation*}
and ($\text{Sub}_m$) follows.
\end{itemize}

We will show ($\text{Sub}_m$) by induction on $m$. 

For $m=2$ and $A\subset \{0,\dots, N_1\}$, using mixing, we have
\begin{align*}
\int_M (S^{\omega}_{A})^2 \d\mu_{\omega} & \leq \sum_{(k,k')\in \{0,\dots, N-1\}^2} \left|\int_M f_{\sigma^k \omega} \circ T_{\omega}^k \cdot f_{\sigma^{k'} \omega} \circ T_{\omega}^{k'} \d\mu_{\omega}\right|\\
&\ll \sum_{n=0}^{N-1} N^*(\sigma^n \omega) + \sum_{n=0}^{N-1} \sum_{k=N^*(\sigma^n \omega)}^{N-n-1} e^{-\gamma k^{\theta}}\\
&\ll N,
\end{align*}
and ($\text{Sub}_2$) holds. 

Now assume, for some even $m$, that ($\text{Sub}_m$) (and by H\"{o}lder's inequality also ($\text{Sub}_{m'}$) for $m'\leq m$) holds. We will show ($\text{Sub}_{m+2}$). Let $A\subset \{0,\dots, N-1\}$ and define analogously to the above
\begin{align*}
\tilde{\Delta}_{good}^{(m)} &= \{\textbf{k}\in A^m \;|\; \forall (j,j')\in \{1,...,m\}^2 \; \text{ it holds that } \; (k_j,k_{j'}) \not\in BP \} ,\\
\tilde{\Delta}_{bad}^{(m)}& = \{\textbf{k}\in A^m \;|\; \exists (j,j')\in \{1,...,m\}^2 \; \text{ such that } \; (k_j,k_{j'}) \in BP \} ,
\end{align*}
then we can split up
\begin{equation*}
(S_{N, \omega}(\mathbf{f}))^m= \sum_{\textbf{k} \in \tilde{\Delta}^{(m)}_{good}} \prod_{j=1}^{m} f_{\sigma^{k_j} \omega} \circ T_{\omega}^{k_j} + \sum_{\textbf{k} \in \tilde{\Delta}^{(m)}_{bad}} \prod_{j=1}^{m} f_{\sigma^{k_j} \omega} \circ T_{\omega}^{k_j}.
\end{equation*}
Therefore, the claim \eqref{momentcltclaim} will follow once we show
\begin{equation}\label{subgoodest}
\left|\sum_{\textbf{k} \in \tilde{\Delta}^{(m)}_{good}} \int_M \prod_{j=1}^{m} f_{\sigma^{k_j} \omega} \circ T_{\omega}^{k_j} \d\mu_{\omega}\right| = O(N^{\frac{m}{2}}),
\end{equation}
and
\begin{equation}\label{subbadest}
\left|\sum_{\textbf{k} \in \tilde{\Delta}^{(m)}_{bad}} \int_M \prod_{j=1}^{m} f_{\sigma^{k_j} \omega} \circ T_{\omega}^{k_j} \d\mu_{\omega}\right| = o(N^{\frac{m}{2}}).
\end{equation}

Arguing as in steps (iv) and (v), with $\{0,\dots, N-1\}$ replaced by $A$, we obtain \eqref{subgoodest} and that \eqref{subbadest} follows from ($\text{Sub}_1$), $\dots$, ($\text{Sub}_m$). Hence ($\text{Sub}_{m+2}$) holds, and by H\"{o}lder's inequality also ($\text{Sub}_{m+1}$).

We conclude that ($\text{Sub}_{m}$) holds for all $m$, and the proof is finished. 

\end{proof}

\begin{proof}[Proof of Proposition \ref{simgaprop}]
(i) If $\mathbf{f}$ is a random coboundary in $L^2$ then it holds that $||g_{\omega}||_{L^2(\d\mu_{\omega})} \in L^2(\P)$, and therefore $||g_{\sigma^N \omega}||_{L^2(\d\mu_{\sigma^N \omega})}$ grows slower than $\sqrt{N}$. Using \eqref{sigmadef} and \eqref{telescope}, it follows that $\sigma_{\mathbf{f}}=0$.

(ii) For $N\geq 1$, let 
\begin{equation*}
\mathbf{g}_N(x,\omega)=\sum_{k=0}^{N-1} f(S^k(x,\omega)),
\end{equation*} 
where $\S(x,\omega)=(T_{\omega}(x),\sigma \omega)$ is the skew product. As in \cite{DFGTV18a}, the $L^2(\mu)$ norm can be computed as
\begin{align*}
\int_{M\times \Omega} \mathbf{g}_N^2 \d\nu & = \E \left( \sum_{n=0}^{N-1} \int_M f_{\sigma^n \omega}^2 \d\mu_{\sigma^n\omega} + 2 \sum_{n=0}^{N-1} \sum_{k=1}^{N-n-1} \int_M f_{\sigma^n \omega} \cdot f_{\sigma^{n+k}\omega} \circ T_{\sigma^n\omega}^k \d\mu_{\sigma^n\omega} \right)\\
&=N \E\left( \int_M f_{\omega}^2 \d\mu_{\omega} + 2\sum_{k=1}^{\infty} \int_M f_{\omega} \cdot f_{\sigma^k\omega} \circ T_{\omega}^n \d\mu_{\omega} \right) \\
& -N \E\left(\sum_{k=N}^{\infty} \int_M f_{\omega} \cdot f_{\sigma^{k}\omega} \circ T_{\omega}^k \d\mu_{\omega} \right) - \sum_{k=1}^{N-1} k \E \left( \int_M f \cdot f_{\sigma^{k}\omega} \circ T_{\omega}^k \d\mu_{\omega} \right),
\end{align*}
and since $\sigma_{\mathbf{f}}=0$ by assumption, considering \eqref{sigmadef}, we obtain
\begin{equation*}
\int_{M\times \Omega} \mathbf{g}_N^2 \d\nu = -N \E\left(\sum_{k=N}^{\infty} \int_M f_{\omega} \cdot f_{\sigma^{k}\omega} \circ T_{\omega}^k \d\mu_{\omega} \right) - \sum_{k=1}^{N-1} k \E \left( \int_M f_{\omega} \cdot f_{\sigma^{k}\omega} \circ T_{\omega}^k \d\mu_{\omega} \right).
\end{equation*}
Using mixing, we estimate
\begin{align*}
\int_{M\times \Omega} \mathbf{g}_N^2 \d\mu &\ll 1 + N \E(\max(N^*(\omega)-N,0) + \E \left( \sum_{k=1}^{\min(N^*(\omega), N-1} k \right)\\
&\ll 1 + N \sum_{k=N}^{\infty} \P(N^* > N) + \E((N^*)^2),
\end{align*}
which is bounded since $\P(N^* > N) \ll N^{-2}$ by Markov's inequality. Now $(\mathbf{g}_N)_N$ is bounded in $L^2(\nu)$, and therefore, using the Banach--Alaoglu Theorem, converges weakly along a subsequence, say
\begin{equation*}
\mathbf{g}_{N_k} \rightarrow \mathbf{g} \quad \text{weakly in } L^2(\nu).
\end{equation*}
Then, using mixing once more, for any function $\mathbf{h}$ such that each $\mathbf{h}(\cdot, \omega) \in \mathfrak{B}$ and $||\mathbf{h}(\cdot, \omega)||_{\mathfrak{B}} \in L^1(\P)$, it holds that
\begin{align*}
\int_{M\times \Omega} \mathbf{h} \cdot (\mathbf{f} - \mathbf{g} + \mathbf{g} \circ S) \d\nu &= \lim_{k\rightarrow\infty} \int_{M\times \Omega} \mathbf{h} \cdot (\mathbf{f} - \mathbf{g}_{N_k} + \mathbf{g}_{N_k} \circ \S) \d\nu\\
&=\lim_{k\rightarrow\infty} \int_{M\times \Omega} \mathbf{h} \cdot \mathbf{f} \circ S^{N_k} \d\nu\\
&=\lim_{k\rightarrow\infty} \E\left(\int_M \mathbf{h}(\cdot,\omega) \cdot f_{\sigma^{N_k} \omega} \circ T_{\omega}^{N_k} \d\mu_{\omega} \d\mu_{\omega} \right)\\
&\ll \lim_{k\rightarrow\infty} e^{-\gamma N_K^{\theta}} + \E\left(N^* \cdot \one_{N^*>N_k} \right)\\
&=0
\end{align*}
thus it follows that $\mathbf{f}=\mathbf{g}-\mathbf{g}\circ \S$.
\end{proof}

\section{Proof of the PLT}\label{pltproofsec}

This section is devoted to the proof of Theorem~\ref{pltthm} under the more general condition~\eqref{asm:mem} from Remark~\ref{memcondpoisrem}, in place of~\eqref{asm:qmem}. For simplicity of notation, we assume without loss of generality that $H_2(x) = x$.

In contrast to the CLT, the present argument allows the function $C(\omega)$ to be arbitrary. By ergodicity of the driving system $\sigma$, the sequence $C(\sigma^n \omega)$ is typically well controlled along most times, allowing us to discard a sparse set of indices and effectively treat $C(\omega)$ as bounded.

More precisely, for any $C_0 > 0$ with $\mathbb{P}(C < C_0) = 1 - \epsilon$, one has that for $\mathbb{P}$-a.e.\ $\omega$, the set $\{n \geq 0 : C(\sigma^n \omega) < C_0\}$ has asymptotic density at least $1 - \epsilon$. Iterating this observation yields control of $C(\sigma^n \omega)$ outside a sparse subsequence; this will be made precise in Lemmas~\ref{badseqclem1} and~\ref{badseqmixlem}.

The proof proceeds by the method of moments. After establishing a slow recurrence property (cf.\ \cite{dfl22}), we control contributions from exceptional times and complete the argument.

The structure is as follows. We first control $C(\sigma^n \omega)$ for typical $n$ (Lemma~\ref{badseqmixlem}). We then prove a slow recurrence estimate (Lemma~\ref{badrecseqlem}). Finally, we carry out the method of moments argument.

\begin{definition}
A sequence $A\subset \N\cup \{0\}$ is said to have density $0$ if 
\begin{equation*}
\lim_{N\rightarrow\infty} \frac{\# (A\cap [0,N])}{N} = 0.
\end{equation*}
\end{definition}

\begin{lemma}\label{badseqclem1}
Let $C:\Omega \rightarrow (0,\infty)$ and $\Lambda_n$ be constants with
\begin{equation*}
\lim_{n\rightarrow\infty} \Lambda_n = \infty.
\end{equation*}
Then, for $\P$-a.e.\ $\omega$, there is a density $0$ sequence $\B=\B(\omega)\subset \N\cup \{0\}$ such that
\begin{equation}\label{Cbound}
C(\sigma^n \omega) \leq \Lambda_n \quad \forall n\geq 1, n\notin \B.
\end{equation}
\end{lemma}
\begin{proof}
Without loss of generality, we assume that $\Lambda_n$ is monotone increasing. Otherwise we replace $\Lambda_n$ by the smaller - but monotonous - sequence $\min_{j\geq n} \Lambda_j$.

Denote the bad sets as $U_n=\{\omega'\in \Omega\;|\; C(\omega')> \Lambda_n\}$. We will consider a generic $\omega$ where it holds that
\begin{equation*}
\lim_{N\rightarrow \infty} \frac{\# \{k\leq N\; |\; \sigma^k \omega \in U_n\}}{N} = \P(U_n) \quad \forall n\geq 1,
\end{equation*}
by the Ergodic Theorem, this holds for $\P$-a.e.\ $\omega$. For the rest of the proof, we will focus on such $\omega$, and omit any dependence on $\omega$ from our notation.

Let $\tilde{N}_n$ be such that 
\begin{equation}\label{uinnumber}
\# \{k\leq N\; |\; \sigma^k \omega \in U_n\} \leq 2N \P(U_n) \quad \forall N\geq \tilde{N}_n.
\end{equation}
Define $N_n$ inductively as 
\begin{equation}\label{defNi}
N_1=\tilde{N}_1 \quad \text{and} \quad N_{n}= \max\left(n, \tilde{N}_{n}, \fceil{\sum_{j=2}^{n-1} N_{j} \P(U_{j-1})}{\P(U_{n})} \right),
\end{equation}
and let
\begin{equation*}
\B_n = \{N_n< k\leq N_{n+1}\; |\; \sigma^k \omega \in U_n\} \subset\{N_n + 1, N_n + 2,\dots N_{n+1}\}.
\end{equation*}
We check that the claims of the Lemma hold for $\B = \bigcup_{n\geq 1} B_n \cup \{0,\dots, N_1\}$.

(i) First we show that $\B$ has density $0$. Let $\epsilon>0$ and $n_{\epsilon}$ be so big that
\begin{equation*}
\P(U_n) < \epsilon \quad \forall n \geq n_{\epsilon}-1,
\end{equation*}
then, for $N> N_{n_{\epsilon}}$, say $N_n < N \leq N_{n+1}$ for some $n\geq n_{\epsilon}$, using \eqref{uinnumber} and \eqref{defNi}, it holds that
\begin{align*}
\# (\B\cap [0,N] ) &\leq \sum_{j=1}^{n-1} \#\B_j + \# (\B_n \cap [0,N])\\
&\leq 2 \sum_{j=2}^{n} N_{j} \P(U_{j-1}) + 2N \P(U_n)\\
&\leq 4 N (\P(U_n)+\P(U_{n-1})) < 8N\epsilon.
\end{align*}

(ii) Now, to check \eqref{Cbound}, let $n\geq 1$ and let $l$ be such that $N_l < n \leq N_{l+1}$. By construction, it holds that $n\geq l$. If $n\notin \B$, then it holds that $\sigma^n \omega \notin U_l$, and by monotonicity
\begin{equation*}
C(\sigma^n \omega) \leq \Lambda_l \leq \Lambda_n.
\end{equation*}
\end{proof}

The exact formulation of this principle, which we will use in our proof, is the following.

\begin{lemma}\label{badseqmixlem}
Let $C:\Omega \rightarrow (0,\infty)$ and, for $m\geq 1$, let $H_m: \R^{m-1} \rightarrow\R$ be a locally bounded function. Then, for any array $(\Lambda_l^{(m)})_{l\geq 1, n \geq 2}$ of positive constants with
\begin{equation*}
\lim_{l\rightarrow\infty} \Lambda_l^{(m)} = \infty \quad \forall m\geq 2, 
\end{equation*}
for any sequence $r_l\searrow 0$, and $\P$-a.e.\ $\omega$, there is a density $0$ sequence $\B=\B(\omega)\subset \N\cup \{0\}$ and a function $M(l)\geq 1$ such that, for all $t>0$, there is a $L(t)\geq 1$ such that, for $l\geq L(t)$, $m\leq M(l)$ and $0\leq k_0\leq \dots \leq k_{m-2} \leq \fceil{t}{\mu(B_{r_l}(x^*))}$ with $k_j\notin B$ for $j=0,\dots, m-2$, it holds that
\begin{equation}\label{Bcond}
H_m(C(\sigma^{k_0} \omega), C(\sigma^{k_1} \omega),\dots, C(\sigma^{k_{m-2}} \omega)) \leq \Lambda_l^{(m)}.
\end{equation}
Furthermore, it holds that 
\begin{equation}\label{Mgrow}
\lim_{l\rightarrow\infty} M(l) = \infty.
\end{equation}
\end{lemma}
\begin{proof}
Since $H_m$ is locally bounded, there are $D_l^{(m)} \geq 0$ such that
\begin{equation}\label{Hbound}
H_m(x_0, \dots, x_{m-2}) \leq \Lambda_l^{(m)} \quad \forall m\leq 2, x_j \leq D_l^{(m)}, j=0,\dots, m-2,
\end{equation}
furthermore $D_l^{(m)}$ can be chosen so that
\begin{equation}\label{Dgrow}
\lim_{l\rightarrow\infty} D_l^{(m)} = \infty \quad \forall m\geq 2.
\end{equation}
Replacing $D_l^{(m)}$ by $\min_{j\geq l} D_l^{(m)}$ if necessary, we may also assume that it is monotonic, meaning that $D_l^{(m)} \leq D_{l+1}^{(m)}$.

Using Lemma \ref{badseqclem1} with the sequence
\begin{equation*}
\tilde{\Lambda}_n^{(m)} = D_{l}^{(m)} \quad \text{if } \fceil{l-1}{\mu(B_{r_{l-1}}(x^*))} < n \leq \fceil{l}{\mu(B_{r_l}(x^*))},
\end{equation*}
provides, for each $m\geq 2$, a density $0$ sequence $\B_m$ such that 
\begin{equation}\label{cineq'}
C(\sigma^n \omega) \leq \tilde{\Lambda}_n^{(m)} \quad \forall n\geq 1, n\notin \B_m, m\geq 2.
\end{equation}
Let $L(t)=\lceil t \rceil$. If $l\geq L(t)$ and $k\leq \fceil{t}{\mu(B_{r_l}(x^*)}$ there is an $l' \leq l$ such that
\begin{equation*}
\fceil{l'-1}{\mu(B_{r_{l'-1}}(x^*))} < k \leq \fceil{l'}{\mu(B_{r_{l'}}(x^*))}.
\end{equation*}
By monotonicity of $D$, from \eqref{cineq'} it follows that
\begin{equation}\label{cineq}
C(\sigma^k \omega) \leq \tilde{\Lambda}_k^{(m)} = D_{l'}^{(m)} \leq D_l^{(m)} \quad \forall k\leq \fceil{t}{\mu(B_{r_l}(x^*))}, k\notin \B_m, m\geq 2.
\end{equation}

Set $N_1=0$ and, for $j\geq 2$, set
\begin{equation*}
N_j=\min(N\geq 1 \;|\; \forall n\geq N, m\leq j \; \text{it holds that} \; \#(\B_m \cap [0,n]) \leq j^{-2} n\}.
\end{equation*}
Then we define the bad sequence $\B$ as
\begin{equation*}
\B = \bigcup_{j\geq 1} \B_j \cap [N_j,\infty),
\end{equation*}
which can be readily checked to have density $0$.

Define $M(l)$ as
\begin{equation*}
M(l)=\max(M\geq 0\;|\; C(\sigma^k \omega) \leq D_l^{(m)} \; \forall k\leq N_{M}, m\leq M
\end{equation*}
Using \eqref{Dgrow}, it is straightforward that \eqref{Mgrow} holds.

Now let $t>0$, $l\geq L(t)$, $m\leq M(l)$, and $k\leq \fceil{t}{\mu(B_{r_l}(x^*))}$ with $k\notin \B$. We claim that
\begin{equation}\label{badconc}
C(\sigma^k \omega) \leq D_l^{(m)}.
\end{equation}
Once this is shown, \eqref{Hbound} will imply \eqref{Bcond}, and the proof of the Lemma will be complete.

(i) If $k\leq N_{M(l)}$, then by definition of $M$, it follows that
\begin{equation*}
C(\sigma^k \omega) \leq D_l^{(m)} \quad \forall m\leq M(l).
\end{equation*}

(ii) If $k> N_{M(l)}$ then, from the definition of $\B$, it follows that $k\notin \B_m$ for all $m\leq M(l)$. Then \eqref{cineq} implies that also in this case
\begin{equation*}
C(\sigma^k \omega) \leq D_l^{(m)} \quad \forall m\leq M(l).
\end{equation*}

Together (i) and (ii) show \eqref{badconc}, and the proof is complete.

\end{proof}

\begin{remark}
The density zero set $\mathcal{B}$ (in Lemmas~\ref{badseqmixlem} and \ref{badrecseqlem} and step~(iii) of the proof of Theorem~\ref{pltthm}) may depend on $\omega$, $x^*$, and the sequence $(r_l)$. However, the pair $(x^*,\omega)$ can be chosen independently of $(r_l)$. Indeed, we take $x^*$ to be quenched slowly recurrent (Definition \ref{qsrpdef}), such $\P$-a.e.\ density $\rho_{\omega'}$ is continuous at $x^*$, and such that $\rho(x^*)>0$, where $\rho=\frac{\d\mu}{\d\vol}$ is the density of the stationary measure. Then we choose $\omega$ such that
\begin{equation*}
\lim_{N\to \infty} \frac{\# \{k\leq N \;|\; \sigma^k \omega \in U\}}{N} = \P(U)
\end{equation*}
for every set $U$ of the form $\{\omega' : C(\omega')<A\}$, $\{\omega' : R_0(\omega')<A\}$, or one of the sets arising in step~(iii) of the proof of Theorem~\ref{pltthm}, with $A>0$, where $R_0$ is as in the definition of slow recurrence. 

This independence is crucial in order to pass to the limit $r \to 0$ in \eqref{hitconthmclaim} and \eqref{retconthmclaim}.
\end{remark}

The following two Lemmas are classical extensions of mixing, and can be proved in a straightforward manner via approximation. We will provide short proofs for convenience.

\begin{lemma}
Assume $(T_{\omega})$ is quenched stretched exponentially mixing. Then there are constants $\theta',\gamma'>0$ and a function $C':\Omega\rightarrow (0,\infty)$ such that for almost every $\omega$ and every $C^{\kappa}$ smooth function $F:M^2 \rightarrow\R$ and integers $k\geq 0$ it holds
\begin{equation}\label{asm:2em2}
\left|\int_{M} F(x, T_{\omega}^{k} x) \d\mu_{\omega}(x) - \int_{M^2} F(x, y) \d\mu_{\omega}(x) \d\mu_{\sigma^{k} \omega}(x') \right|
\leq C'(\omega) e^{-\gamma' k^{\theta'}} \|F\|_{C^{\kappa}}.
\end{equation}
\end{lemma}
\begin{proof}
Let $\gamma,\theta,C(\omega)$ be the constants from \eqref{asm:mem}, and denote $\gamma'=\frac{\gamma}{100 d \kappa}$. For $n\geq 1$ let $S\subset M$ be a maximal $e^{-\gamma' n^{\theta}}$ separated set, meaning $S$ is a maximal set such that
\begin{equation*}
B_{e^{-\gamma' n^{\theta}}}(z) \cap B_{e^{-\gamma' n^{\theta}}}(z') = \emptyset \quad \forall z,z'\in S, z\neq z'.
\end{equation*}
Clearly $\# S \ll e^{d\gamma' n^{\theta}}$. It also holds that $(B_{2 e^{-\gamma' n^{\theta}}}(z))_{z\in S}$ is a cover of $M$, otherwise $S$ would not be maximal. It is now straightforward to construct a partition of unity $(\psi_z)_{z\in S}$ such that
\begin{itemize}
\item $\psi_z(x)\geq 1$ for all $x\in M$,
\item $\psi_z(x)=1$ for $x\in B_{e^{-\gamma' n^{\theta}}-e^{-2\gamma' n^{\theta}}}(z)$, and $\psi_z(x)=0$ for $x\not\in B_{2e^{-\gamma' n^{\theta}}}(z)$
\item $\psi_z(x)=0$ if $x\in B_{e^{-\gamma' n}-e^{-2\gamma' n^{\theta}}}(z')$ for some $z'\in S$ with $z'\neq z$,
\item $\|\psi_z\|_{C^{\kappa}} \ll e^{2\kappa \gamma' n^{\theta}}$ for all $z\in S$,
\item and $\sum_{z\in S} \psi_z(x)=1$ for all $x\in M$.
\end{itemize}
Let $\tilde{F}:M^2 \rightarrow\R$ be defined as
\begin{equation*}
\tilde{F}(x,x') = \sum_{z,z'\in S} F(z,z') \psi_z(x) \cdot \psi_{z'}(x') \quad \forall x,x'\in M.
\end{equation*}
For $x,x'\in M$, by construction it holds that $\psi_z(x) \cdot \psi_{z'}(x') = 0$ whenever $d(x,z)\geq 2e^{-\gamma' n}$ or $d(x',z')\geq 2e^{-\gamma' n^{\theta}}$. Therefore
\begin{align*}
\left| F(x,x') - \tilde{F}(x,x') \right| & \leq \sum_{\substack{z,z'\in S, d(x,z)< 2e^{-\gamma' n^{\theta}},\\ d(x',z')< 2e^{-\gamma' n^{\theta}}}} |F(x,x')-F(z,z')| \psi_z(x) \cdot \psi_{z'}(x')\\
&\leq \sup_{\substack{z,z'\in S, d(x,z)< 2e^{-\gamma' n^{\theta}},\\ d(x',z')< 2e^{-\gamma' n^{\theta}}}} |F(x,x')-F(z,z')|\\
&\ll e^{-\gamma' n^{\theta}} \|F\|_{C^1}.
\end{align*}
Therefore, it holds that
\begin{align*}
&\left|\int_{M} F(x, T_{\omega}^{k} x) \d\mu_{\omega}(x) - \int_{M^2} F(x, y) \d\mu_{\omega}(x) \d\mu_{\sigma^{k} \omega}(x') \right|\\
&\leq \left|\int_{M} F(x, T_{\omega}^{k} x) \d\mu_{\omega}(x) - \int_{M} \tilde{F}(x, T_{\omega}^{k} x) \d\mu_{\omega}(x) \right|\\
&+ \left|\int_{M} \tilde{F}(x, T_{\omega}^{k} x) \d\mu_{\omega}(x) - \int_{M^2} \tilde{F}(x, y) \d\mu_{\omega}(x) \d\mu_{\sigma^{k} \omega}(x') \right|\\
&+ \left|\int_{M^2} \tilde{F}(x, y) \d\mu_{\omega}(x) \d\mu_{\sigma^{k} \omega}(x') - \int_{M^2} F(x, y) \d\mu_{\omega}(x) \d\mu_{\sigma^{k} \omega}(x') \right|\\
&\ll \sup_{x,x'\in M} |F(x,x')-\tilde{F}(x,x')|\\
& + \sum_{z,z'\in S} |F(z,z')| \left|\int_{M} \psi_z(x) \cdot \psi_{z'}(T_{\omega}^k) \d\mu_{\omega}(x) - \int_{M^2} \psi_z(x) \cdot \psi_{z'}(x') \d\mu_{\omega}(x) \d\mu_{\sigma^{k} \omega}(x') \right|\\
&\ll e^{-\gamma' n^{\theta}} \|F\|_{C^1} + C(\omega) e^{4\kappa \gamma' n^{\theta} + d\gamma' n^{\theta}} e^{-\gamma n^{\theta}} \|F\|_{L^{\infty}},
\end{align*}
where we used $\# S \ll e^{d\gamma' n^{\theta}}$ in the last step.
\end{proof}

\begin{lemma}[Mixing estimate for balls]\label{ballmixlem}
Let $\omega$ and $n\geq 2$ be such that \eqref{asm:mem} holds. Let $x^*$ be such that \eqref{asm:volcontrol}, $r_0,...,r_{n-1}>0$ and $0 \leq k_0 \leq \dots \leq k_{n-1}$ then it holds that
\begin{align*}
&\left|\mu_{\omega} \left(\bigcap_{j=0}^{n-1} T_{\sigma^{k_j} \omega}^{-k_j} B_{r_j}(x^*)\right) - \prod_{j=0}^{n-1} \mu_{\sigma^{k_j} \omega}(B_{r_j}(x^*)) \right|\\
&\ll H_n(C(\sigma^{k_0} \omega), C(\sigma^{k_1} \omega),\dots, C(\sigma^{k_{n-2}} \omega)) e^{-\frac{\gamma}{n\kappa+1} \min_{0\leq j\leq n-2}(k_{j+1}-k_j)^{\theta}} \max_{j=0,...,n-1} r_j^{\frac{(d-1) n \kappa}{n\kappa+1}} .
\end{align*}
\end{lemma}
\begin{proof}
Denote $B_j=B_{r_j}(x^*)$ and $p=\min_{0\leq j\leq n-2}(k_{j+1}-k_j)$. For small $\epsilon>0$ there are $h_j\in C^{\kappa}$ such that $0\leq h_j\leq 1_{B_j}$ and
\begin{equation*}
\vol(1_{B_j}\neq h_j)\ll r_j^{d-1} \epsilon, \; \text{ while } ||h_j||_{C^{\kappa}}\ll \epsilon^{-\kappa}.
\end{equation*}
We estimate 
\begin{align*}
&\left|\mu_{\omega} \left(\bigcap_{j=0}^{n-1} T_{\sigma^{k_j} \omega}^{-k_j} B_j\right)-\prod_{j=0}^{n-1} \mu_{\sigma^{k_j} \omega}(B_j)\right|\leq \left|\mu_{\omega} \left(\bigcap_{j=0}^{n-1} T_{\sigma^{k_j} \omega}^{-k_j} B_j\right)-\int_M \left(\prod_{j=0}^{n-1} h_j \circ T_{\omega}^{k_j}\right) \d\mu_{\omega} \right|\\
&\quad +\left| \int_M \left(\prod_{j=0}^{n-1} h_j \circ T_{\omega}^{k_j}\right) \d\mu_{\omega} - \prod_{j=0}^{n-1} \int_M h_j \d\mu_{\sigma^{k_j} \omega} \right|\\
&\quad + \left|\prod_{j=0}^{n-1} \int_M h_j \d\mu_{\sigma^{k_j} \omega} -\prod_{j=0}^{n-1} \mu_{\sigma^{k_j} \omega}(B_j)\right|\\
&\quad \ll \max_{j=0,...,n-1} r_j^{d-1} \epsilon + H_n(C(\sigma^{k_0} \omega), C(\sigma^{k_1} \omega),\dots, C(\sigma^{k_{n-2}} \omega)) e^{-\gamma p^{\theta}} \epsilon^{-n\kappa},
\end{align*}
where we used \eqref{asm:volcontrol} to control $\|1_{B_j} - h_j\|_{L^1(\mu_{\omega})}\ll r_j^{d-1} \epsilon$ for the first and third summand, and \eqref{asm:mem} for the second summand. Letting
\begin{equation*}
\epsilon^*=\left(\max_{j=0,...,n-1} r_j^{d-1} e^{-\gamma p^{\theta}}\right)^{\frac{1}{n\kappa+1}},
\end{equation*}
yields the desired estimate.
\end{proof}

For $x,\omega$ and $r>0$ consider the \textit{self-recurrence times} defined as
\begin{equation*}
\tau_r^{\omega}(x)=\min(n\geq 1 \;|\; d(x, T_{\omega}^n x)<r),
\end{equation*}
and for $p\geq 1$ let the \textit{non-instantaneous self recurrence times} be defined as
\begin{equation*}
\tau_{r,p}^{\omega}(x)=\min(n\geq p+1 \;|\; d(x, T_{\omega}^n x)<r).
\end{equation*}

For a probability measure $\eta$ on a metric space $Y$, and a point $y\in Y$ define the \emph{lower dimension} as
\begin{equation*}
\underline{d}_{\eta}(y) = \liminf_{r\rightarrow 0} \frac{|\log \eta(B_r(y))|}{|\log r|}.
\end{equation*}
and the \emph{upper dimension} as
\begin{equation*}
\overline{d}_{\eta}(y) = \limsup_{r\rightarrow 0} \frac{|\log \eta(B_r(y))|}{|\log r|}.
\end{equation*}

For $\epsilon>0$ we say that a finite set $S\subset M$ is \emph{$\epsilon$-separated}, if for all $x,y\in S$ with $x\not = y$ it holds that
\begin{equation*}
d(x,y)>\epsilon.
\end{equation*}

\begin{lemma}\label{covermeasurelem}
Let $\eta$ be a probability measure on $M$ and, for $\epsilon>0$ small enough, let $S\subset M$ be a $2\epsilon$ separated set. Then it holds that
\begin{equation*}
\sum_{x\in S} \eta(B_{4\epsilon}(x)) \leq 4^d.
\end{equation*}
\end{lemma}
\begin{proof}
Consider the function $\psi(y)=\#\{x\in S \;|\; y\in B_{4\epsilon}(x)\}$. Because $\dim(M)=d$, it is straightforward to show that
\begin{equation*}
\psi(y) \leq \#\{x\in S \;|\; x\in B_{4\epsilon}(y)\} \leq 4^d,
\end{equation*}
provided $\epsilon$ is small. Therefore
\begin{equation*}
\sum_{x\in S} \eta(B_{4\epsilon}(x)) \leq \int_M \psi(y) \d\eta(y) \leq 4^d.
\end{equation*}
\end{proof}

The following non-recurrence statement in Lemma \ref{nonrecRSlem} was shown in \cite{roussrandrec} (see also \cite{sausrec}), assuming additionally that the system exhibits annealed exponential mixing. In \cite{DDWequiv} it was shown that quenched (stretched) exponential mixing implies (stretched) exponential annealed mixing, if the constant $C(\omega)$ has polynomial tails - meaning that $\P(C>n)$ decays polynomially in $n$. However, for us, this is not the case, as we do not make any assumptions on $C(\omega)$. In fact, \cite[Example 1.6(a)]{DDWequiv} provides an example that is quenched exponential mixing, but the speed of annealed mixing is only polynomial.

\begin{lemma}\label{nonrecRSlem}
Assume $(T_{\omega})$ is quenched stretched exponentially mixing and assume that (APER) holds, then for $\nu$-a.e\ $(x,\omega)$ with $\underline{d}_{\mu_{\omega}}(x)>0$ it holds that
\begin{equation}\label{bigret}
\liminf_{r\rightarrow 0} \frac{\log \tau_r^{\omega}(x)}{|\log r|} >0.
\end{equation}
\end{lemma}
\begin{proof}
Notice first that it is enough to show the non-instantaneous version
\begin{equation}\label{noninstrec}
\liminf_{p\rightarrow \infty} \liminf_{r\rightarrow 0} \frac{\log \tau_{r,p}^{\omega}(x)}{|\log r|} >0.
\end{equation}
Indeed, under assumption (APER), it holds that $\tau_{r,p}^{\omega}(x)=\tau_r^{\omega}(x)$ for $\nu$-a.e\ $(x,\omega)$ and $p\geq 1$, for small enough $r$.

Let $\omega$ be such that \eqref{asm:mem} holds and consider $X_+=\{\underline{d}_{\mu_{\omega}}>0\}$. Let $\epsilon>0$, then there is an $a>0$ such that
\begin{equation*}
\mu_{\omega}(X_+)\geq \mu_{\omega}(X_a) \geq \mu_{\omega}(X_+)-\epsilon,
\end{equation*}
where $X_a=\{\underline{d}_{\mu_{\omega}}\geq a\}$. For $\rho>0$ consider the set
\begin{align*}
G_{\rho} = \{ x\in X_a \;|\; \mu_{\omega}(B_r(x)) \leq r^a \; \forall r \leq \rho\}.
\end{align*}
Then, as $\rho \rightarrow 0$, it holds that $\mu_{\omega}(G_{\rho}) \rightarrow \mu_{\omega}(X_a)$ by definition of the lower dimension, therefore we can choose $\rho_0>0$ so small that
\begin{equation*}
\mu_{\omega}(G_{\rho_0}) \geq \mu_{\omega}(X_a) - \epsilon.
\end{equation*}
Denote $G=G_{\rho_0}$.

For $n\geq 1$ let $\delta_n = n^{-\frac{4}{a}}$ and 
\begin{equation*}
A_n = \{y \in M \; |\; T_{\omega}^n(y)\in B_{\delta_n}(y)\},
\end{equation*}
then for any $x\in M$ it holds that
\begin{equation*}
B_{\delta_n}(x) \cap A_n \subset B_{\delta_n}(x) \cap T_{\sigma^n \omega}^{-n} B_{2\delta_n}(x).
\end{equation*}

Let $\iota$ be a $C^{\kappa}$ smooth map with
\begin{equation*}
\iota(y)=1 \;\text{on}\; B_{\delta_n}(x) \quad \text{and} \quad \iota(y)=0 \;\text{on}\; B_{2 \delta_n}(x),
\end{equation*}
and $||\iota||_{C^{\kappa}} \ll \delta_n^{-\kappa}$. Likewise define $\tilde{\iota}$ for $2\delta_n$ instead of $\delta_n$. Using stretched exponential mixing, we obtain
\begin{align*}
\mu_{\omega}(B_{\delta_n}(x) \cap A_n) & \leq \mu_{\omega} (B_{\delta_n}(x) \cap T_{\sigma^n \omega}^{-n} B_{2\delta_n}(x))\\
&\leq \int_M \iota \cdot \tilde{\iota} \circ T_{\omega}^n \d\mu_{\omega}\\
&\ll C(\omega) \delta_n^{-2\kappa} e^{-\gamma n^{\theta}} + \mu_{\omega}(B_{2\delta_n}(x)) \mu_{\sigma^n \omega}(B_{4\delta_n}(x)).
\end{align*}

Let $S\subset G$ be a maximal $\delta_n$ separated set. Then $\#S \ll \delta_n^{-d}$. Since $(B_{\delta_n}(x))_{x\in S}$ is a cover of $G$, for $n$ big enough, we have
\begin{align*}
\mu_{\omega}(G\cap A_n) & \leq \sum_{x\in S} \mu_{\omega}(B_{\delta_n}(x) \cap A_n) \\
&\ll C(\omega) \delta_n^{-2\kappa-d} e^{-\gamma n^{\theta}} + \sum_{x\in S} \mu_{\omega}(B_{2\delta_n}(x)) \mu_{\sigma^n \omega}(B_{4\delta_n}(x))\\
&\ll C(\omega) \delta_n^{-2\kappa-d} e^{-\gamma n^{\theta}} + \delta_n^a,
\end{align*}
where in the last step we used the definition of $G$ and Lemma \ref{covermeasurelem}. 

This bound is summable over $n$, therefore the Borel--Cantelli Lemma implies that, for $\mu_{\omega}$ almost every $x\in G$, there is an $n(x)$ such that $x\not \in A_n$ for $n\geq n(x)$. Therefore
\begin{equation*}
d(T_{\omega}^n (x), x)\geq n^{-\frac{4}{a}} \quad \forall n\geq n(x),
\end{equation*}
and it follows that
\begin{align*}
\liminf_{p\rightarrow \infty} \liminf_{r\rightarrow 0} \frac{\log \tau_{r,p}^{\omega}(x)}{|\log r|}& = \liminf_{p\rightarrow \infty} \liminf_{n\rightarrow \infty} \frac{\log \tau_{n^{-\frac{4}{a}},p}^{\omega}(x)}{\frac{4}{a} \log n}\\
&\geq \frac{a}{4} \lim_{n\rightarrow\infty} \frac{\log n}{\log n} >0.
\end{align*}
Letting $\epsilon \rightarrow0$ proves \eqref{noninstrec}.
\end{proof}

\begin{remark}
Note that if assumption (VOL) is satisfied, then $\P$-a.e.\ fibre measure $\mu_{\omega}$ satisfies $\underline{d}_{\mu_{\omega}}(x)\geq d$ for $\vol$-a.e.\ $x$. Indeed, this holds because
\begin{equation*}
\liminf_{r\rightarrow 0} \frac{|\log \mu_{\omega}(B_r(x))|}{|\log r|} \geq \liminf_{r\rightarrow 0} \frac{|\log (\esssup_{(x',\omega')} \rho_{\omega'}(x')) + \log \vol(B_r(x))|}{|\log r|} = d,
\end{equation*}
for the first inequality we used that $|\log|$ is decreasing on $(0,1)$.
\end{remark}

\begin{definition}[Quenched slowly recurrent point]\label{qsrpdef}
A point $x^*\in M$ is called quenched slowly recurrent if there is a function $C:\Omega\rightarrow(0,\infty)$, for every $A, B>0$, there exists $R_0=R_0(x^*,\omega,A,B)>0$ such that for $\P$-a.e.\ and for all $0<r<R_0$ and positive integers $1\leq n\leq |\log r|^B$, it holds that
\begin{equation}\label{qsrp}
\mu_{\omega} \left( B_r(x^*) \cap T_{\sigma^n \omega}^{-n}(B_r(x^*)) \right) \leq C(\omega) r^d |\log r|^{-A}.
\end{equation}
\end{definition}

\begin{lemma}
Suppose $(T_{\omega})$ is quenched stretched exponentially mixing, and assumptions (APER) and (VOL) hold, then $\vol$-a.e.\ point $x^*$ is quenched slowly recurrent.
\end{lemma}
\begin{proof}
Let $\omega$ be such that \eqref{asm:mem} holds, and let $A, B>0$. Without loss of generality, we may assume $B>\frac{1}{2}$.
 
Because of \eqref{asm:volcontrol}, there is a constant $K>0$ such that
\begin{equation*}
\mu_{\omega'}(B_r(x)) < Kr^d \quad \forall\omega'\in\Omega, x\in M,
\end{equation*}
for $r$ small enough. If $k\geq \sqrt{|\log r|}$ let $\hat{r}=|\log r|^{-A}$ and $F:M\times M \rightarrow [0,1]$ be a $C^{\kappa}$ smooth function with 
\begin{equation*}
F(x,y)=1 \quad \text{if } d(x,y)<\hat{r}, \quad F(x,y)=0 \quad \text{if } d(x,y)>2 \hat{r} \quad \text{and} \quad \|F\|_{C^{\kappa}} \leq \hat{r}^{-\kappa}.
\end{equation*}
Then, using \eqref{asm:2em2}
\begin{equation}\label{sreq}
\begin{aligned}
&\mu_{\omega}(d(x, T_{\omega}^k x) < r) \leq \mu_{\omega}(d(x, T_{\omega}^k x) < \hat{r}) \leq \int_M F(x, T_{\omega}^k x) \d\mu_{\omega}\\
& \ll \hat{r}^{d} + C(\omega) \hat{r}^{-\kappa} e^{-\gamma k^{\theta}} \ll C(\omega) |\log r|^{-A},
\end{aligned}
\end{equation}
for $r$ small enough.

For $j\geq 1,i\geq \sqrt{j}$ let 
\begin{equation*}
F^{\omega}_{j,i} (x) = \mu_{\omega}(B_{2^{-j}}(x) \cap T_{\sigma^i \omega}^{-i} B_{2^{-j}}(x)),
\end{equation*}
then 
\begin{align*}
\int_M F^{\omega}_{j,i} (x) \d\vol(x)& = \int_M \int_M 1_{[0,2^{-j})} (d(x,y)) \cdot 1_{[0,2^{-j})} (d(x,T_{\omega}^i y)) \d\vol(x) \d\mu_{\omega}(y) \\
& \leq \int_M \int_M 1_{[0,2^{-j})} (d(x,y)) \cdot 1_{[0,2^{-j+1})} (d(y,T_{\omega}^i y)) \d\vol(x) \d\mu_{\omega}(y) \\
& \leq \int_M \vol(B_{2^{-j}}(y)) \cdot 1_{[0,2^{-j+1})} (d(y,T_{\omega}^i y)) \d\mu_{\omega}(y) \\
& \leq K 2^{-jd} \int_M  1_{[0,2^{-j+1})} (d(y,T_{\omega}^i y)) \d\mu_{\omega}(y),
\end{align*}
using \eqref{sreq} with $A+B+4$ instead of $A$ we obtain
\begin{equation*}
\int_M F^{\omega}_{j,i} (x) \d\vol(x) \leq C(\omega) 2^{-jd} j^{-A-B-3},
\end{equation*}
for $j$ big enough. Applying Markov's inequality yields
\begin{equation*}
\vol (F^{\omega}_{j,i} > C(\omega) 2^{-jd} j^{-A}) < j^{-3-B},
\end{equation*}
summing the above yields
\begin{equation*}
\vol (\exists \sqrt{j} \leq i \leq j^B \;|\; F^{\omega}_{j,i} > C(\omega) 2^{-jd} j^{-A}) < j^{-3}.
\end{equation*}
The above bound is summable, hence by the Borel--Cantelli Lemma, for $\vol$-a.e.\ $x$, there is a $j_0=j_0(x,\omega,A,B)$ such that
\begin{equation*}
F^{\omega}_{j,i} \leq C(\omega) 2^{-jd} j^{-A} \quad \forall j>j_0, \sqrt{j} \leq i\leq j^B.
\end{equation*}
Clearly this implies that
\begin{equation}\label{srbig}
\mu_{\omega} \left( B_r(x) \cap T_{\sigma^n \omega}^{-n}(B_r(x)) \right) \leq C(\omega) r^d |\log r|^{-A} \quad \forall r<2^{-j_0}, \sqrt{|\log r|} \leq n\leq |\log r|^B.
\end{equation}

Now assume in addition that $x,\omega$ are such that \eqref{bigret} holds. Then, in particular, there is a $\tilde{R}_0=\tilde{R}_0(x,\omega)$ such that 
\begin{equation*}
\tau_{2 \sqrt{r}}^{\omega}(x)> \sqrt{|\log r|} \quad \forall r<\tilde{R}_0.
\end{equation*} 
But then, if $y\in B_r(x)$ and $n\leq \sqrt{|\log r|}$, it follows that
\begin{align*}
d(x,T^n_{\omega} y) \geq d(x, T^n_{\omega} x) - d( T^n_{\omega} x, T^n_{\omega} y) \geq 2\sqrt{r} - \|(T_{\omega})\|_{C^{\kappa}}^n r \geq \sqrt{r},
\end{align*}
making $\tilde{R}_0$ smaller if necessary. In particular, it holds that
\begin{equation*}
B_r(x) \cap T_{\sigma^n \omega}^{-n}(B_r(x))= \emptyset \quad \forall r<\tilde{R}_0, n\leq \sqrt{|\log r|}.
\end{equation*}
Combining this with \eqref{srbig} completes the proof with $R_0=\min(\tilde{R}_0, 2^{-j_0})$.
\end{proof}

\begin{lemma}\label{badrecseqlem}
Let $x^*$ be a quenched slowly recurrent point then, for $\P$-a.e.\ $\omega$, there is a bad set $\B=\B(x^*,\omega,(r_l)) \subset \N$ such that, for $t,A,B>0$, there is an $L=L(x^*,\omega,(r_l),t,A,B)\geq 1$ such that for $l\geq L$ and $j \leq \fceil{t}{\mu(B_{r_l}(x^*))}, j \not \in \B$ it holds that 
\begin{equation}\label{goodindreclemclaim}
\mu_{\sigma^j \omega} \left( B_{r_l}(x^*) \cap T_{\sigma^{n+j} \omega}^{-n}(B_{r_l}(x^*)) \right) \leq C(\sigma^j\omega) r_l^d |\log r_l|^{-A} \quad \forall n\leq |\log r_l|^B.
\end{equation}
\end{lemma}
\begin{proof}
Denote $\epsilon_l = \mu(B_{r_l}(x^*))$. For ease of notation, we will omit dependence on $x^*$ and $(r_l)$.

(i) It is in fact enough to show the claim for fixed $A, B$, meaning for fixed $A, B$ we construct a bad set $\B(\omega, A,B)\subset \N$ of density $0$ such that, for $t>0$, there is an $\tilde{L}(\omega,t, A, B)\geq 1$ and \eqref{goodindreclemclaim} holds. Indeed, once we have constructed such sets, we can find $A_i, B_i ,N_i \nearrow \infty$ and $\delta_i\searrow 0$ such that
\begin{equation*}
\# ([1,N] \cap \B(\omega,A_i,B_i)) \leq \delta_i N \quad \forall N >N_i, i\geq 1,
\end{equation*}
then the claim of the Lemma is satisfied for the set 
\begin{equation*}
\B(\omega) = \bigcup_{i\geq 1} [N_i,N_{i+1}] \cap \B(\omega,A_i,B_i).
\end{equation*}
Since the sets $\B(\omega,A,B)$ can be assumed to be increasing in $A,B$, for $N\in\N$, say $N_i<N\leq N_{i+1}$, it holds that 
\begin{equation*}
\# ([1,N] \cap \B(\omega)) \leq \# ([1,N] \cap \B(\omega,A_i,B_i)) \leq \delta_i N,
\end{equation*}
therefore $\B(\omega)$ has density $0$.

Now for $t,A,B>0$, let $i_0$ be minimal so that $A_{i_0}\geq A$ and $B_{i_0}\geq B$. Using slow recurrence for the measures $\mu_{\omega},\dots, \mu_{\sigma^{N_{i_0}}\omega}$, we can find an $L'\geq 1$ so that, for $j\leq N_{i_0}$ and $l > L'$ it holds that
\begin{equation*}
\mu_{\sigma^j \omega} \left( B_{r_l}(x^*) \cap T_{\sigma^{n+j} \omega}^{-n}(B_{r_l}(x^*)) \right) \leq C(\sigma^j\omega) r_l^d |\log r_l|^{-A} \quad \forall n\leq |\log r_l|^B.
\end{equation*}
Let $L=L(\omega,t,A,B)=\max(L',\tilde{L}(\omega,t, A_{i_0}, B_{i_0}))$. If $l>L$ and $N_{i_0} < j\leq \lceil t \epsilon_{l}^{-1} \rceil, j \not\in\B$, then by construction, since the $\B(\omega,A,B)$ are increasing, it also holds that $j\not\in \B(\omega,A_{i_0},B_{i_0})$ and therefore
\begin{equation*}
\mu_{\sigma^j \omega} \left( B_{r_l}(x^*) \cap T_{\sigma^{n+j} \omega}^{-n}(B_{r_l}(x^*)) \right) \leq C(\sigma^j \omega) r_l^d |\log r_l|^{-A_{i_0}} \quad \forall n\leq |\log r_l|^{B_{i_0}}.
\end{equation*}
So \eqref{goodindreclemclaim} holds in both cases.

(ii) For the rest of the proof, fix $A,B>0$, and omit dependence on $A,B$ for simplicity.

Let $R_0$ be as in the definition of quenched slow recurrence, then \eqref{goodindreclemclaim} follows once we show
\begin{equation}\label{gidrad}
R_0(\sigma^j \omega) \geq r_l.
\end{equation}
Denote $U_l=\{\omega \;|\; R_0(\omega) < r_l \}$, by the Ergodic Theorem, for $\P$-a.e.\ $\omega$ there is an $J_l\geq 1$ such that 
\begin{equation*}
\#\{j\leq J \;|\; \sigma^j\omega\in U_l\} \leq 2 \P(U_l) J \quad \forall J\geq J_l,
\end{equation*}
and denote $\B_l=\{j\geq 1\;|\; \sigma^j\omega\in U_l\}$. Similarly to the proof of Lemma \ref{badseqmixlem}, we define $N_l$ inductively as
\begin{equation*}
N_1=J_1 \quad \text{and} \quad N_{l+1} = \max \left(\fceil{l}{\mu(B_{r_l}(x^*))}, J_{l+1}, \sum_{j = 1}^l \fceil{N_j}{\P(U_{l+1})} \right),
\end{equation*}
and $\B=[0,N_1) \cup \bigcup_{l \geq 1} \left( \B_l \cap [N_l, N_{l+1}) \right)$. The proof will be complete once we show (1) $\B$ has density $0$ and (2) for every $t>0$ there is an $L(t)\geq 1$ such that \eqref{goodindreclemclaim} holds.

(1) Let $\epsilon>0$ and $l_0\geq 1$ be so big that
\begin{equation*}
\P(U_l)<\epsilon \quad \forall l\geq l_0 - 1,
\end{equation*}
then for $N\geq N_{l_0+1}$, say $N_l \leq N <N_{l+1}$ for some $l\geq l_0+1$,it holds that
\begin{align*}
\# \left(\B \cap [1,N]\right) & = N_1 + \sum_{l'=1}^{l-1} \# \left( \B_{l'} \cap [N_{l'}, N_{l'+1}) \right) + \# \left( \B_l \cap [N_l, N] \right)\\
&\leq N_1 + 2 \sum_{l'=1}^{l-1} \P(U_{l'}) N_{l'+1} + \P(U_l) N\\
&\leq 2 \sum_{j=1}^{l-1} N_{j} + 2 \P(U_{l-1}) N_l + \P(U_l) N\\
&\leq 2 (\P(U_l) + \P(U_{l-1})) N \\
&\leq 4 \epsilon N.
\end{align*}

(2) For $t>0$ and $L(t)=\lceil t \rceil$, it holds that $N_l \geq \fceil{t}{\mu(B_{r_l}(x^*))}$ for all $l\geq L(t)$. If $j\leq \fceil{t}{\mu(B_{r_l}(x^*))}$ and $j\not \in \B$, say if $N_{l'} \leq j < N_{l'+1}$ for some $l' \leq l$, then it holds
\begin{equation*}
R(\sigma^j \omega) > r_{l'}\geq r_l,
\end{equation*}
and \eqref{goodindreclemclaim} follows.
\end{proof}

\begin{proof}[Proof of Theorem \ref{pltthm}]
The proofs of \eqref{hitconthmclaim} and \eqref{retconthmclaim} are almost identical. We will focus the main part of the proof, steps (i) through (v), on proving \eqref{hitconthmclaim}, and highlight the differences for the proof of \eqref{retconthmclaim}, in step (vi). 

(i) Using \eqref{sumpoiscon}, is equivalent to show that
\begin{equation}\label{poisconclaim}
\left(S_{\fceil{t}{\mu(B_{r_l}(x^*))},\omega} (1_{B_{r_l}(x^*)}) \right)_{t>0} \xRightarrow{\mu_{\omega}} \mathcal{P} \quad \as{l}, \text{for $\P$-a.e.\ $\omega$},
\end{equation}
where $\mathcal{P}=(\mathcal{P}_t)$ is a standard Poisson process, for all sequences $r_l\searrow 0$. For the rest of the proof, we fix a sequence $(r_l)$ and show \eqref{poisconclaim}.

We will prove this by showing that all possible moments of the terms on the left-hand side converge to the corresponding moments of the Poisson process. We will make this statement more explicit later in (v).

(ii) We first compute the expectation of a single term on the left-hand side of \eqref{poisconclaim}. For $t>0$ we have to compute
\begin{equation*}
\int_M S_{\fceil{t}{\mu(B_{r_l}(x^*))},\omega} (1_{B_{r_l}(x^*)}) \d\mu_{\omega} = \sum_{j=1}^{\fceil{t}{\mu(B_{r_l}(x^*))}} \mu_{\sigma^j \omega}(B_{r_l}(x^*)).
\end{equation*}
Heuristically, here we would like to apply the Ergodic Theorem with the function $\omega \mapsto \mu_{\sigma^j \omega}(B_{r_l}(x^*))$ in order to conclude that the above converge to $t$. However, this cannot be done directly, as the function in question also varies with $l$.

We will use the following idea: Suppose that every density $\rho_{\omega}$ was Lipschitz continuous with a uniform Lipschitz constant $L>0$, then we would have
\begin{equation*}
|\mu_{\omega'} (B_{r_l}(x^*)) - \rho_{\omega}(x^*) \vol(B_{r_l}(x^*)) | \leq L r_l \vol(B_{r_l}(x^*)),
\end{equation*}
now we can apply the Ergodic Theorem with $\omega\mapsto \rho_{\omega}(x^*)$ and, considering that $\int_{\Omega} \rho_{\omega}(x^*) \d\P(\omega) = \rho(x^*)$ for almost all points $x^*$, the desired asymptotics follow. Of course, uniformity of the Lipschitz constant is too strong an assumption, we will, however, use a similar idea.

(iii) Using the assumption (VOL), for $\vol$-a.e.\ $x^*$, the density $\rho_{\omega'}$ is continuous at $x^*$ for $\P$-a.e.\ $\omega'$. Additionally choose $x^*$ such that $\rho(x^*)>0$, where $\rho=\frac{\d\mu}{\d\vol}$ is the density of the stationary measure. Therefore, for $\epsilon>0$, let $\delta(x^*,\omega,\epsilon)>0$ be so that 
\begin{equation*}
d(x^*,y)<\delta(x^*,\omega,\epsilon) \implies |\rho_{\omega}(x^*) - \rho_{\omega}(y)| <\epsilon.
\end{equation*}

For $N\geq 1$ let $\delta'(x^*,N)>0$ be so small that
\begin{equation*}
\P\left(\omega' \;\left|\; \delta\left(x^*,\omega',\frac{1}{N}\right) \leq \delta'(x^*,N)\right.\right) < \frac{1}{N} 
\end{equation*}
and denote this set by $H(x^*,N) = \{\omega' \;|\; \delta(x^*,\omega',\frac{1}{N}) \leq \delta'(x^*,N)\}$. For every $N$ and $\P$-a.e.\ $\omega$, by the Ergodic Theorem, it follows that there is a $J(x^*,\omega,N)\geq 1$ such that
\begin{equation}\label{jnumber}
\#\{j\leq J \;|\; \sigma^j\omega\in H(x^*,N)\} = \sum_{j=1}^{J} 1_{H(x^*,N)}(\sigma^j\omega) \leq \frac{2J}{N} \quad \forall J > J(x^*,\omega,N).
\end{equation}
Now, for $t, N, x^*,\omega$ let $L=L(x^*,\omega,N,t)\geq 1$ be so that $\fceil{t}{\mu(B_{r_L}(x^*))} > J(x^*,\omega,N)$ and $r_L< \delta'(x^*,N)$. We may assume that $L(x^*,\omega,N,t)\nearrow \infty$ as $N\rightarrow\infty$. Finally, for $l\geq 1$, let 
\begin{equation*}
N(x^*,\omega,t,l) = \max(N \;|\; L(x^*,\omega,N,t) < l).
\end{equation*}

Let us first collect some basic properties of $N_l=N(x^*,\omega,t,l)$, for the moment omitting $x^*,\omega,t$ from our notation.
\begin{itemize}
\item $N_l\nearrow \infty$ as $l\rightarrow \infty$, since each $L(N)$ is a finite number, this is obvious.
\item $r_l < \delta'(N_l)$, by construction it holds that $l>L(N_l)$ so the claim follows from the definition of $L(N)$.
\item $\fceil{t}{\mu(B_{r_l}(x^*))} > J(N_l)$, as above, this follows from $l>L(N_l)$.
\end{itemize}
Denote $\mathcal{B}(l,t,x^*,\omega)=\{j\leq \fceil{t}{\mu(B_{r_l}(x^*))} \;|\; \sigma^j\omega\in H(N_l)\}$. Using \eqref{jnumber} we obtain
\begin{align*}
\sum_{j=1}^{\fceil{t}{\mu(B_{r_l}(x^*))}} \mu_{\sigma^j \omega}(B_{r_l}(x^*)) = \sum_{\substack{ 1 \leq j \leq \fceil{t}{\mu(B_{r_l}(x^*))} \\ j\not \in \mathcal{B}(l,t,x^*,\omega)}} \mu_{\sigma^j \omega}(B_{r_l}(x^*)) + O\left(N_l^{-1} \esssup_{\omega'} \frac{\mu_{\omega'}(B_{r_l}(x^*)}{\mu(B_{r_l}(x^*))} \right).
\end{align*}
By choice of $x^*$, the quantity $\frac{\mu_{\omega'}(B_{r_l}(x^*)}{\mu(B_{r_l}(x^*))}$ is bounded uniformly in $\omega'$ and $l$, Therefore this bound simply becomes
\begin{align*}
\sum_{j=1}^{\fceil{t}{\mu(B_{r_l}(x^*))}} \mu_{\sigma^j \omega}(B_{r_l}(x^*)) = \sum_{\substack{ 1 \leq j \leq \fceil{t}{\mu(B_{r_l}(x^*))} \\ j\not \in \mathcal{B}(l,t,x^*,\omega)}} \mu_{\sigma^j \omega}(B_{r_l}(x^*)) + O\left(N_l^{-1}\right).
\end{align*}
Since $r_l<\delta'(N_l)$, we further have
\begin{align*}
\sum_{\substack{1 \leq j \leq \fceil{t}{\mu(B_{r_l}(x^*))} \\ j\not \in \mathcal{B}(l,t,x^*,\omega)}} \mu_{\sigma^j \omega}(B_{r_l}(x^*)) & = \sum_{\substack{ 1 \leq j \leq \fceil{t}{\mu(B_{r_l}(x^*))} \\ j\not \in \mathcal{B}(l,t,x^*,\omega)}} \rho_{\sigma^j\omega}(x^*) \vol(B_{r_l}(x^*)) + O(N_l^{-1})\\
& = \sum_{1 \leq j \leq \fceil{t}{\mu(B_{r_l}(x^*))}} \rho_{\sigma^j\omega}(x^*) \vol(B_{r_l}(x^*)) + O(N_l^{-1}).
\end{align*}
Finally, we can apply the Ergodic Theorem to the last term and obtain
\begin{equation*}
\vol(B_{r_l}(x^*)) \sum_{1 \leq j \leq \fceil{t}{\mu(B_{r_l}(x^*))}} \rho_{\sigma^j\omega}(x^*) \sim \vol(B_{r_l}(x^*)) \fceil{t}{\mu(B_{r_l}(x^*))} \rho(x^*) \sim t.
\end{equation*}

Altogether we have shown that for almost every $(x^*,\omega)$ it holds that
\begin{equation}\label{poismeancon}
\lim_{l \rightarrow \infty} \sum_{j=1}^{\fceil{t}{\mu(B_{r_l}(x^*))}} \mu_{\sigma^j \omega}(B_{r_l}(x^*)) = t \quad \forall t>0.
\end{equation}

(iv)
Using Lemmas \ref{badseqmixlem} and \ref{badrecseqlem} with $\Lambda_l^{(n)}=|\log r_l|$ and $A_l=B_l=l$, for $\nu$-a.e.\ $(x^*,\omega)$, there is a bad set of indices $\B\subset \N$ that has density $0$ and such that for every $t>0$ there is an $L(t)\geq 1$ such that
\begin{itemize}
\item for all $l\geq L(t), m\geq 2$ and $0\leq k_0 \leq ... \leq k_{m-2} \leq \fceil{t}{\mu(B_{r_l}(x^*))}$ with $k_j\notin \B$ for $j=0,...,m-2$, it holds that
\begin{equation}\label{Hmcond}
H_m(C(\sigma^{k_0} \omega), C(\sigma^{k_1} \omega),...,C(\sigma^{k_{m-2}} \omega)) \leq |\log r_l|,
\end{equation}
\item for all $l\geq L(t)$ and $j \leq \fceil{t}{\mu(B_{r_l}(x^*))}, j \not \in \B$ it holds that 
\begin{equation}\label{poisslowrec}
\mu_{\sigma^j \omega} \left( B_{r_l}(x^*) \cap T_{\sigma^n \omega}^{-n}(B_{r_l}(x^*)) \right) \leq C(\omega) r_l^d |\log r_l|^{-l} \quad \forall n\leq |\log r_l|^l.
\end{equation}
\end{itemize}

Now we claim that for all such choices $(x^*,\omega)$ it holds that
\begin{equation}\label{poisconclaimgood}
\left(S_{\fceil{t}{\mu(B_{r_l}(x^*))},\omega} (1_{B_{r_l}(x^*)}) \right)_{t>0} \xRightarrow{\mu_{\omega}} \mathcal{P} \quad \as{l}.
\end{equation}
Once we show this, the proof is complete. 

Define
\begin{equation*}
S_{l,t} = \sum_{1\leq j \leq \fceil{t}{\mu(B_{r_l}(x^*))}, j\not \in \B} 1_{B_{r_l}(x^*)} \circ T_{\omega}^j,
\end{equation*}
then, for every $t>0$, it follows that
\begin{align*}
\int_M \left|S_{\fceil{t}{\mu(B_{r_l}(x^*))},\omega} (1_{B_{r_l}(x^*)}) - S_{l,t} \right| \d\mu_{\omega} &= \sum_{1\leq j \leq \fceil{t}{\mu(B_{r_l}(x^*))}, j\in \B} \mu_{\sigma^j\omega} (B_{r_l}(x^*))\\
& \ll \# \left(\B \cap \left[1,\fceil{t}{\mu(B_{r_l}(x^*))} \right] \right) \vol(B_{r_l}(x^*)) \\
&= o(1), 
\end{align*}
where we used the fact that, for $\mu$-a.e.\ $x^*$ it holds that $\rho(x^*)>0$ and therefore $\mu(B_{r_l}(x^*)) \asymp \vol(B_{r_l}(x^*))$. Therefore \eqref{poisconclaimgood} is equivalent to
\begin{equation}\label{poisconclaim'}
\left(S_{l,t} \right)_{t>0} \xRightarrow{\mu_{\omega}} \mathcal{P} \quad \as{l}.
\end{equation}

Employing the method of moments, \eqref{poisconclaim'} follows once we show the following; For every $J\in \N$, $0=t_0<t_1<...<t_J$ and $m_1,...,m_J\in\N$ it holds that
\begin{equation}\label{poismomentclaim}
\lim_{l\rightarrow\infty} \int_M \prod_{j=1}^J \binom{S_{l,t_j} - S_{l,t_{j-1}}}{m_j} \d\mu_{\omega} = \prod_{j=1}^J \frac{(t_j-t_{j-1})^{m_j}}{m_j!}.
\end{equation}

(v)
Rewriting the left-hand side, we obtain
\begin{equation*}
\prod_{j=1}^{J} {{S_{t_j,l}-S_{t_{j-1},l}} \choose m_j}=\prod_{j=1}^J \sum_{\substack{\fceil{t_{j-1}}{\mu(B_{r_l}(x^*))}+1\leq k_{1,j} <...< k_{m_j,j} \leq \fceil{t_j}{\mu(B_{r_l}(x^*))} \\ k_{1,j},...,k_{m_j,j} \not \in \mathcal{B}_{l,t_j}}} \prod_{i=1}^{m_j} 1_{B_{r_l}(x^*)} \circ T_{\omega}^{k_{i,j}}.
\end{equation*}
Let
\begin{equation}\label{poisdeltadef}
\Delta_l=\left\{\textbf{k}=(k_{i,j})_{\substack{j=1,...,J \\ i=1,...,m_j}} \;\left| 
\begin{aligned}
&\fceil{t_{j-1}}{\mu(B_{r_l}(x^*))}+1\leq k_{1,j} < ... < k_{m_j,j} \leq \fceil{t_j}{\mu(B_{r_l}(x^*))} \\
& \text{ and } k_{1,j},...,k_{m_j,j} \not \in \B \; \text{ for } j=1,...,J
\end{aligned}
\right. \right\},
\end{equation}
then we can simply write
\begin{equation}\label{slbinomprod}
\prod_{j=1}^{J} {{S_{t_j,l}-S_{t_{j-1},l}} \choose m_j}=\sum_{\textbf{k}\in \Delta_l} \prod_{1\leq j\leq J, 1\leq i\leq m_j} 1_{B_{r_l}(x^*)} \circ T_{\omega}^{k_{i,j}}.
\end{equation}
To simplify notation, from now on denote $\textbf{k}\in \Delta_l$ as $\textbf{k}=(k_1,...,k_m)$ with $k_1<...< k_m$ instead, where $m=m_1+...+m_J$.

Now for any multi-index $\textbf{k}\in \Delta_l$ there is a $\iota(\textbf{k})\in \{1,..., 1 + m^2\}$ such that for all $i\neq i'$ it holds that
\begin{equation}\label{ksep}
|k_{i}-k_{i'}|\leq |\log r_l|^{\frac{\iota(\textbf{k})}{\theta}} \quad \text{or} \quad |k_{i}-k_{i'}| > |\log r_l|^{\frac{\iota(\textbf{k}) + 1}{\theta}}.
\end{equation}
Indeed, if the above is not satisfied for some $\iota\in \{1,..., 1 + m^2\}$ then there must be some $i_{\iota}\neq i_{\iota}'$ with
\begin{equation*}
|\log r_l|^{\frac{\iota}{\theta}}< |k_{i_{\iota}}-k_{i_{\iota}'}| \leq |\log r_l|^{\frac{\iota + 1}{\theta}},
\end{equation*}
however there are only ${m\choose 2}$ possible choices for $i_{\iota}\neq i_{\iota}'$, hence there must be a $\iota(\textbf{k})$ where \eqref{ksep} holds.

Furthermore, we define the separation number as the maximum number of indices that are \textit{well separated}, more precisely
\begin{equation*}
s(\textbf{k}) = \max(s \geq 1 \;|\; \exists i_1<...<i_s \text{ with }  \min_{b=1,\dots,s-1} |k_{i_{b+1}}-k_{i_b}| > |\log r_l|^{\frac{\iota(\textbf{k}) + 1}{\theta}}),
\end{equation*}
and let $\Delta_l^s=\{\textbf{k}\in \Delta_l \;|\; s(\textbf{k})=s\}$. It holds that
\begin{equation*}
\#\Delta_l^s\ll r_l^{-ds} |\log r_l|^{\frac{(m-s)(m^2+1)}{\theta}}.
\end{equation*}
Indeed, to obtain a $\textbf{k} \in \Delta_l^s$, we can choose the well separated indices freely, with $\fceil{t_j}{\mu(B_{r_l}(x^*))} \asymp r_l^{-d}$ choices each, and then we can choose the other indices within distance $|log r_l|^{\frac{m^2 + 1}{\theta}}$ of one of the well separated ones. Of course, using this construction, we will have counted some multiindices too often, as well as counted some multiindices that do not appear, due to the restrictions in \eqref{poisdeltadef}. However, this still provides an upper bound.

For $\textbf{k}\in \Delta_l^m$, using Lemma \ref{ballmixlem} it holds
\begin{align*}
\left|\int_M \prod_{i=1}^m 1_{B_{r_l}(x^*)} \circ T_{\omega}^{k_i} - \prod_{i=1}^m \mu_{\sigma^{k_{i}} \omega}(B_{r_l}(x^*)) \right| &\ll H_m(C(\sigma^{k_1} \omega), C(\sigma^{k_2} \omega),\dots, C(\sigma^{k_{m-1}} \omega)) e^{-\frac{\gamma}{m \kappa + 1} |\log r_l|} \\
&\ll |\log r_l| r_l^{100dm},
\end{align*}
summing over $\textbf{k}\in \Delta_l^m$, and using \eqref{Hmcond}, we have
\begin{align*}
&\sum_{\textbf{k}\in \Delta_l^m} \int_M \prod_{i=1}^m 1_{B_{r_l}(x^*)} \circ T_{\omega}^{k_i} \d\mu_{\omega} = \sum_{\textbf{k}\in \Delta_l^m} \prod_{i=1}^m \mu_{\sigma^{k_{i}} \omega}(B_{r_l}(x^*)) + |\log r_l| r_l^{100dm} \#\Delta_l^m \\
& = \prod_{j=1}^J \sum_{\substack{\fceil{t_{j-1}}{\mu(B_{r_l}(x^*))}+1\leq k_{1,j} <...< k_{m_j,j} \leq \fceil{t_j}{\mu(B_{r_l}(x^*))} \\ k_{1,j},...,k_{m_j,j} \not \in \B}} \prod_{i=1}^{m_j} \mu_{\sigma^{k_{i,j}} \omega}(B_{r_l}(x^*)) + o(1)\\
& = \prod_{j=1}^J \sum_{\fceil{t_{j-1}}{\mu(B_{r_l}(x^*))}+1\leq k_{1,j} <...< k_{m_j,j} \leq \fceil{t_j}{\mu(B_{r_l}(x^*))}} \prod_{i=1}^{m_j} \mu_{\sigma^{k_{i,j}} \omega}(B_{r_l}(x^*)) + o(1)\\
& = \prod_{j=1}^J \frac{1}{m_j!} \left( \sum_{\fceil{t_{j-1}}{\mu(B_{r_l}(x^*))}+1\leq k \leq \fceil{t_j}{\mu(B_{r_l}(x^*))}} \mu_{\sigma^k \omega} (B_{r_l}(x^*)) \right)^{m_j} + o(1).
\end{align*}
Finally, applying \eqref{poismeancon} to the last term yields 
\begin{equation}\label{poismommain}
\lim_{l\rightarrow\infty} \sum_{\textbf{k}\in \Delta_l^m} \int_M \prod_{i=1}^m 1_{B_{r_l}(x^*)} \circ T_{\omega}^{k_i} \d\mu_{\omega} = \prod_{j=1}^J \frac{(t_j-t_{j-1})^{m_j}}{m_j!}.
\end{equation}

Now for $s\leq m-1$ and $1\leq j \leq \fceil{t_J}{\mu(B_{r_l}(x^*))}$ with $j\not\in \B$ let 
\begin{equation*}
\Delta_l^{s,j}=\{\textbf{k}\in \Delta_l \;|\; s(\textbf{k})=s \text{ and } \exists i<i' \text{ such that } k_i=j, k_{i'}\leq j + |\log r_l|^{\frac{\iota(\textbf{k})}{\theta}}\},
\end{equation*}
the $\Delta_l^s \subset \bigcup_{1\leq j\leq \fceil{t_J}{\mu(B_{r_l}(x^*))}, j\not\in \B} \Delta_l^{s,j}$. Furthermore it holds that $\# \Delta_l^{s,j} \ll r_l^{-d(s-1)} |\log r_l|^{\frac{(m-s)(1+m^2)}{\theta}}\ll r_l^{-d(s-1)} |\log r_l|^{\frac{2m^3}{\theta}}$. This can be shown by the same counting argument that we used for $\Delta_l^s$, with the only difference that now one of the well-separated indices is fixed to $j$. For $\textbf{k}\in \Delta_l^{s,j}$ let $i_1<...<i_s$ be the indices such that $\min_{b=1,\dots,s-1} |k_{i_{b+1}}-k_{i_b}| > |\log r_l|^{\frac{\iota(\textbf{k})+1}{\theta}}$ and denote $j'=k_{i'}$. In fact, if we only require that $\min_{b=1,\dots,s-1} |k_{i_{b+1}}-k_{i_b}| > \frac{1}{2}|\log r_l|^{\frac{\theta(\textbf{k})+1}{\theta}}$, we may assume that one of the $k_{i_b}$ is equal to $j$, say $j=k_{i_{b^*}}$. Now let $h$ be the smooth function approximating $1_{B_{r_l}(x^*)}$ as in Lemma \ref{ballmixlem}, using mixing we obtain
\begin{align*}
&\int_M \prod_{i=1}^m 1_{B_{r_l}(x^*)} \circ T_{\omega}^{k_i} \d\mu_{\omega} \leq \int_M \left( \prod_{1\leq b \leq s, b \neq b^*} h \circ T_{\omega}^{k_{i_b}} \right) \cdot (h\cdot h\circ T_{\sigma^j \omega}^{j'-j}) \circ T_{\omega}^j \d\mu_{\omega}\\
& \ll H_s(C(\sigma^{k_{i_1}} \omega), C(\sigma^{k_{i_2}} \omega),\dots, C(\sigma^{k_{i_{s-1}}} \omega)) r_l^{d(s-1)} \mu_{\sigma^j\omega}(B_{r_l}(x^*) \cap T_{\sigma^{j'} \omega}^{-(j'-j)} B_{r_l}(x^*))\\
& \ll |\log r_l| r_l^{ds} |\log r_l|^{-2m^3 - 2},
\end{align*}
where we used \eqref{poisslowrec} for $l\geq 2m^3+2$ in the last step. Summing over $\textbf{k}\in \Delta_l^{s,j}$ and then over $s,j$ yields
\begin{equation}\label{poismomrest}
\sum_{\textbf{k}\in \Delta_l \setminus \Delta_l^m} \int_M \prod_{i=1}^m 1_{B_{r_l}(x^*)} \circ T_{\omega}^{k_i} \d\mu_{\omega} \ll |\log r_l|^{-1}.
\end{equation}
Together \eqref{poismommain} and \eqref{poismomrest} show \eqref{poismomentclaim}, and the proof of \eqref{hitconthmclaim} is complete.

(vi) In order to prove \eqref{retconthmclaim} we have to show
\begin{equation*}
\lim_{l\rightarrow\infty} \int_M \prod_{j=1}^J \binom{S_{l,t_j} - S_{l,t_{j-1}}}{m_j} \d\mu_{\omega}|_{B_{r_l}(x^*)} = \prod_{j=1}^J \frac{(t_j-t_{j-1})^{m_j}}{m_j!}.
\end{equation*}
This can be shown completely analogous to \eqref{poismomentclaim}, only with the conditional measure $\mu_{\omega}|_{B_{r_l}(x^*)}$ instead of $\mu_{\omega}$. In the proof, we only need to replace any index $\textbf{k}$ by $(0,\textbf{k})$, then we can proceed in the same fashion.
\end{proof}

\section{Sufficient conditions for mixing of all orders}\label{condsec}

\subsection{Mixing against $L^{\infty}$}\label{linfsec}

In some natural systems, especially uniformly expanding ones, (quenched) exponential mixing holds, where one of the functions is required to be regular in some sense, while the second function only has to be bounded. We saw this, for example, in Example \ref{renewalexple}, where the mixing properties are essentially those of the doubling map, and one can test $BV$ functions against $L^{\infty}$ functions. We will show that this implies (stretched) exponential mixing of all orders, and provide an explicit form for the random threshold.

\begin{lemma}\label{linfmemlem}
Suppose there is a Banach space $\B\subset L^{\infty}$ and constants $\gamma>0$ and $\theta\in (0,1]$ as well as a function $C(\omega):\Omega \rightarrow (0,\infty)$ such that, for all $f\in \B,g\in L^{\infty}, n\geq 1$, it holds
\begin{equation*}
\left|\int_M f \cdot g\circ T_{\omega}^n \d\mu_{\omega}- \int_M f \d\mu_{\omega} \int_M g \d\mu_{\sigma^n \omega} \right| \leq C(\omega) e^{-\gamma n^{\theta}} \|f\|_{\mathfrak{B}} \|g\|_{L^{\infty}} \quad \text{for $\P$-a.e.\ $\omega$}.
\end{equation*}
Then $(T_{\omega})$ is quenched stretched exponential mixing of all orders for the random thresholds
\begin{equation*}
N^*_m(\omega) = \left( \frac{2}{\gamma} \log(m C(\omega)) \right)^{\frac{1}{\theta}}.
\end{equation*}
\end{lemma}
\begin{proof}
This follows by iterating the assumed correlation bound as in \eqref{linfiter}, arriving at \eqref{linfmem}. Rewriting as in Section \ref{mixsec} yields the desired expression for the random thresholds. We omit the details.
\end{proof}

\subsection{Spectral decay}\label{specsec}

A powerful modern tool for analysing mixing properties is the random Perron--Frobenius transfer operator. For $\omega\in \Omega$ we define the random transfer operator $L_{\omega}:L^1(\mu_{\omega}) \rightarrow L^1(\mu_{\sigma \omega})$ by the relation
\begin{equation}\label{ldef}
\int_M f \cdot g \circ T_{\omega} \d\mu_{\omega} = \int_M L_{\omega} (f) \cdot g \d\mu_{\sigma \omega} \quad \forall g\in L^{\infty}.
\end{equation}
The iterates are defined as
\begin{equation*}
L^n_{\omega} = L_{\sigma^{n-1} \omega} \circ L_{\sigma^{n-2} \omega} \circ \cdots \circ L_{\omega}  \quad n\geq 1.
\end{equation*}
If the maps $T_{\omega}$ are invertible, then it holds that $L_{\omega}(f) = f\circ T_{\sigma \omega}^{-1}$. We will often consider $L_{\omega}$ acting on Banach spaces other than $L^1(\mu_\omega)$. In this case, \eqref{ldef} is understood for all test functions $g$ in the corresponding dual space, so that both sides are well-defined.

\begin{remark}
Often, when defining the transfer operator, a fixed reference measure is used in place of the fibre measures $\mu_{\omega}$. A common choice is the volume measure $\vol$, in which case $L_{\omega,\vol}$ is defined by
\begin{equation*}
\int_M f \cdot g \circ T_{\omega} \d\vol = \int_M L_{\omega,\vol} (f) \cdot g \d\vol.
\end{equation*}
This has the advantage that $L$ can also be expressed as
\begin{equation*}
L_{\omega,\vol}(f)(x)=\sum_{T_{\omega} y = x} \frac{f(y)}{|\det(DT_{\omega})(y)|},
\end{equation*}
and it is easier to apply analytic methods to study the operator. The two operators are intimately related. Indeed, if we denote the density of the fibre measure by $\frac{\d\mu_{\omega}}{\d\vol} = \rho_{\omega}$, then they are related by
\begin{equation*}
\rho_{\sigma \omega} \cdot L_{\omega}(f) = L_{\omega,\vol} (\rho_{\omega} \cdot f) \quad \forall f \in L^1(\mu_{\omega}).
\end{equation*}
Thus, the two formulations are equivalent up to a change of density. We adopt \eqref{ldef}, as it simplifies the notation in Lemma \ref{specgapmemlem}.
\end{remark}

We prove quenched (stretched) exponential mixing of all orders via an abstract decay condition that plays the role of a spectral gap in deterministic systems. Since we are working with a family of operators $(L_{\omega})$ rather than a single operator, this cannot be formulated in terms of spectral theory in the usual sense. Instead, following the philosophy of \cite[Theorem 13]{DHS23}, we assume suitable decay of norms. More explicitly we assume that, for $\P$-a.e.\ $\omega$, there are Banach spaces $\B(\omega)$ and a common Banach algebra $\B$ such that the following conditions are satisfied

(H0) For $\P$-a.e.\ $\omega$ it holds $L_{\omega}(\B(\omega)) \subset \B(\sigma \omega)$.

(H1) There are constants $\gamma>0$, $\theta\in (0,1]$ and a function $C:\Omega \rightarrow(0,\infty)$ such that, for $f \in \B(\omega)$ with $\int_M f \d\mu_{\omega}=0$, it holds that
\begin{equation*}
\|L_{\omega}^n(f)\|_{\B(\sigma^n \omega)} \leq C(\omega) e^{-\gamma n^{\theta}} \|f\|_{\B(\omega)} \quad \forall n\geq 1.
\end{equation*}

(H2) It holds that $\B \subset \B(\omega)$, and for $f \in \B(\omega)$ and $g\in \B$ we have
\begin{equation*}
\|g\|_{\B(\omega)} \leq \|g\|_{\B} \quad \text{and} \quad \|f g\|_{\B(\omega)} \leq \|f\|_{\B(\omega)} \|g\|_{\B}.
\end{equation*}

(H3) For $f \in \B(\omega)$ and $g \in \B$ we have
\begin{equation*}
\left|\int_M f \cdot g \, d\mu_\omega\right|
\leq \|f\|_{\B(\omega)} \|g\|_{\B}.
\end{equation*}

\begin{lemma}\label{specgapmemlem}
Suppose conditions (H0) -- (H3) are satisfied, then $(T_{\omega})$ is quenched stretched exponential mixing of all orders in $\B$. The random thresholds are given by
\begin{equation*}
N^*_m(\omega) = \left( \frac{2}{\gamma} \log(m C(\omega)) \right)^{\frac{1}{\theta}}.
\end{equation*}
\end{lemma}
\begin{proof}
Let $m \geq 2$, $f_0,\dots,f_{m-1} \in \B$, and integers $0 = k_0 \leq \dots \leq k_{m-1}$ satisfying
\begin{equation}\label{specgapcond}
k_{j+1} - k_j \geq N^*_m(\sigma^{k_j}\omega).
\end{equation}
We show that
\begin{equation}\label{speclemmemeq}
\begin{aligned}
\left| \int_M \prod_{j=0}^{m-1} f_j \circ T_{\omega}^{k_j} \d\mu_{\omega}
- \prod_{j=0}^{m-1} \int_M f_j \d\mu_{\sigma^{k_j}\omega} \right|
\leq e^{-\frac{\gamma}{2} \min_{0\leq j\leq m-2} (k_{j+1}-k_j)^{\theta}}
\prod_{j=0}^{m-1} \|f_j\|_{\B}.
\end{aligned}
\end{equation}
for $\P$-a.e.\ $\omega$.

It suffices to prove this under the additional assumption that
\begin{equation*}
\int_M f_0 \d\mu_{\omega} = 0,
\end{equation*}
since the general case follows by replacing $f_0$ with $f_0 - \int_M f_0 \d\mu_{\omega}$ and arguing by induction, replacing $\|\cdot\|_{\B}$ by $2\|\cdot\|_{\B}$ if necessary.

Assuming $\int_M f_0 \d\mu_{\omega} = 0$, we rewrite the left-hand side of \eqref{speclemmemeq} using the transfer operators as
\begin{equation}\label{speclemmemeq2}
\left| \int_M f_{m-1} \cdot
L^{k_{m-1}-k_{m-2}}_{\sigma^{k_{m-2}} \omega} \left(
f_{m-2} \cdot L^{k_{m-2}-k_{m-3}}_{\sigma^{k_{m-3}} \omega} \left(
\cdots L^{k_1}_{\omega}(f_0)
\right)\right)
\d\mu_{\sigma^{k_{m-1}} \omega} \right|.
\end{equation}

To simplify notation, define inductively
\begin{equation*}
g_0 = L^{k_1}_{\omega}(f_0), \qquad
g_j = L^{k_{j+1}-k_j}_{\sigma^{k_j}\omega}\bigl(f_j \cdot g_{j-1}\bigr),
\quad j = 1,\dots,m-2.
\end{equation*}
Then \eqref{speclemmemeq2} becomes
\begin{equation*}
\left| \int_M f_{m-1} \cdot g_{m-2} \d\mu_{\sigma^{k_{m-1}} \omega} \right|.
\end{equation*}
By (H3), it suffices to show that, if \eqref{specgapcond} holds, then
\begin{equation}\label{goalestimate}
\| g_{m-2} \|_{\B(\sigma^{k_{m-1}} \omega)}
\leq e^{-\frac{\gamma}{2} \min_{0\leq j\leq m-2} (k_{j+1}-k_j)^{\theta}}
\prod_{j=0}^{m-2} \|f_j\|_{\B}.
\end{equation}

For $j=0,\dots,m-2$, set
\begin{equation*}
A_j := C(\sigma^{k_j}\omega)\, e^{-\gamma (k_{j+1}-k_j)^{\theta}}.
\end{equation*}
We prove by induction that for $l=0,\dots,m-2$,
\begin{equation}\label{specgapnormest}
\|g_l\|_{\B(\sigma^{k_{l+1}} \omega)}
\leq \left( \sum_{i=0}^l \prod_{j=0}^i A_j \right)
\prod_{j=0}^l \|f_j\|_{\B}.
\end{equation}

For $l=0$, since $\int_M f_0 \d\mu_{\omega} = 0$, (H1) gives
\begin{equation*}
\|g_0\|_{\B(\sigma^{k_1} \omega)}
= \|L^{k_1}_{\omega}(f_0)\|_{\B(\sigma^{k_1} \omega)}
\leq C(\omega) e^{-\gamma k_1^{\theta}} \|f_0\|_{\B}
= A_0 \|f_0\|_{\B}.
\end{equation*}

Now assume \eqref{specgapnormest} holds for some $l \leq m-3$, and set
\begin{equation*}
h = f_{l+1} \cdot g_l.
\end{equation*}
By (H2),
\begin{equation*}
\|h\|_{\B(\sigma^{k_{l+1}} \omega)}
\leq \|f_{l+1}\|_{\B} \|g_l\|_{\B(\sigma^{k_{l+1}} \omega)}
\leq \left( \sum_{i=0}^l \prod_{j=0}^i A_j \right)
\prod_{j=0}^{l+1} \|f_j\|_{\B}.
\end{equation*}
By (H3),
\begin{equation*}
\left| \int_M h \d\mu_{\sigma^{k_{l+1}} \omega} \right|
\leq \left( \sum_{i=0}^l \prod_{j=0}^i A_j \right)
\prod_{j=0}^{l+1} \|f_j\|_{\B}.
\end{equation*}

Applying (H1) to the centred function
\begin{equation*}
h - \int_M h \d\mu_{\sigma^{k_{l+1}} \omega},
\end{equation*}
we obtain
\begin{equation*}
\begin{aligned}
\|g_{l+1}\|_{\B(\sigma^{k_{l+2}} \omega)}
&= \|L^{k_{l+2}-k_{l+1}}_{\sigma^{k_{l+1}} \omega} h\|_{\B(\sigma^{k_{l+2}} \omega)} \\
&\leq \left\| L^{k_{l+2}-k_{l+1}}_{\sigma^{k_{l+1}} \omega}
\left(h - \int_M h \d\mu_{\sigma^{k_{l+1}} \omega}\right) \right\|_{\B(\sigma^{k_{l+2}} \omega)} \\
&\quad + \left| \int_M h \d\mu_{\sigma^{k_{l+1}} \omega} \right| \\
&\leq A_{l+1}
\left( \sum_{i=0}^l \prod_{j=0}^i A_j \right)
\prod_{j=0}^{l+1} \|f_j\|_{\B}
+ \left( \sum_{i=0}^l \prod_{j=0}^i A_j \right)
\prod_{j=0}^{l+1} \|f_j\|_{\B} \\
&= \left( \sum_{i=0}^{l+1} \prod_{j=0}^i A_j \right)
\prod_{j=0}^{l+1} \|f_j\|_{\B}.
\end{aligned}
\end{equation*}
This proves \eqref{specgapnormest}.

Finally, under the separation condition \eqref{specgapcond}, each $A_j \leq m^{-1} e^{-\frac{\gamma}{2} (k_{j+1}-k_j)^{\theta}}$, and therefore
\begin{equation}\label{specajest}
\sum_{i=0}^{m-2} \prod_{j=0}^i A_j \leq e^{-\frac{\gamma}{2} \min_{0\leq j\leq m-2} (k_{j+1}-k_j)^{\theta}},
\end{equation}
which yields \eqref{goalestimate} and completes the proof.
\end{proof}

\begin{remark}
In the final estimate \eqref{specajest}, one can in fact show that
\begin{equation*}
\sum_{i=0}^{m-2} \prod_{j=0}^i A_j \leq e^{-\frac{\gamma}{2} k_1^{\theta}}.
\end{equation*}
Consequently, if the first function $f_0$ is centred, i.e.\ $\int_M f_0 \d\mu_{\omega} = 0$, the above argument yields the sharper bound
\begin{equation*}
\begin{aligned}
\left| \int_M \prod_{j=0}^{m-1} f_j \circ T_{\omega}^{k_j} \d\mu_{\omega}
- \prod_{j=0}^{m-1} \int_M f_j \d\mu_{\sigma^{k_j}\omega} \right|
\leq e^{-\frac{\gamma}{2} k_1^{\theta}}
\prod_{j=0}^{m-1} \|f_j\|_{\B}.
\end{aligned}
\end{equation*}

However, when passing to the general case via the standard induction argument, thus removing the centring assumption, this asymmetry is lost, and one recovers the symmetric estimate
\begin{equation*}
\begin{aligned}
\left| \int_M \prod_{j=0}^{m-1} f_j \circ T_{\omega}^{k_j} \d\mu_{\omega}
- \prod_{j=0}^{m-1} \int_M f_j \d\mu_{\sigma^{k_j}\omega} \right|
\leq e^{-\frac{\gamma}{2} \min_{0\leq j\leq m-2} (k_{j+1}-k_j)^{\theta}}
\prod_{j=0}^{m-1} \|f_j\|_{\B}.
\end{aligned}
\end{equation*}
\end{remark}

\section{Proofs for examples}\label{expleproofsec}

\subsection{Proof of Example \ref{renewalexple} and Remark \ref{renewalrem}}\label{renewalproofsec}

Recall that the maps $T_{\omega}:[0,1)\rightarrow [0,1)$ are given by 
\begin{equation*}
T_{\omega}(x) = \begin{cases}
2x \; \text{(mod 1)} & \textit{if } \omega_0=0\\
x & \textit{otherwise}.
\end{cases}
\end{equation*}

Consider the probabilities $p_k$ as
\begin{equation*}
p_k = \frac{c}{(k+1)^{\alpha}} \quad \text{for some } \alpha >2,
\end{equation*}
where $c>0$ is a constant such that $\sum_{k\geq 0} p_k = 1$. The claims of Example \ref{renewalexple} and Remark \ref{renewalrem} will be proven, once we show
\begin{itemize}
\item[(i)] if $\alpha\leq 3$, then $(T_{\omega})$ is not quenched stretched exponentially mixing with an integrable threshold,
\item[(ii)] if $\alpha>3$, then $(T_{\omega})$ is quenched stretched exponentially mixing of all orders with an integrable threshold,
\item[(iii)] if $\alpha>4$, then $(T_{\omega})$ is quenched exponentially mixing of all orders with an integrable threshold,
\item[(iv)] and if $\alpha\leq 4$, then $(T_{\omega})$ is not quenched exponentially mixing with an integrable threshold.
\end{itemize}
Then claims (i) and (ii) will prove Example \ref{renewalexple}, while (iii) and (iv) justify Remark \ref{renewalrem}. 

\begin{remark}\label{renewalstretchrem}
Although the system $(T_{\omega})$ is quenched exponentially mixing of all orders for all $\alpha$, as shown in Remark \ref{renewalrem}, for $\alpha\in (3,4]$ it is necessary to pass to the stretched exponential regime to obtain an integrable threshold. Moreover, the argument used in Lemma \ref{renewalstretchmixlem} suggests that similar mechanisms should apply more generally, in situations where sufficiently strong expansion occurs with adequate frequency. We do not pursue this here, leaving it instead as future work.
\end{remark}

Note that (i) follows directly from \cite[Appendix A]{DHS23}, for if $(T_{\omega})$ were quenched stretched exponentially mixing with an integrable threshold, then Lemma \ref{sigmaexlem} would show that the asymptotic variance converges. For illustrative purposes, however, we shall give a more direct argument.

\begin{lemma}[Proof of (i)]\label{renewalnonmixlem}
Let $\alpha\leq 3$, then there is no $N^*(\omega)\in L^1(\P)$ such that $(T_{\omega})$ is quenched stretched exponentially mixing with random threshold $N^*$.
\end{lemma}
\begin{proof}
If $\omega_0 > n$, then for any $f$ it holds that
\begin{equation*}
\int_0^1 f \cdot f \circ T_{\omega}^n \d\text{Leb} = \int_0^1 f^2 \d\text{Leb}.
\end{equation*}
Therefore, if $(T_{\omega})$ is quenched stretched exponentially mixing with random threshold $N^*$, then
\begin{equation*}
\P(N^* > n) \gg \P(\omega_0 > n) \gg \sum_{k>n} k p_k \gg \frac{1}{n}.
\end{equation*}
Hence $N^*\not\in L^1$.
\end{proof}

\begin{remark}
The above argument, in fact, shows that $(T_{\omega})$ does not satisfy quenched mixing estimates with any mixing rate and an integrable random threshold. We do not formalise this statement here, as we have only introduced the notion of stretched exponential mixing, and such a generalisation is not needed for the purposes of this paper.
\end{remark}

In the following lemmas, we will use the following standard fact on large deviations for heavy-tailed sums; see \cite[Theorem 8.6.3]{EKMextreme} (see also \cite{NagaevLD}).

\begin{lemma}\label{nagaevlem}
Let $(Y_j)_{j\geq 1}$ be i.i.d.\ non-negative random variables with
\begin{equation*}
\P(Y_1 \geq t) \asymp t^{-\beta}
\end{equation*}
for some $\beta>1$. Then for any $K > \E(Y_1)$, it holds that
\begin{equation*}
\P\left( \sum_{j=1}^N Y_j \geq K N \right) \asymp N \P(Y_1 \geq N).
\end{equation*}
\end{lemma}

\begin{lemma}[Proof of (ii)]\label{renewalstretchmixlem}
Let $\alpha > 3$, then there is an $N^*(\omega)\in L^1(\P)$ such that $(T_{\omega})$ is quenched stretched exponentially mixing of all orders with random threshold $N^*$.
\end{lemma}
\begin{proof}
Since $T_0(x)=2x$ (mod 1) is uniformly expanding, it is well known that it is exponentially mixing in the following sense: There exist constants $C>0$ and $\gamma>0$ such that for all $f\in BV$ with $\int_0^1 f \, \d \text{Leb}=0$, all $g\in L^{\infty}$, and all $n\geq 1$, one has
\begin{equation*}
\left| \int_0^1 f \cdot g \circ T_0^n \d\text{Leb} \right| \leq C e^{-\gamma n} \|f\|_{BV} \|g\|_{L^{\infty}}.
\end{equation*}
It is therefore clear that quenched mixing depends on the number of occurrences of $0$ along a given path $\omega$. More precisely, one has
\begin{equation}\label{doubranmix}
\left| \int_0^1 f \cdot g \circ T_{\omega}^n \d\text{Leb} \right| \leq C e^{-\gamma \mathcal{N}_n(\omega)} \|f\|_{BV} \|g\|_{L^{\infty}},
\end{equation}
where $\mathcal{N}_n(\omega)$ is the number of $0$s observed until time $n$, i.e.\
\begin{equation*}
\mathcal{N}_n(\omega) = \#\{j = 0,\dots, n-1 \;|\; \omega_j=0\}.
\end{equation*}
Thus, a polynomial lower bound on $\mathcal{N}_n(\omega)$ yields a stretched exponential estimate in \eqref{doubranmix}. Indeed, let $a=\frac{\alpha+1}{2\alpha-2}<1$ and
\begin{equation*}
N^*(\omega) = \min(N\geq 1 \;|\; \mathcal{N}_n(\omega) \geq \frac{1}{2} n^{1-a} \; \forall n \geq N).
\end{equation*}
Then, from \eqref{doubranmix}, it follows that 
\begin{equation*}
\left| \int_0^1 f \cdot g \circ T_{\omega}^n \d\text{Leb} \right| \leq C e^{-\frac{\gamma}{2} n^{1-a}} \|f\|_{BV} \|g\|_{L^{\infty}},
\end{equation*}
for all $n\geq N^*(\omega)$. From \eqref{linfmem} it follows that $(T_{\omega})$ is quenched stretched exponentially mixing of all orders with random threshold $N^*$. To finish the proof, we will show that $N^*(\omega)\in L^1(\P)$ is integrable.

Let us define $Y_j(\omega)$ as the number of steps from the $j$th $0$ to the $j+1$st, i.e.\
\begin{equation*}
Y_0=\min(k\geq 0 \;|\; \omega_k=0) + 1 = \omega_0 + 1 \quad \text{and} \quad Y_{j+1}=\min(k\geq 1\;|\; \omega_{Y_j + k - 1} = 0) = \omega_{Y_j} + 1.
\end{equation*}
The random variables $Y_j$ satisfy the following properties:
\begin{itemize}
\item the sequence $(Y_j)_{j \geq 1}$ is i.i.d.\ with
\begin{equation*}
\P(Y_j= k+1)=p_k \asymp k^{-\alpha},
\end{equation*}
\item and $Y_0$ is independent of $(Y_j)_{j \geq 1}$, and distributed as
\begin{equation*}
\P(Y_0= k+1)=\lambda(k) \asymp k^{1-\alpha}.
\end{equation*}
\end{itemize}

Let $K=2\E(Y_0)>2\E(Y_1)$. Now suppose that $N$ is such that
\begin{equation*}
Y_0(\omega) + \sum_{j=1}^N Y_j(\omega) < 2KN \quad \text{and} \quad Y_j(\omega) < j^{a} \; \forall j>N,
\end{equation*}
then it holds that
\begin{equation*}
\mathcal{N}_n(\omega) \geq \frac{1}{2} n^{1-a} \quad \forall n > 4KN,
\end{equation*}
and therefore $N^*(\omega)\leq 4KN$. Indeed, let $n > 4KN$, then in any interval of size $n^{a}$ there is at least one $0$, more precisely
\begin{equation*}
0 \in \{\omega_j, \dots, \omega_{j+n^{a} - 1}\} \quad \forall 2KN \leq j \leq n.
\end{equation*} 
It follows that
\begin{equation*}
\mathcal{N}_n(\omega) \geq \frac{n - 2KN}{n^a} \geq \frac{1}{2} n^{1-a}.
\end{equation*}

Now denote $S_N=\sum_{j=1}^N Y_j$ then
\begin{align*}
\P(N^* > 4KN) & \leq \P(Y_0 + S_N \geq 2KN ) + \sum_{j=N+1}^{\infty} \P(Y_j \geq j^a)\\
&\leq \P(Y_0 \geq KN) + \P( S_N \geq KN ) + \sum_{j=N+1}^{\infty} \P(Y_j \geq j^a)\\
&\ll N^{2-\alpha} + \P( S_N \geq KN ) + N^{-\frac{1-\alpha}{2}}.
\end{align*}
By Lemma \ref{nagaevlem} we have
\begin{equation*}
\P( S_N \geq KN ) \ll N\P(Y_1 \geq N) \ll N^{2-\alpha}.
\end{equation*}
Altogether, it holds that 
\begin{equation*}
\P(N^* > 4KN) \ll N^{-\frac{1-\alpha}{2}},
\end{equation*}
and $N^*(\omega)\in L^1(\P)$.
\end{proof}

\begin{lemma}[Proof of (iii)]
Let $\alpha > 4$, then there is a $N^*(\omega)\in L^1(\P)$ such that $(T_{\omega})$ is stretched exponentially mixing of all orders with random threshold $N^*$.
\end{lemma}
\begin{proof}
Let $\mathcal{N}_n, Y_j$ and $S_N$ be defined as in the proof of Lemma \ref{renewalstretchmixlem}.

Note first that
\begin{equation*}
\E(\mathcal{N}_n(\omega))= n \P(\omega_0 = 0) = \lambda(0) n.
\end{equation*}
From \eqref{doubranmix} it follows that
\begin{equation*}
\left| \int_0^1 f \cdot g \circ T_{\omega}^n \d\text{Leb} \right| \leq C e^{-\frac{\gamma\lambda(0)}{2} n} \|f\|_{BV} \|g\|_{L^{\infty}},
\end{equation*}
for $n\geq N^*(\omega)$, where
\begin{equation}\label{renewaln*expdef}
N^*(\omega)=\min(N\geq 1 \;|\; \mathcal{N}_n(\omega) \geq \frac{\lambda(0)}{2} n \; \forall n \geq N).
\end{equation}
From \eqref{linfmem} it follows that $(T_{\omega})$ is quenched exponentially mixing of all orders with random threshold $N^*$. To finish the proof, we will show that $N^*(\omega)\in L^1(\P)$ is integrable.

It holds that 
\begin{equation*}
\mathcal{N}_n(\omega) \geq t \iff Y_0 + S_{\lceil t \rceil -1} < n.
\end{equation*}
By Lemma \ref{nagaevlem} we have 
\begin{align*}
\P(\mathcal{N}_n < \frac{\lambda(0)}{2} n) &= \P(Y_0 + S_{\lceil \frac{\lambda(0)}{2} n \rceil -1}  Y_j \geq n)\\
& \leq \P(Y_0 \geq \frac{n}{3}) + \P(S_{\lceil \frac{\lambda(0)}{2} n \rceil -1}  Y_j \geq \frac{2n}{3})\\
&\ll n^{2-\alpha} + \left( \lceil \frac{\lambda(0)}{2} n \rceil -1 \right) \P(Y_1 \geq \frac{2n}{3})\\
&\ll n^{2-\alpha}.
\end{align*}
Therefore 
\begin{equation*}
\P(N^*>n) \leq \sum_{l \geq n} \P(\mathcal{N}_l(\omega) < \frac{\lambda(0)}{2} l) \ll n^{3-\alpha},
\end{equation*}
and $N^*(\omega)\in L^1(\P)$ is integrable.
\end{proof}

\begin{lemma}[Proof of (iv)]
Let $\alpha\leq 4$, then there is no $N^*(\omega)\in L^1(\P)$ such that $(T_{\omega})$ is quenched exponentially mixing with random threshold $N^*$.
\end{lemma}
\begin{proof}
If $(T_{\omega})$ is quenched exponentially mixing with random threshold $\tilde{N}^*$, then necessarily $\tilde{N}^*\asymp N^*$, where $N^*$ is defined in \eqref{renewaln*expdef}, since exponential mixing requires a linear lower bound on $\mathcal{N}_n(\omega)$. Thus, the proof will be complete, once we show that $N^*(\omega)\not\in L^1(\P)$ is not integrable if $\alpha\leq 4$.

Using Lemma \ref{nagaevlem}, we have
\begin{align*}
\P(N^*=n+1) & \geq \P(\mathcal{N}_n < \frac{\lambda(0)}{2} n) \\
&\geq \P(Y_0 + S_{\lceil \frac{\lambda(0)}{2} n \rceil -1}  Y_j \geq n)\\
&\geq \P(S_{\lceil \frac{\lambda(0)}{2} n \rceil -1}  Y_j \geq n)\\
&\gg n \P(Y_1 \geq n)\\
&\gg n^{2-\alpha}.
\end{align*}
and it follows that $N^*(\omega)\not\in L^1(\P)$ is not integrable.
\end{proof}

\subsection{Proof of Example \ref{pltexple}}\label{pltexpleproofsec}

Here we verify the conditions of Theorem \ref{pltthm} for Example \ref{pltexple}. Note that conditions (R'), (VOL), and (APER) are trivially satisfied:
\begin{itemize}
\item[(R')] The Anosov flow $\phi$ is smooth, and therefore
\begin{equation*}
\sup_{\omega\in \Omega} \|T_{\omega}\|_{C^1}
= \sup_{t\in [0,1]} \|\phi_t\|_{C^1} < \infty,
\end{equation*}
\item[(VOL)] Since $\phi$ preserves a smooth measure $\mu$, it follows that $\mu_{\omega} = \mu$ for all $\omega \in \Omega$,
\item[(APER)] Any point $x$ satisfying $T_{\omega}^n x = x$ for some $\omega \in \Omega$ and $n \geq 1$ must lie on a closed orbit of the Anosov flow $\phi$. Hence
\begin{equation*}
\nu\bigl((x,\omega) \;|\; \exists n \in \N \text{ such that } T_{\omega}^n x = x\bigr)
\leq \mu\left(\bigcup_{\gamma \text{ closed orbit of } \phi} \gamma \right) = 0.
\end{equation*}
\end{itemize}

It remains to verify that $(T_{\omega})$ is quenched exponentially mixing of all orders in $C^1$. We do this using the spectral gap obtained in \cite[Theorem 2.4]{Lcontact} and applying Lemma \ref{specgapmemlem}.

\begin{lemma}
The system $(T_{\omega})$ is quenched exponentially mixing of all orders in $C^1$.
\end{lemma}

\begin{proof}
Let $\mathcal{L}_t$ be the Perron--Frobenius transfer operator associated to $\phi$, defined by $\mathcal{L}_t f = f \circ \phi_{-t}$. By \cite[Theorem 2.4]{Lcontact}, there exists a Banach space $\mathcal{B}$ and constants $C>0$, $\gamma>0$ such that
\begin{equation*}
\| \mathcal{L}_t f \|_{\mathcal{B}} \leq C e^{-\gamma t} \|f\|_{\mathcal{B}} \quad \forall t>0,
\end{equation*}
whenever $f \in \mathcal{B}$ satisfies $\int_M f \d\mu = 0$.

For $n \geq 1$, define
\begin{equation*}
S_n(\omega) = \sum_{j=0}^{n-1} t_{\omega_j}, \qquad \text{where } t_0 = 1.
\end{equation*}
Then the random transfer operator satisfies $L_{\omega}^n = \mathcal{L}_{S_n(\omega)}$. Hence, for $f \in \mathcal{B}$ with $\int_M f \d\mu = 0$,
\begin{equation}\label{pltexplecond1}
\| L_{\omega}^n f \|_{\mathcal{B}}
\leq C e^{-\gamma S_n(\omega)} \|f\|_{\mathcal{B}}
\leq C(\omega) e^{-\frac{\E(t_{\omega_0}) \gamma}{2} n} \|f\|_{\mathcal{B}},
\end{equation}
where
\begin{equation*}
C(\omega) = C e^{\frac{\E(t_{\omega_0}) \gamma}{2} N^*(\omega)}
\end{equation*}
and
\begin{equation*}
N^*(\omega)
= \min\left\{ N \geq 1 \;:\; S_n(\omega) \geq \tfrac{\E(t_{\omega_0})}{2} n \ \text{for all } n \geq N \right\}.
\end{equation*}
By the strong law of large numbers, $C(\omega)$ is finite for $\P$-a.e.\ $\omega$.

Moreover, one checks that $C^1 \subset \mathcal{B}$ and that for $f \in \mathcal{B}$ and $g \in C^1$,
\begin{equation}\label{pltexplecond2}
\|g\|_{\mathcal{B}} \leq \|g\|_{C^1}
\quad \text{and} \quad 
\|f g\|_{\mathcal{B}} \leq \|f\|_{\mathcal{B}} \|g\|_{C^1},
\end{equation}
and
\begin{equation}\label{pltexplecond3}
\left| \int_M f g \d\mu \right| \leq \|f\|_{\mathcal{B}} \|g\|_{C^1}.
\end{equation}

Thus, by \eqref{pltexplecond1}, \eqref{pltexplecond2}, and \eqref{pltexplecond3}, the assumptions of Lemma \ref{specgapmemlem} are satisfied, and it follows that $(T_{\omega})$ is quenched exponentially mixing of all orders in $C^1$.
\end{proof}

To complete the proof of Example \ref{pltexple}, it remains to show that $(T_{\omega})$ does not satisfy the CLT when $p_k \asymp \frac{1}{k^3}$, but does when $p_k \asymp \frac{1}{k^{3+\epsilon}}$. Since the argument is entirely analogous to Example \ref{renewalexple}, we omit the proof and instead summarise the relevant facts.

Let
\begin{equation*}
p_k = \frac{c}{(k+1)^{\alpha}}, \quad \alpha > 2,
\end{equation*}
where $c>0$ is chosen so that $\sum_{k\geq 0} p_k = 1$. Then:
\begin{itemize}
\item If $\alpha \leq 3$, then $(T_{\omega})$ does not satisfy the quenched CLT. Analogously to \cite[Appendix A]{DHS23}, the asymptotic variance fails to converge.
\item If $\alpha \leq 3$, then $(T_{\omega})$ is not quenched stretched exponentially mixing with an integrable threshold. As in Lemma \ref{renewalnonmixlem}, if $\omega_0 > n$, then the system exhibits essentially no mixing up to time $n$.
\item If $\alpha > 3$, then $(T_{\omega})$ is quenched stretched exponentially mixing of all orders with an integrable threshold. As in Lemma \ref{renewalstretchmixlem}, one must work in the stretched exponential regime to obtain integrability. In this case, Theorem \ref{cltthm} applies and yields the quenched CLT.
\end{itemize}

\providecommand{\bysame}{\leavevmode\hbox to3em{\hrulefill}\thinspace}
\providecommand{\MR}{\relax\ifhmode\unskip\space\fi MR }
\providecommand{\MRhref}[2]{%
  \href{http://www.ams.org/mathscinet-getitem?mr=#1}{#2}
}
\providecommand{\href}[2]{#2}

\end{document}